\documentclass[final,notitlepage,12pt,reqno]{amsart}
\usepackage{graphicx}
\usepackage{CJK}
\usepackage{url}

\makeatletter
\let\I\@undefined
\makeatother
\usepackage{geometry}
\usepackage{indentfirst}
\usepackage[normalem]{ulem}
\usepackage{float}
\usepackage{amsthm}

\usepackage{enumitem}[2011/09/28]
\setenumerate{align=left, leftmargin=0pt,labelsep=.5em, labelindent=0\parindent,listparindent=\parindent,itemindent=*}
\usepackage{array}
\usepackage{empheq}
\usepackage{natbib}
\usepackage[all]{xy}

\usepackage{mathrsfs}

\usepackage{caption}[2012/02/19]

\usepackage{longtable,lscape}

\usepackage{fouriernc}
\usepackage{fourier}
\usepackage{amssymb,bm,amsmath}
\usepackage{amsfonts}

\usepackage[OT2,T1,T2A]{fontenc}

\usepackage[english,french,german,russian]{babel}
\usepackage{appendix}

\DeclareMathOperator{\sn}{sn}
\DeclareMathOperator{\cn}{cn}
\DeclareMathOperator{\dn}{dn}

\DeclareMathOperator{\Int}{Int}

\DeclareMathOperator{\D}{d}
\DeclareMathOperator{\I}{Im}
\DeclareMathOperator{\R}{Re}

\def\XXint#1#2#3{{\setbox0=\hbox{$#1{#2#3}{\int}$}
     \vcenter{\hbox{$#2#3$}}\kern-.5\wd0}}

\bibpunct{[}{]}{,}{n}{}{;}

\def\qed{\hfill$ \blacksquare$}
\def\eor{\hfill$ \square$}

\theoremstyle{plain}
\newtheorem{theorem}{Theorem}[subsection]
\newtheorem{proposition}[theorem]{Proposition}
\newtheorem{lemma}[theorem]{Lemma}

\newtheorem{conjecture}[theorem]{Conjecture}

\theoremstyle{definition}
\newtheorem{definition}[theorem]{Definition}

\theoremstyle{remark}
\newtheorem{remark}{Remark}[theorem]

\numberwithin{equation}{subsection}

\usepackage{color}
\begin{document}

\selectlanguage{english}
\pagenumbering{roman}
\title[Kontsevich--Zagier Integrals for  Automorphic Green's Functions.\ II]{Kontsevich--Zagier Integrals for  Automorphic Green's Functions. II
}
\author{Yajun Zhou}
\address{\hspace{2em}}
\email{yajunz@math.princeton.edu}
\date{\today}



\begin{abstract}We introduce interaction entropies, which can be represented as logarithmic couplings of certain cycles on a class of algebraic curves of arithmetic interest. In particular, via interaction entropies for  Legendre--Ramanujan curves $ Y^n=(1-X)^{n-1}X(1-\alpha X)$ ($ n\in\{6,4,3,2\}$), we reformulate the Kontsevich--Zagier integral representations of  weight-4 automorphic Green's functions $ G_2^{\mathfrak H/\overline{\varGamma}_0(N)}(z_1,z_2)$ ($N=4\sin^2(\pi/n )\in\{1,2,3,4\}$), in a geometric  context. These geometric entropies allow us to establish algebraic relations between  certain weight-4 automorphic self-energies and special values of weight-6 automorphic Green's functions. \textit{Subject Classification (AMS 2010)}: 11F03 (Primary),    14E05  (Secondary)   \end{abstract}
\maketitle
\tableofcontents
\clearpage
\pagenumbering{arabic}
\section{Introduction}
\subsection{Statement of results}
In Part~I of this series \cite{AGF_PartI}, we constructed integral representations (in the spirit of Kontsevich and Zagier \cite[][\S3.4]{KontsevichZagier}) for automorphic Green's functions $ G_{k/2}^{\mathfrak H/\overline{\varGamma}_0(N)}(z_{1},z_{2})$ satisfying the cusp-form-free condition $ \dim\mathscr S_k(\varGamma_0(N))=0$ for even weights $k\geq4$ and positive integer levels $N$. Here, $ \overline{\varGamma}_0(N)=\varGamma_0(N)/\{\hat I,-\hat I\}$ is the projective version for the Hecke congruence group of level $N$, \textit{i.e.}~$ \varGamma_0(N):=\left\{ \left.\left(\begin{smallmatrix}a&b\\ c&d\end{smallmatrix}\right)\right| a,b,c,d\in\mathbb Z;ad-bc=1;c\equiv 0\pmod N\right\}$, and $ \overline{\varGamma}_0(1)= PSL(2,\mathbb Z)$ is the full modular group. As in Part~I, we define the automorphic Green's functions  according to the conventions set in \cite[][p.~207]{GrossZagier1985}, \cite[][pp.~238--239]{GrossZagierI} and \cite[][p.~544]{GrossZagierII}:
\begin{align}G_{k/2}^{\mathfrak H/\overline{\varGamma}_0(N)}(z_{1},z_{2}):={}&-2\sum_{\hat  \gamma\in\overline{\varGamma}_0(N)}Q_{\frac{k}{2}-1}
\left( 1+\frac{\vert z_{1} -\hat  \gamma z_2\vert ^{2}}{2\I z_1\I(\hat\gamma z_2)} \right)\notag\\={}&-\sum_{\substack{a,b,c,d\in\mathbb Z\\ N\mid c,ad-bc=1}}Q_{\frac{k}{2}-1}
\left( 1+\frac{\left\vert z_{1} -\frac{a z_2+b}{cz_{2}+d}\right\vert ^{2}}{2\I z_1\I\frac{a z_2+b}{cz_{2}+d}} \right),\quad z_1\notin\varGamma_0(N)z_2,\end{align}with   $Q_\nu $ being the Legendre function of the second kind $ Q_\nu(t):=\int_0^{\infty}(t+\sqrt{t^2-1}\cosh u)^{-\nu-1}\D u,t>1,\nu>-1$.

In the current instalment (Part~II) and its sequel (Part III), we will focus on the weight-4 automorphic Green's functions, where the solutions to the cusp-form-free condition $ \dim\mathscr S_4(\varGamma_0(N))=0$ are exhausted by $ N\in\{1,2,3,4\}$ \cite[][Appendix B]{AGF_PartI}. An outstanding problem concerning these automorphic Green's functions is a conjecture  about the nature of their values at CM points (quadratic irrationals in the upper half-plane), as recapitulated below.\begin{conjecture}[Gross--Zagier  {\cite[][p.~317]{GrossZagierI}} and Gross--Kohnen--Zagier {\cite[][p.~556]{GrossZagierII}}]\label{conj:wt4_GKZ}The CM values for weight-4 automorphic Green's functions of levels $1$, $2$, $3$ and $4$ are always expressible as  logarithms of algebraic numbers:\begin{empheq}[box=\fbox]{align*}\exp \left[\I z\I z'G^{\mathfrak H/\overline {\varGamma}_0(N)}_{2}(z,z')\right]\in\overline{\mathbb Q}\quad \text{if }[\mathbb Q(z):\mathbb Q]=[\mathbb Q(z'):\mathbb Q]=2 \text{ and }N\in\{1,2,3,4\}.\end{empheq}Here, it is understood that the pair of points in question are not equivalent per modular transformations: $ z\notin\varGamma_0(N)z'$.\end{conjecture}

To prepare for a verification of this conjecture in Part III, we develop some analytic tools in this article (Part II). The main result of Part II is a new set of  integral representations for the automorphic Green's functions $ G^{\mathfrak H/\overline {\varGamma}_0(N)}_{2}(z,z')$ ($ N\in\{1,2,3,4\}$), as logarithmic couplings of certain algebraic curves.

Before stating the main result of this article in Theorem \ref{thm:HC_wt4_Green},
we recall from Part I some terminologies and notations pertaining to automorphic functions on the compact Riemann surfaces\footnote{Hereafter, as in Part~I \cite{AGF_PartI}, we maintain the distinction between lowercase backslash (``$\smallsetminus$'' for set minus operations) and uppercase backslash (``$ \backslash$'' for orbit spaces). } $ X_0(N)(\mathbb C)=\varGamma_0(N)\backslash\mathfrak H^*=\varGamma_0(N)\backslash(\mathfrak H\cup\mathbb Q\cup\{i\infty\})$, for $ N\in\{1,2,3,4\}$. One writes \begin{align} \eta(z):=e^{\pi iz/12}\prod_{n=1}^\infty(1-e^{2\pi inz}),\quad z\in\mathfrak H\label{eq:Dedekind_eta_defn}\end{align} for the Dedekind eta function, through which the modular lambda function \begin{align} \lambda(z):=\frac{16\eta^{8}(z/2)\eta^{16}(2z)}{\eta^{24}(z)},\quad z\in\mathfrak H\label{eq:mod_lambda_defn}\end{align} is defined.  The higher-level invariants\footnote{We write ``a.e.'' for ``almost every'' point in question, so as to accommodate to possible exceptions that form a set of zero measure. } \begin{align}
\alpha_N(z):=
\left\{1+\frac{1}{N^{6/(N-1)}}\left[ \frac{\eta(z)}{\eta(Nz)} \right]^{24/(N-1)}\right\}^{-1}=1-\alpha_N\left( -\frac{1}{N z} \right),\quad \text{a.e. }z\in\mathfrak H\label{eq:alpha_N_defn}
\end{align}are tailored for the modular elliptic curves $ X_0(N)(\mathbb C)=\varGamma_0(N)\backslash\mathfrak H^*$ where the positive integer $ N-1$ divides $24$  \cite[][Eq.~2.1.3]{AGF_PartI}. Especially, one has $ \alpha_4(z)=\lambda(2z)$.
In what follows, we denote by
\begin{align}
\Int\mathfrak D_N={}&\left\{ z\in\mathfrak H\left| -\frac{1}{2}<\R z<\frac{1}{2},\left|z+\frac{1}{N}\right|>\frac{1}{N},\left|z-\frac{1}{N}\right|>\frac{1}{N} \label{eq:Int_D_Hecke234}\right. \right\}\intertext{the interior of the fundamental domain $ \mathfrak D_N$ for $ \varGamma_0(N)$, $ N\in\{2,3,4\}$ \cite[][Fig.~1\textit{c}--\textit{e}]{AGF_PartI}.
We retroactively introduce}
\Int\mathfrak D_1={}&\left\{ z\in\mathfrak H\left| -\frac{1}{2}<\R z<\frac{1}{2},|z+1|>1,|z-1|>1\right.\right\}. \label{eq:Int_D_1}\end{align}We normalize the principal modular invariant $ j(z)$  on $ X_0(1)(\mathbb C)=SL(2,\mathbb Z)\backslash\mathfrak H^*$ as\begin{align}
j(z):={}&\frac{256\{1-\lambda(z)+[\lambda(z)]^{2}\}^{3}}{[\lambda(z)]^{2}[1-\lambda(z)]^2},\quad z\in\mathfrak H,\label{eq:Klein_j_defn}
\end{align} and define \begin{align}
\alpha_1(z):=\frac{1}{2}-\frac{4[1-2\lambda(z)][2-\lambda(z)][1+\lambda(z)]}{\sqrt{j(z)}\lambda(z)[1-\lambda(z)]}=1-\alpha_{1}\left( -\frac{1}{ z} \right),\quad z\in\mathfrak \Int\mathfrak D_1\label{eq:alpha_1_defn}
\end{align}using the principal branch of square root.

\begin{theorem}[Entropy Formulae for Automorphic Green's Functions of Weight 4]\label{thm:HC_wt4_Green} For  $ \nu\in\{-1/6,-1/4,-1/3,-1/2\}$, we define  the entropy coupling\begin{align}
H_\nu(\alpha\Vert\beta):=\mathbb E^U_{\nu,\alpha}\mathbb E^V_{\nu,\beta}\log\frac{1-\beta UV}{1-\beta V}-\mathbb E^U_{\nu,1-\alpha}\mathbb E^V_{-\nu-1,\beta}\log\left[ 1-\frac{1-\alpha}{\alpha}\frac{\beta}{1-\beta}(1-U)(1-V)\right]\label{eq:HC_defn_intro}
\end{align}for \emph{a.e.}~$ (\alpha,\beta)\in\mathbb C^{2}$, via the Legendre expectations\begin{align}
\mathbb E^u_{\nu,t}f(u):=\frac{\displaystyle\int_0^1\frac{f(u)\D u}{u^{-\nu}(1-u)^{\nu+1}(1-t u)^{-\nu}}}{\displaystyle\int_0^1\frac{\D u}{u^{-\nu}(1-u)^{\nu+1}(1-t u)^{-\nu}}},\quad \mathbb E^u_{-\nu-1,t}f(u):=\frac{\displaystyle\int_0^1\frac{f(u)\D u}{u^{\nu+1}(1-u)^{-\nu}(1-t u)^{\nu+1}}}{\displaystyle\int_0^1\frac{\D u}{u^{\nu+1}(1-u)^{-\nu}(1-t u)^{\nu+1}}},\label{eq:Legendre_expt_ratio_defn}
\end{align}where the integrations are carried out along the  open unit interval $ u\in(0,1)$. Then, the weight-4 automorphic Green's functions $ G_2^{\mathfrak H/\overline{\varGamma}_0(N)}(z_{1},z_{2})$, \emph{a.e.}~$ z_{1},z_{2}\in\Int\mathfrak D_N$ admit the following integral representations: \begin{align}G_2^{\mathfrak H/PSL(2,\mathbb Z)}(z_{1},z_{2})=G_2^{\mathfrak H/\overline{\varGamma}_0(1)}(z_{1},z_{2})={}&\mathscr I_{1}(z_{1},z_{2})+\mathscr I_{1}(z_{1},-1/z_{2}),\label{eq:G2_PSL2Z_via_I}\\
G_2^{\mathfrak H/\overline{\varGamma}_0(N)}(z_{1},z_{2})={}&\mathscr I_{N}(z_{1},z_{2}),\quad N\in\{2,3,4\}\label{eq:G2_Hecke234_via_I}
\end{align}where\begin{align}
\mathscr I_{N}(z,z')={}&\R\left\{\I z\I z'
\frac{\partial}{\partial\I z}\frac{\partial}{\partial\I z'}\frac{zz'[H_{\nu}(\alpha_N(z)\Vert\alpha_N(z'))+H_{\nu}(1-\alpha_N(z)\Vert1-\alpha_N(z'))]}{\I z\I z'}\right\}\label{eq:I_N_via_H_nu}
\end{align}for  degrees $ \nu\in\{-1/6,-1/4,-1/3,-1/2\}$ and the corresponding levels  $ N=4\sin^2(\nu\pi)\in\{1,2,3,4\}$.
 \qed\end{theorem}
We point out that the entropy formulae in the theorem above are analytic reformulations of the Kontsevich--Zagier integral representations for weight-4 automorphic Green's functions appearing in Part~I
\cite{AGF_PartI}, which were modeled after the integrations of modular forms in the Eichler--Shimura theory. These entropy formulae also serve as bridges towards the arithmetic analysis of the CM values of automorphic Green's functions in Part III, which will mainly deal with the complex multiplication of theta functions.

In this article, we will present a few simple applications of the entropy formulae to renormalized weight-4 automorphic Green's functions, also known as automorphic self-energies. \begin{theorem}[Weight-4 Automorphic Self-Energies]\label{thm:app_HC}\begin{enumerate}[label=\emph{(\alph*)}, ref=(\alph*), widest=a] \item \label{itm:app_HC_a}The weight-4 level-4 automorphic self-energy is defined through the procedure of Gross--Zagier renormalization ~(see \cite[][Chap.~II,  Eq.~5.7]{GrossZagierI} or \cite[][Eq.~3.2.1]{AGF_PartI}): \begin{align}
G_2^{\mathfrak H/\overline{\varGamma}_0(4)}(z):={}&-2\sum_{\hat  \gamma\in\overline {\varGamma}_0(4),\hat \gamma z\neq z}Q_{1}
\left( 1+\frac{\vert z -\hat  \gamma z\vert ^{2}}{2\I z\I(\hat\gamma z)} \right)-2\left\{ \log\left\vert 4\pi\eta^{4}(z)  \I z \right|-1 \right\},\quad z\in\mathfrak H,
\end{align} and admits a closed-form evaluation:\begin{align}
G_2^{\mathfrak H/\overline{\varGamma}_0(4)}(z)=-\frac{1}{3}\log\frac{2^4|1-\alpha_4(z)|^{2}}{|\alpha_4(z)|}=-\frac{1}{3}\log\frac{2^4|1-\lambda(2z)|^{2}}{|\lambda(2z)|}
,\quad \forall z\in\mathfrak H.
\end{align}\item\label{itm:app_HC_c} As long as $z$ is not an elliptic point on $ \varGamma_0(N)$ for $ N\in\{1,2,3\}$, we have\begin{align}
G_2^{\mathfrak H/\overline{\varGamma}_0(N)}(z):={}&-2\sum_{\hat  \gamma\in\overline {\varGamma}_0(N),\hat \gamma z\neq z}Q_{1}
\left( 1+\frac{\vert z -\hat  \gamma z\vert ^{2}}{2\I z\I(\hat\gamma z)} \right)-2\left\{ \log\left\vert 4\pi\eta^{4}(z)  \I z \right|-1 \right\},
\end{align} and \begin{align}G_{2}^{\mathfrak H/\overline{\varGamma}_0(1)}(z)+\frac{\log|j(z)-1728|}{3}={}&\frac{2}{3}G_3^{\mathfrak H/\overline{\varGamma}_0(1)}(i,z)+\frac{2}{3}G_3^{\mathfrak H/\overline{\varGamma}_0(1)}\left(\frac{1+i\sqrt{3}}{2},z\right),\\
G_{2}^{\mathfrak H/\overline{\varGamma}_0(2)}(z)+G_{2}^{\mathfrak H/\overline{\varGamma}_0(2)}\left( -\frac{1}{2z} \right)-\frac{1}{3}\log\frac{|\alpha_{2}(z)[1-\alpha_2(z)]|}{2^{12}}={}&\frac{4}{3}G_3^{\mathfrak H/\overline{\varGamma}_{0}(2)}\left(\frac{1+i}{2},z\right),\\G_{2}^{\mathfrak H/\overline{\varGamma}_0(3)}(z)+G_{2}^{\mathfrak H/\overline{\varGamma}_0(3)}\left( -\frac{1}{3z} \right)+2\log 3={}&\frac{4}{3}G_3^{\mathfrak H/\overline{\varGamma}_{0}(3)}\left(\frac{3+i\sqrt{3}}{6},z\right),
\end{align}where the expressions on the right-hand side are weight-6 automorphic Green's functions.

In addition, both $ G_3^{\mathfrak H/\overline{\varGamma}_0(1)}(i,z)$ and $ G_3^{\mathfrak H/\overline{\varGamma}_0(1)}\left(\frac{1+i\sqrt{3}}{2},z\right)$ can be expressed in terms of weight-4 automorphic Green's functions:\begin{align}G_3^{\mathfrak H/\overline{\varGamma}_0(1)}(i,z)={}&\frac{3}{2}\left[ G_{2}^{\mathfrak H/\overline{\varGamma}_0(4)}\left( -\frac{1}{2z} ,\frac{z}{2}\right) +G_{2}^{\mathfrak H/\overline{\varGamma}_0(4)}\left( \frac{z+1}{2} ,\frac{z}{2}\right)+ G_{2}^{\mathfrak H/\overline{\varGamma}_0(4)}\left( \frac{z}{2(z+1)} ,\frac{z}{2}\right)\right]\notag\\{}&+\frac{\log|j(z)-1728|}{2}-3\log2,\\
G_3^{\mathfrak H/\overline{\varGamma}_0(1)}\left(\frac{1+i\sqrt{3}}{2},z\right) ={}&3G_{2}^{\mathfrak H/\overline{\varGamma}_0(4)}\left( -\frac{1}{2(z+1)} ,\frac{z}{2}\right),
\end{align} so long as both sides of the equations above assume finite values. \qed\end{enumerate}\end{theorem}

\subsection{Notational conventions\label{subsec:notations}}

The following rules regarding complex-valued functions and their integrations apply to this article:\begin{itemize}
\item
The complex logarithms of non-zero numbers are prescribed as \begin{align} \R\log \xi=\log|\xi|, \quad \I \log \xi=\arg \xi\in(-\pi,\pi].\label{eq:Re_Im_log}\end{align} Fractional powers are defined using the aforementioned univalent branch of the logarithm:  \begin{align} \xi^\beta\equiv\sqrt[1/\beta]{\xi}:=e^{\beta\log \xi} \quad\text{for} \quad \xi\in\mathbb C\smallsetminus\{0\}.\label{eq:frac_pow}\end{align}\item For $ a,b\in\mathbb C$, the integration path for $ \int_{a}^{b}f(\xi)\D \xi$ is a straight-line segment starting from $a$ and ending in $ b$, unless explicitly designated otherwise. When we say that ``the integration paths for a certain multiple integral $ \int_{a_1}^{b_1}\D\xi_1\cdots\int_{a_n}^{b_n}\D \xi_n\,f(\xi_1,\dots,\xi_n)$ are straight-line segments'', we are referring to the integration paths for the iterated integrals that constitute the multiple integral.
\item At times, we write $ \infty$ as a shorthand for the positive infinity $+\infty $. We always denote the infinite cusp by $ i\infty$. For any real number $a$, the integration $ \int_a^\infty(\cdots)\D x:=\lim_{M\to+\infty}\int_a^M(\cdots)\D x$ is always carried out along the real axis.\item Whereas  slant $ \varGamma$ will be used for congruence subgroups, the upright $ \Gamma$ is reserved for the Euler integral of the second kind $ \Gamma(\xi):=\int_0^\infty t^{\xi-1}e^{-t}\D t,\R \xi>0$ and its analytic continuation. The digamma function is denoted by $ \psi^{(0)}(\xi):=\partial\log\Gamma(\xi)/\partial\xi $. The Euler--Mascheroni constant is defined by $ \gamma_0:=-\psi^{(0)}(1)$.  \item Via  Euler's integral representation for hypergeometric functions\begin{align}&
_2F_1\left( \left.\begin{array}{c}
a,b \\
c \\
\end{array}\right|t \right)\notag\\={}&\frac{\Gamma(c)}{\Gamma(b)\Gamma(c-b)}\int_0^1\frac{u^{b-1}(1-u)^{c-b-1}}{(1-tu)^{a}}\D u,\quad \R c>\R b>0,-\pi<\arg(1-t)<\pi,\label{eq:Euler_int}\end{align}
one   defines the Legendre function of the first kind: \begin{align}&
P_\nu(1-2t):={_2}F_1\left(  \left.\begin{array}{c}
-\nu,\nu+1 \\
1 \\
\end{array}\right| t\right)\equiv P_{-\nu-1}(1-2t)\notag\\={}&-\frac{\sin(\nu\pi)}{\pi}\int_0^{1}\left[ \frac{u(1-tu)}{1-u} \right]^\nu\frac{\D u}{1-u},\quad -1<\nu<0,t\in\mathbb C\smallsetminus[1,+\infty).\label{eq:Euler_int_Pnu}
\end{align}In particular, the complete elliptic integral of the first kind \begin{align} \mathbf K(\sqrt{t}):=\int_0^{\pi/2}\frac{\D \varphi}{\sqrt{1-t\sin^2\smash[b]{\varphi}}}=\frac{1}{2}\int_0^{1}\frac{\D u}{\sqrt{u(1-u)(1-tu)}}\end{align} is equal to $\frac{\pi}{2}P_{-1/2}(1-2t)$. Powers of the Legendre functions are always written in bracketed form, such as $[P_\nu(1-2t)]^{2} $, in order to avoid confusion with the standard notations for associated Legendre functions.
We write \begin{align}
F(\theta,\lambda):=\int_0^{\theta}\frac{\D \varphi}{\sqrt{1-\lambda\sin^2\smash[b]{\varphi}}}\label{eq:inc_F_defn}
\end{align}for the incomplete elliptic integrals of the first kind.
\item We introduce an \textit{ad hoc} notation for an elementary function (adhering to Eq.~\ref{eq:frac_pow} for the definition of fractional powers) \begin{align}
\mathbb Y_{\nu,\alpha}(u):=\frac{(1-u)^{\nu+1}}{u^\nu(1-\alpha u)^\nu}, \label{eq:Y_nu_alpha_defn}
\end{align}where $ \nu \in (-1,0);\alpha \in (\mathbb C\smallsetminus\mathbb R)\cup(0,1);u,1-u,1-\alpha u\in \mathbb C\smallsetminus(-\infty,0]$, so that we can sometimes use an abbreviation\begin{align}
\int_a^bf(u)\frac{u^\nu(1-\alpha u)^\nu}{(1-u)^{\nu+1}}\D u=\int_a^bf(u)\frac{\D u}{\mathbb Y_{\nu,\alpha}(u)},
\end{align} for integrations over a suitably regular function $f$. For the particular case where $a=0$, $b=1$, we introduce another short-hand notation \begin{align}
\int f(u)\mathbb D_{\nu,\alpha} u:=\int_0^1f(u)\frac{\D u}{\mathbb Y_{\nu,\alpha}(u)}=\int_0^1f(u)\frac{u^\nu(1-\alpha u)^\nu}{(1-u)^{\nu+1}}\D u.\label{eq:mathbb_D_nu_alpha_nu_defn}
\end{align} If $ P_\nu(1-2\alpha)\neq0$, then the symbol $ \mathbb E_{\nu,\alpha}^u$ (the Legendre expectation of degree $\nu$ and parameter $\alpha$, being  compatible with Eq.~\ref{eq:Legendre_expt_ratio_defn}) is defined as\begin{align}
\mathbb E_{\nu,\alpha}^uf(u):=-\frac{\sin(\nu \pi)}{\pi P_\nu(1-2\alpha)}\int f(u)\mathbb D_{\nu,\alpha} u.\label{eq:Legendre_expectation_defn}
\end{align}Whenever writing $ \mathbb D_{\nu,\alpha}u$ or $ \mathbb E^u_{\nu,\alpha}$, we tacitly assume that the integration path coincides with the open unit interval $ (0,1)\ni u$.
 \item As in \cite[][\S3.3]{AGF_PartI}, we use the Jacobi $ \Theta$-function:\begin{align}
\Theta(u|\lambda):={}&\prod_{n=1}^\infty\Bigg\{\Bigg[ 1-e^{-2n{\pi\mathbf K(\sqrt{1-\lambda})}/{\mathbf K(\sqrt{\lambda})}} \Bigg]\times\notag\\{}&\times\Bigg[1-2e^{-(2n-1){\pi\mathbf K(\sqrt{1-\lambda})}/{\mathbf K(\sqrt{\lambda})}}\cos\frac{\pi u}{{\mathbf K(\sqrt{\lambda})}}+e^{-2(2n-1){\pi\mathbf K(\sqrt{1-\lambda})}/{\mathbf K(\sqrt{\lambda})}}\Bigg]\Bigg\}\label{eq:Jacobi_Theta_defn}
\end{align} to define the Jacobi elliptic functions $ \sn(u|\lambda)$, $\cn(u|\lambda)$, $\dn(u|\lambda)$: {\allowdisplaybreaks\begin{align}
\sn(u|\lambda):={}&-\frac{i}{\sqrt[4]{\lambda}}\frac{\Theta(u+i\mathbf K(\sqrt{1-\lambda})|\lambda)}{\Theta(u|\lambda)}\exp\left\{{\frac{\pi i}{2\mathbf K(\sqrt{\lambda})}}\left[ u+\frac{i\mathbf K(\sqrt{1-\lambda})}{2} \right]\right\},\label{eq:Jacobi_sn_defn}\\\cn(u|\lambda):={}&\sqrt[^4\!\!]{\frac{1-\lambda}{\lambda}}\frac{\Theta(u+\mathbf K(\sqrt{\lambda})+i\mathbf K(\sqrt{1-\lambda})|\lambda)}{\Theta(u|\lambda)}\exp\left\{{\frac{\pi i}{2\mathbf K(\sqrt{\lambda})}}\left[ u+\frac{i\mathbf K(\sqrt{1-\lambda})}{2} \right]\right\},\label{eq:Jacobi_cn_defn}\\\dn(u|\lambda):={}&\sqrt[4]{1-\lambda}\frac{\Theta(u+\mathbf K(\sqrt{\lambda})|\lambda)}{\Theta(u|\lambda)},\label{eq:Jacobi_dn_defn}
\end{align}}for $\lambda\in(\mathbb C\smallsetminus\mathbb R)\cup(0,1)$ and a.e.~$ u\in\mathbb C$. \end{itemize}

\subsection{Plan of the proof}
This article is organized as follows. In \S\ref{sec:Interaction_Entropies_Transformations}, we recall from \cite[][\S2.2]{AGF_PartI} the Kontsevich--Zagier integral representations of weight-4 automorphic Green's functions on $ \varGamma_0(N),N\in\{1,2,3,4\}$, and reformulate them into various multiple integrals whose integrands are elementary functions.  In \S\ref{sec:Green_HC}, we prove all the statements in Theorems \ref{thm:HC_wt4_Green} and \ref{thm:app_HC},
drawing on materials from  \S\ref{sec:Interaction_Entropies_Transformations}.

In Part I (see \cite[][\S2.2]{AGF_PartI}, as well as Theorem~\ref{thm:Eichler_KZ_wt4_Green} of this article), weight-4 automorphic Green's functions were given as meromorphic analogs of Eichler integrals over modular forms. In Theorem \ref{thm:HC_wt4_Green} of this paper, the same Green's functions are described by integrations over abelian varieties  $ Y^n=(1-X)^{n-1}X(1-\alpha X)$ (the so-called Legendre--Ramanujan curves for $ n\in\{6,4,3,2\}$). The  modular representation will be bridged to its geometric counterpart through a set of integral identities involving the  Legendre functions and the Legendre--Ramanujan curves. These bridging identities are essentially algebraic relations among members in the ring $ (\mathscr P_{\mathrm{KZ}},+,\cdot)$ of Kontsevich--Zagier periods \cite[][\S1.1]{KontsevichZagier}, which consists of absolutely convergent integrals\begin{align}
\idotsint_{D^{n}_{A}\subset\mathbb R^n}f(\bm x)\D^n\bm x
\end{align} of algebraic functions $ f(\bm x)$ over algebraic domains $ D^{n}_A$ in Euclidean spaces $ \mathbb R^n$. Here, $ f(\bm x)=f(x_1,\dots, x_n)$ is an algebraic function in $n$ real variables, with algebraic coefficients; $ \D^n\bm x=\D x_1\wedge\cdots\wedge\D x_n$ is the Euclidean volume element in  $ \mathbb R^n$; and the algebraic domain  $D^{n}_A$ is a subset in  $ \mathbb R^n$ specified by polynomial inequalities with algebraic coefficients.

The key ingredient of our proof is to convert among various integral representations of automorphic Green's functions using an assortment of  analytic techniques.
Our computations will be  based on  the set of transformations permitted in the Kontsevich--Zagier program \cite[][\S1.2]{KontsevichZagier}, namely, we will rely on nothing else than the following three principles:\begin{enumerate}[leftmargin=*,  label=(KZ\arabic*),ref=(KZ\arabic*),
widest=a, align=left]\item \label{itm:KZ-1}\textbf{(Linear Additivity)} For two algebraic functions $ f(\bm x)$ and $ g(\bm x)$ defined over an  algebraic domain $ D=D_1\cup D_2$, where $ D_1$ and $ D_2$ are two disjoint algebraic sub-domains of $ D\subset\mathbb R^n$, one has \begin{align}
\idotsint_{D}[f(\bm x)+g(\bm x)]\D^n\bm x=\idotsint_{D}f(\bm x)\D^n\bm x+\idotsint_{D}g(\bm x)\D^n\bm x
\end{align}and \begin{align}
\idotsint_{D=D_1\cup D_2}f(\bm x)\D^n\bm x=\idotsint_{D_1}f(\bm x)\D^n\bm x+\idotsint_{D_2}f(\bm x)\D^n\bm x
\end{align}whenever all the integrals involved are absolutely convergent.\item  \label{itm:KZ-2}\textbf{(Algebraic Transformations of Variables)}  If $ \bm y=h(\bm x),\bm x\in D$ is an algebraic and invertible mapping, then \begin{align}
\idotsint_{h(D)}f(\bm y)\D ^n\bm  y=\pm\idotsint_{D}f(h(\bm x))\det h'(\bm x)\D^n\bm x,
\end{align}where $\det h'(\bm x)$ is the Jacobian determinant.\item  \label{itm:KZ-3}\textbf{(Newton--Leibniz--Stokes Formula)} For an algebraic  differential form $ \omega(\bm x)$, one has\begin{align}
\idotsint_D\D\omega (\bm x)=\idotsint_{\partial D}\omega (\bm x),
\end{align} where $ \partial D$ is the boundary of the algebraic domain $D$.  \end{enumerate}

In our derivations, we recast the Kontsevich--Zagier integral representations of certain weight-4 automorphic Green's functions \cite[][\S\S2.1 and 2.2]{AGF_PartI}  into certain  multiple integrals of elementary functions, which will be referred to as ``interaction entropies''.
These multiple integrals of interest are named as such, because they resemble the  entropy functional  defined in information theory. For example, for the probability density
(cf.~Eq.~\ref{eq:Pnu_sqr_F_trip_int})\begin{align}
\rho_\lambda(\phi,\theta ,\psi)=\frac{2}{\pi[\mathbf K(\sqrt{\lambda})]^{2}}\frac{1}{1-\lambda\sin^2\phi\sin^2\theta-\lambda\cos^2\phi\sin^2\psi},\quad \lambda\in(0,1);(\phi,\theta,\psi)\in[0,\pi/2]^3,\label{eq:rho_lambda}
\end{align}
 a ``Jacobi self-interaction formula'' (cf.~Eq.~\ref{eq:log_self_sqr_K}) evaluates its  entropy $ \mathscr E[\rho_\lambda]$ in closed form:  \begin{align} \mathscr E[\rho_\lambda]:={}&-\iiint_{(\phi,\theta,\psi)\in[0,\pi/2]^3}\rho_\lambda(\phi,\theta ,\psi)\log\rho_\lambda(\phi,\theta ,\psi)\D\phi\D\theta \D\psi\notag\\={}&-\frac{\pi}{3}\frac{\mathbf K(\sqrt{1-\lambda})}{\mathbf K(\sqrt{\lambda})}+\frac{2}{3}\log\frac{\sqrt{2\pi^3}(1-\lambda)[\mathbf K(\sqrt{\lambda})]^3}{\sqrt{\lambda}}.\label{eq:Ent_rho_lambda}\end{align}Meanwhile, a special case of the  ``relative interaction entropy'' in Eq.~\ref{eq:Pnu_sqr_intn_bir} is just the relative entropy (Kullback--Leibler divergence)\begin{align}
\mathscr D[\rho_\lambda\Vert\rho_\mu]:=\iiint_{(\phi,\theta,\psi)\in[0,\pi/2]^3}\rho_\lambda(\phi,\theta ,\psi)\log\frac{\rho_\lambda(\phi,\theta ,\psi)}{\rho_\mu(\phi,\theta ,\psi)}\D\phi\D\theta \D\psi\label{eq:Kullback-Leibler}
\end{align}in disguise.

To evaluate the entropy in Eq.~\ref{eq:Ent_rho_lambda}, we shall resort to  Principle \ref{itm:KZ-2}, and use a birational substitution of variables  to relate the triple integral defining    $  \mathscr E[\rho_\lambda]$ to a double integral (see  Eqs.~\ref{eq:log_self_sqr} and \ref{eq:log_self_sqr_K}):\begin{align}
\int_0^{\pi/2}\int_0^{\pi/2}\frac{\log(1-\lambda\sin^2\theta\sin^2\varphi)\D\theta\D\varphi}{\sqrt{1-\lambda\sin^2\theta}\sqrt{1-\lambda\sin^2\smash[b]{\varphi}}}=\int_{0}^{\mathbf K(\sqrt{\lambda})}\int_{0}^{\mathbf K(\sqrt{\lambda})}\log[1-\lambda\sn^2(u|\lambda)\sn^2(v|\lambda)]\D u\D v,
\end{align}which is effectively  an integration\begin{align}
\int_{\gamma_1}\frac{\D X_1}{Y_1}\int_{\gamma_2}\frac{\D X_2}{Y_2}\log(1-\lambda X_1X_2)
\end{align} over two cycles $ \gamma_1,\gamma_2$ on the elliptic curve $ E_{\lambda}(\mathbb C):Y^2=X(1-X)(1-\lambda X)$. Similarly, in the framework of    \ref{itm:KZ-1}--\ref{itm:KZ-3}, one may also transform the sum of  entropies $ \mathscr E[\rho_\lambda]+\mathscr D[\rho_\lambda\Vert\rho_\mu]$ into  integrations over certain cycles on two elliptic curves $ E_{\lambda}(\mathbb C):Y^2=X(1-X)(1-\lambda X)$ and $ E_{\mu}(\mathbb C):Y^2=X(1-X)(1-\mu X)$.
The result is the following ``entropy formula'':{\allowdisplaybreaks\begin{align}
{}&[\mathbf K(\sqrt{\vphantom1\smash[b]{1-\mu}})]^2\int_0^\mu\frac{[\mathbf K(\sqrt{t})]^2\D t}{t-\lambda}-\mathbf K(\sqrt{\vphantom1\smash[b]{\mu }})\mathbf K(\sqrt{1-\smash[b]{\mu}})\int_0^\mu\frac{\mathbf K(\sqrt{t})\mathbf K(\sqrt{1-t})\D t}{t-\lambda}\notag\\={}&\{[\mathbf{K}(\sqrt{\lambda})]^2[\mathbf K(\sqrt{\vphantom1\smash[b]{1-\mu}})]^2-\mathbf{K}(\sqrt{\lambda})\mathbf{K}(\sqrt{1-\lambda})\mathbf K(\sqrt{\vphantom1\smash[b]{\mu }})\mathbf K(\sqrt{1-\smash[b]{\mu}})\}\log\left( 1-\frac{\mu}{\lambda} \right)\notag\\{}&-\frac{2[\mathbf K(\sqrt{\vphantom1\smash[b]{1-\mu}})]^2}{\pi} \int_0^{\pi/2}\int_0^{\pi/2}\int_0^{\pi/2}\frac{\log(1-\mu\sin^2\phi\sin^2\theta-\mu\cos^2\phi\sin^2\psi)\D\phi\D\theta\D\psi}{1-\lambda\sin^2\phi\sin^2\theta-\lambda\cos^2\phi\sin^2\psi}\notag\\{}&+\frac{2\mathbf K(\sqrt{\vphantom1\smash[b]{\mu }})\mathbf K(\sqrt{1-\smash[b]{\mu}})}{\pi} \int_0^{\pi/2}\int_0^{\pi/2}\int_0^{\pi/2}\frac{\log\frac{1-\mu\sin^2\phi\sin^2\theta-(1-\mu)\cos^2\phi\sin^2\psi}{1-\cos^2\phi\sin^2\psi}\D\phi\D\theta\D\psi}{1-\lambda\sin^2\phi\sin^2\theta-(1-\lambda)\cos^2\phi\sin^2\psi}
\notag\\={}&\mathbf K(\sqrt{1-\lambda} ) \mathbf K(\sqrt{1-\smash[b]{\mu}})\int_{0}^{\mathbf K(\sqrt{\lambda})}\D u\int_{0}^{ \mathbf K(\sqrt{\vphantom1\smash[b]{\mu }})}\D v\,\log[1-\mu\sn^2(u|\lambda)\sn^2(v|\mu)]\notag\\{}&-\mathbf K(\sqrt{\lambda} ) \mathbf K(\sqrt{1-\smash[b]{\mu}})\int_{\mathbf K(\sqrt{\lambda})}^{\mathbf K(\sqrt{\lambda})+i\mathbf K(\sqrt{1-\lambda})}\frac{\D u}{i}\int_{0}^{ \mathbf K(\sqrt{\vphantom1\smash[b]{\mu }})}\D v\,\log[1-\mu\sn^2(u|\lambda)\sn^2(v|\mu)],\label{eq:Eichler_ent_HC}\end{align}}for $ 0<\mu<\lambda<1$, which translates into a special case ($ N=4$) of Theorem \ref{thm:HC_wt4_Green}. In the ``entropy formula'' above, the cycle integrals defined on $ E_\lambda(\mathbb C)\times E_\mu(\mathbb C)$ are modest extensions of the standard abelian integrals on Jacobian varieties.
Unlike the self-interaction formula in Eq.~\ref{eq:Ent_rho_lambda}, the entropic coupling on  $ E_\lambda(\mathbb C)\times E_\mu(\mathbb C)$  (where $ \lambda\neq\mu$) does not admit simple closed-form evaluations. Nonetheless, we will show in Part III that for CM\ moduli parameters, the aforementioned entropy formula still decomposes into finite terms of theta function expressions, thereby proving Conjecture~\ref{conj:wt4_GKZ}.

We originally conceived the entropy formulae for automorphic Green's functions (Theorem \ref{thm:HC_wt4_Green}) by analogy to the renormalization group expansion in quantum field theory. In lieu of the usual loop Feynman diagrams in the momentum space, one encounters, in the ``renormalization group expansion'' on moduli spaces $X_0(N)(\mathbb C),N\in\{1,2,3,4\} $,  multiple integrals such as (cf.~Eq.~\ref{eq:rel_intn_self_intn}) \begin{align}
\int_0^1\left\{\int_0^U\left[\int_W^1\frac{\D V}{\mathbb Y_{\nu,\alpha}(V)}\right]\frac{\D W}{\mathbb Y_{-\nu-1,\alpha}(W)}\right\}\frac{\D U}{\mathbb Y_{\nu,\beta}(U)},
\end{align}where the addition law of abelian integrals on abelian varieties plays the r\^ole of ``momentum sum rule''. Such physical heuristics was later abandoned in favor of a mathematically rigorous formulation in the framework of Kontsevich--Zagier periods, as will be presented in  \S\ref{sec:Interaction_Entropies_Transformations}.

\section{Interaction entropies and their transformations\label{sec:Interaction_Entropies_Transformations}}
\setcounter{equation}{0}
\setcounter{theorem}{0}
In this section, we prepare some analytic transformations for integrations over Legendre functions $ P_\nu=P_{-\nu-1}$ of fractional degrees $ \nu\in(-1,0)$.
To illustrate the arithmetic r\^oles of Legendre functions, we recall from \cite[][Proposition~2.2.2]{AGF_PartI} the Kontsevich--Zagier integral representations for certain weight-4 automorphic Green's functions, in the theorem below. \begin{theorem}[Integral Representations for $ G_2^{\mathfrak H/\overline{\varGamma}_0(N)},N\in\{1$, $2$, $3$, $4\}$]\label{thm:Eichler_KZ_wt4_Green} We have the following Kontsevich--Zagier integral representation for the weight-4 automorphic Green's function on the full modular group $ PSL(2,\mathbb Z)$:\begin{align}
G_2^{\mathfrak H/PSL(2,\mathbb Z)}(z,z')={}&\frac{\pi}{\I \frac{iP_{-1/6}(2\alpha_{1}(z')-1)}{P_{-1/6}(1-2\alpha_{1}(z'))}}\R\int_{1-2\alpha_{1}(z')}^1[P_{-1/6}(\xi)]^2\rho_{2,-1/6}(\xi|z)\times\notag\\&\times\left[\frac{iP_{-1/6}(-\xi)}{P_{-1/6}(\xi)}-\vphantom{\overline{\frac{1}{}}}\frac{iP_{-1/6}(2\alpha_{1}(z')-1)}{P_{-1/6}(1-2\alpha_{1}(z'))}\right]\left[\frac{iP_{-1/6}(-\xi)}{P_{-1/6}(\xi)}-\overline{\left(\frac{iP_{-1/6}(2\alpha_{1}(z')-1)}{P_{-1/6}(1-2\alpha_{1}(z'))}\right)}\right]\D \xi\notag\\&+\frac{\pi}{\I \frac{iP_{-1/6}(2\alpha_{1}(z')-1)}{P_{-1/6}(1-2\alpha_{1}(z'))}
}\R\int_{-1}^1[P_{-1/6}(\xi)]^2\rho_{2,-1/6}(\xi|z)\D \xi,\quad \emph{a.e. }z,z'\in\Int \mathfrak  D_{1}\label{eq:G2_z_z'_arb1}
\intertext{and the following integral representations for degrees $ \nu\in\{-1/4,-1/3,-1/2\}$  with corresponding  levels  $ N=4\sin^2(\nu\pi)\in\{2,3,4\}$:}G_2^{\mathfrak H/\overline{\varGamma}_0(N)}(z,z')={}&\frac{\pi}{\sqrt{N}\I\frac{iP_\nu(2\alpha_{N}(z')-1)}{P_\nu(1-2\alpha_{N}(z'))}}\R\int_{1-2\alpha_{N}(z')}^1[P_\nu(\xi)]^2\varrho_{2,\nu}(\xi|z)\times\notag\\&\times\left[\frac{iP_\nu(-\xi)}{P_\nu(\xi)}-\vphantom{\overline{\frac{1}{}}}\frac{iP_\nu(2\alpha_{N}(z')-1)}{P_\nu(1-2\alpha_{N}(z'))}\right]\left[\frac{iP_\nu(-\xi)}{P_\nu(\xi)}-\overline{\left(\frac{iP_\nu(2\alpha_{N}(z')-1)}{P_\nu(1-2\alpha_{N}(z'))}\right)}\right]\D \xi\notag\\&+\frac{\pi}{\sqrt{N}\I \frac{iP_{\nu}(2\alpha_{N}(z')-1)}{P_\nu(1-2\alpha_{N}(z'))}
}\R\int_{-1}^1[P_{\nu}(-\xi)]^2\varrho_{2,\nu}(\xi|z)\D \xi,\quad \emph{a.e. }z,z'\in\mathfrak \Int\mathfrak D_N.\label{eq:G2Hecke234_Pnu1}\end{align}Here,  in the integral representations, we have, for \emph{a.e.}~$ \zeta,z\in\Int\mathfrak D_N$,  \begin{align}&
\varrho_{2,\nu}(1-2\alpha_N(\zeta)|z)\notag\\:={}&-\frac{\I z}{2\pi\ }\frac{\partial}{\partial \I z}\left\{ \frac{1}{\I z} \frac{1}{\alpha_{N}(\zeta)-\alpha_{N}(z)}\frac{1}{[P_\nu(1-2\alpha_N(z))]^2}\right\},\quad \notag\\{}&\text{with }N=4\sin^2(\nu\pi)\in\{1,2,3,4\},-\frac{1}{2}\leq\nu<0,
\end{align}and $ \rho_{2,-1/6}(\xi|z)=\varrho_{2,-1/6}(\xi|z)+\varrho_{2,-1/6}(\xi|-1/z)=\varrho_{2,-1/6}(\xi|z)+\varrho_{2,-1/6}(-\xi|z)$.\qed\end{theorem}
\begin{remark}
According to  the residue analysis of the Kontsevich--Zagier integrals   \cite[][\S\S2.1--2.2]{AGF_PartI}, in Eqs.~\ref{eq:G2_z_z'_arb1} and \ref{eq:G2Hecke234_Pnu1}, one may choose the paths of integration  as any curves in the double slit plane $ \xi\in \mathbb C\smallsetminus((-\infty,-1]\cup[1,+\infty))$ that circumvent the singularities of the integrands.
 By setting the conditions ``a.e.~$ z,z'$'', we are excluding the scenarios where $ \alpha_N(z)=\alpha_N(z')$ or $ 1-2\alpha_N(z)\in(-\infty,-1]\cup[1,+\infty)$ or  $ 1-2\alpha_N(z')\in(-\infty,-1]\cup[1,+\infty)$, so as to evade the non-integrable singularities in the integrands and the branch cuts of the Legendre functions.\eor\end{remark}\begin{remark}
It is also worth pointing out that one may alleviate the notations   in  Eqs.~\ref{eq:G2_z_z'_arb1} and \ref{eq:G2Hecke234_Pnu1}   with the help of  Ramanujan's relations \cite[][Eq.~2.2.22]{AGF_PartI}\begin{align}
z={}&\frac{i P_\nu(2\alpha_N(z)-1)}{\sqrt{N}P_\nu(1-2\alpha_N(z))},&&
\forall z\in\Int\mathfrak D_N,\label{eq:Pnu_ratio_to_z}
\end{align}where the degrees $ \nu\in\{-1/6,-1/4,-1/3,-1/2\}$ correspond to levels $ N=4\sin^2(\nu\pi)\in\{1,2,3,4\}$.\eor\end{remark}

In this section, we focus our energy on a precursor to the Kontsevich--Zagier integrals in   Eqs.~\ref{eq:G2_z_z'_arb1}--\ref{eq:G2Hecke234_Pnu1}:\begin{align}
[P_{\nu }(2\beta-1)]^2\int_0^{\beta}\frac{[P_\nu(1-2t)]^2\D t}{t-\alpha}-P_\nu(1-2\beta)P_\nu(2\beta-1)\int_0^{\beta}\frac{P_\nu(1-2t)P_\nu(2t-1)\D t}{t-\alpha},\label{eq:KZ_int_precursor}
\end{align} and reformulate it as multiple integrals amenable to further  analysis. \begin{theorem}[An Integral Identity Involving Legendre Functions]\label{thm:KZ_precursor}For $ -1<\nu<0$ and $ 0<\beta<\alpha<1$, we have the following identity:{\allowdisplaybreaks\begin{align}&[P_{\nu }(2\beta-1)]^2\int_0^{\beta}\frac{[P_\nu(1-2t)]^2\D t}{t-\alpha}-P_\nu(1-2\beta)P_\nu(2\beta-1)\int_0^{\beta}\frac{P_\nu(1-2t)P_\nu(2t-1)\D t}{t-\alpha}\notag\\={}&\frac{\pi[P_{\nu }(2\alpha-1)]^{2}P_\nu(1-2\beta)P_\nu(2\beta-1)}{2\sin(\nu\pi)}\notag\\{}&-\frac{\sin^2(\nu\pi)}{2\pi^2}P_\nu(2\beta-1)\int_0^1\left\{\int_0^U\left[\int_W^1\frac{\D V}{\mathbb Y_{\nu,\alpha}(V)}\right]\frac{\D W}{\mathbb Y_{-\nu-1,\alpha}(W)}\right\}\frac{\D U}{\mathbb Y_{\nu,\beta}(U)}\notag\\{}&-\frac{\sin^2(\nu\pi)}{2\pi^2}P_\nu(2\beta-1)\int_0^1\left\{\int_0^U\left[\int_W^1\frac{\D V}{\mathbb Y_{-\nu-1,\alpha}(V)}\right]\frac{\D W}{\mathbb Y_{\nu,\alpha}(W)}\right\}\frac{\D U}{\mathbb Y_{-\nu-1,\beta}(U)}\notag\\{}&-\frac{ \sin(\nu\pi)P_{\nu }(1-2\alpha)P_\nu(2\beta-1)}{2\pi}\int^{1/\alpha}_0\left[\frac{1}{\mathscr U^{\nu+1}(1-\alpha\mathscr  U)^{\nu+1}}\int^1_{\mathscr U}\frac{u^{\nu}}{(1-\beta u)^{-\nu}}\left( \frac{1-\mathscr U}{1-u} \right)^\nu\frac{\D u}{1-u}\right]\D\mathscr U\notag\\{}&-\frac{ \sin(\nu\pi)P_{\nu }(1-2\alpha)P_\nu(2\beta-1)}{2\pi}\int^{1/\alpha}_0\left[\frac{1}{\mathscr U^{-\nu}(1-\alpha\mathscr  U)^{-\nu}}\int^1_{\mathscr U}\frac{u^{-\nu-1}}{(1-\beta u)^{\nu+1}}\left( \frac{1-\mathscr U}{1-u} \right)^{-\nu-1}\frac{\D u}{1-u}\right]\D\mathscr U.\label{eq:precursor_int_id}
\end{align}}Here in Eq.~\ref{eq:precursor_int_id}, all the integration paths are straight-line segments joining the end points, and the algebraic function $ \mathbb Y_{\nu,\alpha}$ is defined in Eq.~\ref{eq:Y_nu_alpha_defn}.      \qed\end{theorem}To verify the integral identity stated in the theorem above, we will rewrite Eq.~\ref{eq:KZ_int_precursor} as multiple integrals of elementary functions (\S\ref{subsec:Interaction}), before performing birational transformations on these multiple integrals (\S\ref{subsec:bir_intn_entropy}), and applying the Newton--Leibniz--Stokes formula to them (\S\ref{subsec:Legendre_add_intn_entropy}). The analytic constraint $ 0<\beta<\alpha<1$ in Theorem \ref{thm:KZ_precursor} is actually dispensable, upon subsequent analysis in  \S\ref{sec:Green_HC}, so that  Eq.~\ref{eq:KZ_int_precursor} can be reformulated as other forms of multiple integrals for generic  moduli parameters $ \alpha$ and $\beta$.
\subsection{Definition and examples for interaction entropies\label{subsec:Interaction}}
While   deriving the Kontsevich--Zagier integral representations for automorphic Green's functions satisfying the cusp-form-free condition $ \dim\mathscr S_k(\varGamma_0(N))=0$ in \cite[][\S2]{AGF_PartI},
we have encountered integrals of  meromorphic modular forms in the spirit of Eichler--Shimura theory. These Kontsevich--Zagier integrals motivate us to consider a class of integrations in the complex plane, which will be termed as  ``interaction entropies''.\begin{definition}[Interaction Entropy] A formal finite sum of integrals\begin{align}
S(\alpha\Vert\beta)=\sum_{m=1}^M\int_{L_m(\alpha,\beta)}^{U_m(\alpha,\beta)} f_m(t,\alpha,\beta)
\D t\end{align}is called an \textit{interaction entropy} if the following three conditions hold:\begin{enumerate}[leftmargin=*,  label=(IE\arabic*),ref=(IE\arabic*),
widest=a, align=left]\item \label{itm:IE1}(\textbf{Analyticity})
The functions $ f_m(t,\alpha,\beta),m=1,\dots,M$ are holomorphic with respect to $(t,\alpha, \beta)$ in a certain non-void open subset of $ \mathbb C^3$.  \item \label{itm:IE2}\textbf{(Periodness)} The functions $ L_{m}(\alpha,\beta),U_m(\alpha,\beta),m=1,\dots,M$ are algebraic.  In their domains of analyticity, each summand of $ S(\alpha\Vert\beta)$  maps algebraic arguments to members in the extended ring of Kontsevich--Zagier periods, \textit{i.e.} the relation\begin{align}\int_{L_m(\alpha,\beta)}^{U_m(\alpha,\beta)} f_m(t,\alpha,\beta)
\D t\in\widehat{\mathscr P}_{\mathrm{KZ}}=\mathscr P_{\mathrm{KZ}}[1/\pi],  \end{align}holds whenever the choice of  $\alpha,\beta\in\overline{\mathbb Q}$ and a path in the $t$-plane joining $L_m(\alpha,\beta)$ to $ U_m(\alpha,\beta)$  guarantees absolute convergence of the integral. \item \label{itm:IE3}(\textbf{Symmetry}) There exists a Gau{\ss}--Manin   differential operator $ \widehat O_\alpha$ in  the variable $ \alpha $, such that both $ \widehat O_\alpha S(\alpha\Vert\beta)$ and $ \widehat O_\alpha S(\beta\Vert\alpha) $ assume path-independent values, and the following differential equation\begin{align}
\widehat O_\alpha S(\alpha\Vert\beta)=\widehat O_\alpha S(\beta\Vert\alpha)=\frac{1}{\alpha-\beta}
\label{eq:IE_symm}
\end{align}holds.\eor\end{enumerate}\end{definition}  To show that such a definition is not vacuous, we provide some concrete examples that are relevant to the  Kontsevich--Zagier integral representations for automorphic Green's functions. \begin{proposition}[Interaction Entropies of Weight 4 and Degree $\nu$]\label{prop:S_nu_alpha_beta_recip} For  $ \nu\in\mathbb  (-1,0)$
and \emph{a.e.}~$ \alpha,\beta\in(\mathbb C\smallsetminus\mathbb R)\cup(0,1)$, we set \begin{align} \Lambda_\nu(\alpha,\beta)={}&[P_\nu(1-2\alpha)]^2[P_{\nu }(2\beta-1)]^2\mathbb Z+P_\nu(1-2\alpha)P_\nu(2\alpha-1)P_{\nu }(1-2\beta)P_{\nu }(2\beta-1)\mathbb Z\notag\\{}&+[P_\nu(2\alpha-1)]^2[P_{\nu }(1-2\beta)]^2\mathbb Z,\label{eq:Lambda_nu_alpha_beta_defn}\end{align} then the following bivariate function \begin{align}S_{\nu}(\alpha\Vert\beta)\equiv{}&[P_{\nu }(2\beta-1)]^2\int_0^{\beta}\frac{[P_\nu(1-2t)]^2\D t}{t-\alpha}-P_\nu(1-2\beta)P_\nu(2\beta-1)\int_0^{\beta}\frac{P_\nu(1-2t)P_\nu(2t-1)\D t}{t-\alpha}\notag\\{}&-P_\nu(2\beta-1)P_\nu(1-2\beta)\int_0^{1-\beta}\frac{P_\nu(1-2t)P_\nu(2t-1)\D t}{t-1+\alpha}+[P_{\nu }(1-2\beta)]^2\int_0^{1-\beta}\frac{[P_\nu(1-2t)]^2\D t}{t-1+\alpha},\notag\\{}&\pmod{2\pi i\Lambda_\nu(\alpha,\beta)},\label{eq:S_nu_alpha_beta_defn}\end{align}is well-defined up to the residues arising from the simple poles, provided that all the paths of integration circumvent the singularities of the integrands. The $ \frac{\mathbb C}{2\pi i\Lambda_\nu(\alpha,\beta)}$-valued function $ S_{\nu}(\alpha\Vert\beta)$ satisfies a functional identity:\begin{align}&S_{\nu}(\alpha\Vert\beta)-S_{\nu}(\beta\Vert\alpha)\notag\\\equiv{}& \frac{\pi}{2\sin(\nu\pi)}\{[P_\nu(1-2\alpha)]^2+[P_\nu(2\alpha-1)]^2\}P_\nu(2\beta-1)P_\nu(1-2\beta)
\notag\\{}&-\frac{\pi}{2\sin(\nu\pi)}\{[P_\nu(1-2\beta)]^2+[P_\nu(2\beta-1)]^2\}P_\nu(2\alpha-1)P_\nu(1-2\alpha),\pmod{2\pi i\Lambda_\nu(\alpha,\beta)}.\label{eq:S_recip}\end{align}Furthermore, for  each $ \nu\in\mathbb Q\cap(-1,0)$, the expression\begin{align}
\frac{\pi^2}{2\sin^2(\nu\pi)}S_\nu(\alpha\Vert\beta)
\end{align}represents an interaction entropy.
\end{proposition}\begin{proof}
It is routine to check that the $ C^3$-solutions to the following third order  differential equation of Appell type (see, for example, \cite[][Eq.~23]{Appell1881} and \cite[][Eq.~16]{Zhou2013Pnu})
\begin{align}{\widehat A}_tf(t):=\frac{\partial}{\partial t}\left\{t(1-t)\frac{\partial  }{\partial t}\left[t(1-t)\frac{\partial f(t)  }{\partial t}\right]+4\nu (\nu +1)t(1-t)f(t)\right\}+4\nu(\nu+1)\frac{2t-1}{2}f(t)=0\label{eq:At_diff_op}
\end{align}can be written as linear combinations of $ [P_\nu(1-2t)]^2$, $ P_\nu(1-2t)P_\nu(2t-1)$ and $ [P_\nu(2t-1)]^2$. Moreover, it is also easy to verify that the Appell differential operator $ {\widehat A}_t$ satisfies
\begin{align}{\widehat A}_t
\frac{1}{t-\alpha}+{\widehat A}_\alpha
\frac{1}{t-\alpha}=0.
\end{align}

Thus, integrating by parts in the variable $t$, one can establish the following inhomogeneous third-order differential equation with respect to the variable $ \alpha$:
\begin{align}&
{\widehat A}_\alpha \left\{[P_{\nu }(2\beta-1)]^2\int_0^{\beta}\frac{[P_\nu(1-2t)]^2\D t}{t-\alpha}-P_\nu(1-2\beta)P_\nu(2\beta-1)\int_0^{\beta}\frac{P_\nu(1-2t)P_\nu(2t-1)\D t}{t-\alpha}\right\}\notag\\={}&\frac{\sin ^2  (\nu \pi)}{\pi ^2}\frac{1}{\alpha-\beta}+\frac{\sin(\nu\pi)}{\pi}\beta(1-\beta)[P_\nu(1-2\beta)P_\nu(2\beta-1)]^{2}\frac{\partial}{\partial\beta}\left[\frac{1}{\alpha-\beta}\frac{1}{P_\nu(1-2\beta)P_\nu(2\beta-1)}\right].\label{eq:A_L_M}
\end{align}To complete the computations above, we have availed ourselves with the recursion relations \((2\nu+1)(1-\xi^{2})\partial P_\nu(\xi)/\partial \xi=\nu(\nu+1)[P_{\nu+1}(\xi)-P_{\nu-1}(\xi)]$, $ (2\nu+1)\xi P_\nu(\xi)=(\nu+1)P_{\nu+1}(\xi)+\nu P_{\nu-1}(\xi)$ as well as Legendre's relation $ P_\nu(\xi)P_{\nu+1}(-\xi)+P_\nu(-\xi)P_{\nu+1}(\xi)+\frac{2\sin(\nu\pi)}{\pi(\nu+1)}=0$.

In Eq.~\ref{eq:A_L_M},  we may trade $ \alpha$ for $ 1-\alpha$ and $ \beta$ for $ 1-\beta$, which leads us to
\begin{align}&
{\widehat A}_\alpha \left\{ [P_{\nu }(1-2\beta)]^2\int_0^{1-\beta}\frac{[P_\nu(1-2t)]^2\D t}{t-1+\alpha} -P_\nu(2\beta-1)P_\nu(1-2\beta)\int_0^{1-\beta}\frac{P_\nu(1-2t)P_\nu(2t-1)\D t}{t-1+\alpha}\right\}\notag\\={}&\frac{\sin ^2  (\nu \pi)}{\pi ^2}\frac{1}{\alpha-\beta}-\frac{\sin(\nu\pi)}{\pi}\beta(1-\beta)[P_\nu(2\beta-1)P_\nu(1-2\beta)]^{2}\frac{\partial}{\partial\beta}\left[\frac{1}{\alpha-\beta}\frac{1}{P_\nu(2\beta-1)P_\nu(1-2\beta)}\right].\label{eq:A_L_M'}\tag{\ref{eq:A_L_M}$'$}
\end{align} Adding up Eqs.~\ref{eq:A_L_M} and \ref{eq:A_L_M'}, we obtain
\begin{align}&
{\widehat A}_\alpha S_{\nu}(\alpha\Vert\beta)=\frac{\sin ^2  (\nu \pi)}{\pi ^2}\frac{2}{\alpha-\beta}.
\label{eq:A_S_ab}\end{align}
Meanwhile,  by a standard Wro\'nskian argument for inhomogeneous differential equations, we can verify that
\begin{align}&
{\widehat A}_\alpha S_{\nu}(\beta\Vert\alpha)=\frac{\sin ^2  (\nu \pi)}{\pi ^2}\frac{2}{\alpha-\beta}.\label{eq:A_S_ba}
\end{align}Here, in both Eqs.~\ref{eq:A_S_ab} and \ref{eq:A_S_ba}, the identities are true irrespective of the winding numbers possessed by the integration paths for the definitions of $ S_{\nu}(\alpha\Vert\beta)$ and $ S_{\nu}(\beta\Vert\alpha)$, because every member in $ \Lambda_\nu(\alpha,\beta)$ is annihilated by the operator $\widehat A_\alpha $.

 Therefore, we have $S_{\nu}(\alpha\Vert\beta)-S_{\nu}(\beta\Vert\alpha)\equiv f_1(\beta)[P_\nu(1-2\alpha)]^2+f_2(\beta)P_\nu(1-2\alpha)P_\nu(2\alpha-1)+f_{3}(\beta)[P_\nu(2\alpha-1)]^2 \pmod{2\pi i\Lambda_\nu(\alpha,\beta)}$. Now, we may momentarily  adjust the integration paths so that the net  contribution from the residues  to  $ S_{\nu}(\alpha\Vert\beta)-S_{\nu}(\beta\Vert\alpha)$ is $ 0\in2\pi i\Lambda_\nu(\alpha,\beta) $. In these circumstances,  we can determine  $ f_1(\beta)=f_3(\beta)=\frac{\pi}{2\sin(\nu\pi)} P_\nu(1-2\beta)P_\nu(2\beta-1)$  from the asymptotic behavior of $ S_{\nu}(\alpha\Vert\beta)-S_{\nu}(\beta\Vert\alpha)$ in the regimes $ \alpha\to0^+$ and $ \alpha\to1-0^+$:{\allowdisplaybreaks\begin{align}\lim_{\alpha\to0^+}\frac{S_\nu(\alpha\Vert\beta)}{[P_\nu(2\alpha-1)]^2 }={}&-P_\nu(1-2\beta)P_\nu(2\beta-1)\lim_{\alpha\to0^+}\int_0^{\beta}\frac{P_\nu(1-2t)P_\nu(2t-1)}{t-\alpha}\frac{\D t}{[P_\nu(2\alpha-1)]^2}\notag\\={}&-P_\nu(1-2\beta)P_\nu(2\beta-1)\frac{\sin(\nu\pi)}{\pi}\lim_{\alpha\to0^+}\int_0^{\beta}\frac{\log t}{t-\alpha}\frac{\D t}{[P_\nu(2\alpha-1)]^2}\notag\\={}&\frac{\pi P_\nu(1-2\beta)P_\nu(2\beta-1)}{2\sin(\nu\pi)}=\lim_{\alpha\to1-0^+}\frac{S_\nu(\alpha\Vert\beta)}{[P_\nu(1-2\alpha)]^2 };\label{eq:S_nu_lim1}
\\\lim_{\alpha\to0^+}\frac{S_\nu(\beta\Vert\alpha)}{[P_\nu(2\alpha-1)]^2 }={}&0=\lim_{\alpha\to1-0^+}\frac{S_\nu(\beta\Vert\alpha)}{[P_\nu(1-2\alpha)]^2 };\label{eq:S_nu_lim2}
\end{align}}and subsequently find out $ f_2(\beta)=-\frac{\pi}{2\sin(\nu\pi)}\{[P_\nu(1-2\beta)]^2+[P_\nu(2\beta-1)]^2\}$ after switching the r\^oles of $ \alpha$ and $ \beta$ in the expression $S_{\nu}(\alpha\Vert\beta)-S_{\nu}(\beta\Vert\alpha)$. The rest of Eq.~\ref{eq:S_recip} then  follows from analytic continuation and residue calculus.

To summarize,  for $ \nu\in\mathbb Q\cap(-1,0)$, we have two observations: (i) The Euler integral representation of Legendre functions (Eq.~\ref{eq:Euler_int_Pnu}) entails that $ P_\nu(1-2\alpha)\in\widehat{\mathscr P}_{\mathrm{KZ}}$ for $ \alpha\in\overline{\mathbb Q}\smallsetminus[1,+\infty)$; (ii) The differential operator defined in Eq.~\ref{eq:At_diff_op} is of generalized Picard--Fuchs type, hence a member in the Gau{\ss}--Manin systems.  In view of these  observations,  the expression $ \frac{\pi^2}{2\sin^2(\nu\pi)}S_\nu(\alpha\Vert\beta)$ fulfills all the  criteria \ref{itm:IE1}--\ref{itm:IE3}, thus qualifying as an interaction entropy.
\end{proof} \begin{remark}By  \cite[][Propositions~2.1.1, 2.1.2 and 2.2.2]{AGF_PartI}, we know that for the degrees  $\nu\in\{-1/6,-1/4,-1/3,-1/2\}$ and the corresponding levels $ N=4\sin^2(\nu\pi)\in\{1,2,3,4\}$, one may express the weight-4 automorphic Green's functions $ G_2^{\mathfrak H/\overline{\varGamma}_0(N)}(z,z')$, a.e.~$ z,z'\in\Int\mathfrak D_N$  in terms of\begin{align}
\mathscr I_N(z,z'):=-\frac{1}{N}\R\left\{\I z\I z'
\frac{\partial}{\partial\I z}\frac{\partial}{\partial\I z'}\frac{S_\nu(\alpha_N(z)\Vert\alpha_N(z'))}{\I z\I z'[P_{\nu}(1-2\alpha_{N}(z))]^2[P_{\nu}(1-2\alpha_{N}(z'))]^2}\right\},\label{eq:I_N_z_z'_defn}
\end{align}in the sense that $ G_2^{\mathfrak H/\overline{\varGamma}_0(N)}(z,z')=\mathscr I_N(z,z')$ for $ N\in\{2,3,4\}$ and $G_2^{\mathfrak H/\overline{\varGamma}_0(1)}(z,z')=\mathscr I_1(z,z')+\mathscr I_N(z,-1/z') $. Thus one may combine the functional identity in Eq.~\ref{eq:S_recip}  with  the Ramanujan relations in Eq.~\ref{eq:Pnu_ratio_to_z} to support Green's reciprocity $ G_2^{\mathfrak H/\overline{\varGamma}_0(N)}(z,z')=G_2^{\mathfrak H/\overline{\varGamma}_0(N)}(z',z)$, a.e.~$ z,z'\in\Int\mathfrak D_N$ for $ N\in\{1,2,3,4\}$.   \eor\end{remark}\begin{remark}In \cite[][\S2.2]{AGF_PartI}, we referred to certain expressions associated with   $ S_\nu(\alpha\Vert\beta)$ (defined in Eq.~\ref{eq:S_nu_alpha_beta_defn}) as ``Legendre--Ramanujan representations'' for automorphic Green's functions. Such nomenclature is meant to acknowledge  Ramanujan's contribution  to the reformulation of certain modular forms  as algebraic expressions involving fractional degree Legendre functions. It is now appropriate to call the formula for    $ S_\nu(\alpha\Vert\beta)$  given in  Eq.~\ref{eq:S_nu_alpha_beta_defn} as the ``interaction entropy in Legendre--Ramanujan form''.   \eor\end{remark}\begin{remark}For $ \nu\in\mathbb Q\cap(-1,0)$, we have already seen that  the interaction entropy $ S_\nu(\alpha\Vert\beta)$ can be represented as   absolutely convergent integrals of algebraic functions over algebraic domains, for almost every  $ \alpha,\beta\in\overline{\mathbb Q}$. If there are other integral representations of  $ S_\nu(\alpha\Vert\beta)$  as Kontsevich--Zagier periods (perhaps apparently very different from the Legendre--Ramanujan form, like the one related to Eq.~\ref{eq:precursor_int_id}), then we would anticipate the existence of finitely many steps of elementary manipulations (permissible in the Kontsevich--Zagier program \ref{itm:KZ-1}--\ref{itm:KZ-3}) that connect the  Legendre--Ramanujan form to  these alternative  integral representations of   $ S_\nu(\alpha\Vert\beta)$, according to a motivic Hodge conjecture of Kontsevich and Zagier \cite[][\S1.2 and \S4.1]{KontsevichZagier}.
In  \S\S\ref{sec:Interaction_Entropies_Transformations}--\ref{sec:Green_HC}, we shall present explicit constructions of these elementary transformations for Kontsevich--Zagier periods. \eor\end{remark}
To facilitate the analysis of   interaction entropies $S_\nu(\alpha\Vert\beta)$ of weight 4 and degree $\nu$, we will first convert them into multiple integrals of elementary functions.

A useful identity for manipulating integrals of algebraic functions is Eq.~\ref{eq:F_ab} in the next lemma, which turns a product into an integral. Physicists often attribute this  technique (along with some of its generalizations)  to the seminal  works of R. P.  Feynman and J. \selectlanguage{german}Schwinger\selectlanguage{english} in the late 1940s  \cite[][p.~190]{QFT},  in the context of  quantum field theory. However, we note that the mathematical ideas behind  Eq.~\ref{eq:F_ab} had essentially appeared in an article of A. Erd\'elyi  published in 1937 \cite{Erdelyi1937b}.
In the following lemma, we state the ``EFS formula'' (Eq.~\ref{eq:F_ab}, which may be also termed as the Erd\'elyi identity, or the Feynman parametrization, or the \selectlanguage{german}Schwinger\selectlanguage{english} trick), without reiterating its standard proof.    \begin{lemma}[EFS Formula]\label{lm:Feynman_parameter}  Suppose that $ A,B\in\mathbb C\smallsetminus(-\infty,0]$, $ \R a>0$, $\R b>0$, and that the straight-line segment joining $A$ to $B$ also lies in $ \mathbb C\smallsetminus(-\infty,0]$, then we have the following identity\begin{align}\frac{1}{A^a B^b}=\frac{\Gamma(a+ b)}{\Gamma(a)\Gamma(b)}\int_0^1\frac{u^{a-1}(1-u)^{b-1}\D u}{[uA+(1-u)B]^{a+b}},
\label{eq:F_ab}
\end{align}where the fractional powers are defined according to the convention in Eq.~\ref{eq:frac_pow}.\qed\end{lemma}\begin{remark}Based on the formula \cite[][Eq.~4.1]{Erdelyi1937b}\begin{align}
\frac{1}{(1-t\xi)^a}=\frac{\Gamma(c)}{\Gamma(a)\Gamma(c-a)}\int_0^1\frac{\tau^{a-1}(1-\tau)^{c-a-1}\D\tau}{(1-t\tau\xi)^c},\quad 0<\R a<\R c,t\xi\in\mathbb C\smallsetminus[1,+\infty),\label{eq:Erdelyi_frac}
\end{align}which is a special case of Eq.~\ref{eq:F_ab} with $ A=1-t\xi$, $ B=1$, $ b=c-a$, Erd\'elyi symmetrized Euler's integral representation for hypergeometric functions (Eq.~\ref{eq:Euler_int}) into  the following form \cite[][Eq.~4.2]{Erdelyi1937b}:\begin{align}_2F_1\left( \left.\begin{array}{c}
a,b \\
c \\
\end{array}\right|\xi \right)={}&\frac{[\Gamma(c)]^{2}}{\Gamma(a)\Gamma(b)\Gamma(c-a)\Gamma(c-b)}\int_0^1\int _{0}^{1}\frac{t^{b-1}\tau^{a-1}(1-t)^{c-b-1}(1-\tau)^{c-a-1}}{(1-t\tau \xi)^c}\D t\D\tau,\quad
\label{eq:Erdelyi_int}
\end{align}on the conditions that $ \R c>\R a>0,\R c>\R b>0,-\pi<\arg(1-\xi)<\pi$. In particular, Erd\'elyi's  formula allows us to  establish a double integral representation for Legendre functions \begin{align}
P_{\nu}(1-2t)={}&\frac{\sin^{2} (\nu\pi   )}{\pi^{2}}\int_0^{1}\int _{0}^{1}\frac{\left[ \frac{U(1-V)}{V(1-U)} \right]^\nu}{1-tUV}\frac{\D U\D V}{(1-U)V}=:\frac{\sin^{2} (\nu\pi   )}{\pi^{2}}\iint\frac{\mathbb D_{\nu,0}U\mathbb D_{-\nu-1,0}V}{1-tUV}\label{eq:Pnu_double_int_repn}
\end{align}for $ -1<\nu<0$ and $ t\in\mathbb C\smallsetminus[1,+\infty)$. Here, one may wish to refer to Eq.~\ref{eq:mathbb_D_nu_alpha_nu_defn} for the usage of the notation $ \mathbb D_{\nu,\alpha}u$. \eor\end{remark}\begin{remark}The following variation on Eq.~\ref{eq:Erdelyi_frac}:\begin{align}
0={}&\left.\frac{\partial}{\partial\varepsilon}\right|_{\varepsilon=0}\frac{1}{(1-\xi)^{\nu+1}}=\left.\frac{\partial}{\partial\varepsilon}\right|_{\varepsilon=0}\left[\frac{\Gamma(1+\varepsilon)}{\Gamma(\nu+1)\Gamma(-\nu+\varepsilon)}\int_0^1\frac{u^{\nu}(1-u)^{\varepsilon-\nu-1}\D u}{(1-u\xi)^{1+\varepsilon}}\right]\notag\\={}&-\frac{\gamma_0+\psi^{(0)}(-\nu)}{(1-\xi)^{\nu+1}}-\frac{\sin(\nu\pi)}{\pi}\int\frac{\log\frac{1-u}{1-u\xi}}{1-u \xi}\mathbb D_{\nu,0}u,\quad -1<\nu<0,\xi\in\mathbb C\smallsetminus[1,+\infty)\label{eq:Erdelyi_log_deriv}
\end{align}will be used later in Proposition \ref{prop:multi_integrals_Jacobi_involution_again} and Lemma \ref{lm:nu_reflection_symm}.\eor\end{remark}As  direct applications of  Eq.~\ref{eq:F_ab} in the last lemma, we investigate some triple integrals of algebraic functions that represent products of two Legendre functions, as well as some triple integrals of elementary functions that are related to the summands in $ S_\nu(\alpha\Vert\beta)$. In the proposition below, we will also see that the arrangement of the variables $ \alpha$ and $ \beta$ in the interaction entropy  $ S_\nu(\alpha\Vert\beta)$ is compatible with the notation for  Kullback--Leibler divergence (Eq.~\ref{eq:Kullback-Leibler}).\begin{proposition}[Some Triple Integrals]\label{prop:triple_integrals_PnuPnu}\begin{enumerate}[label=\emph{(\alph*)}, ref=(\alph*), widest=a] \item We have the following triple integral representations for certain products of two Legendre functions:\begin{align}
[P_{\nu}(1-2t)]^{2}={}&-\frac{\sin^3(\nu\pi)}{\pi^3}\int_0^{1}\int_0^{1}\int_0^{1}\frac{\left[ \frac{U(1-V)(1-W)}{(1-U)VW} \right]^{\nu}\frac{\D U\D V\D W}{(1-U)VW}}{1-tUW-tV(1-W)}\notag\\=:{}&\mathbb E_{\nu,0}^U\mathbb E_{-\nu-1,0}^V\mathbb E_{-\nu-1,0}^W\frac{1}{1-tUW-tV(1-W)},\label{eq:Pnu_sqr_F_trip_int}\intertext{}P_\nu(1-2t)P_\nu(2t-1)={}&-\frac{\sin^3(\nu\pi)}{\pi^3}\int_0^{1}\int_0^{1}\int_0^{1}\frac{\left[ \frac{U(1-V)(1-W)}{(1-U)VW} \right]^{\nu}\frac{\D U\D V\D W}{(1-U)VW}}{1-tUW-(1-t)V(1-W)}\notag\\=:{}&\mathbb E_{\nu,0}^U\mathbb E_{-\nu-1,0}^V\mathbb E_{-\nu-1,0}^W\frac{ 1}{1-tUW-(1-t)V(1-W)},\label{eq:Pnu_mir_F_trip_int}\end{align}where $ \nu\in(-1,0)$ and $t\in(\mathbb C\smallsetminus\mathbb R)\cup(0,1)$. The Legendre function $ P_\nu$  follows the definition in Eq.~\ref{eq:Euler_int_Pnu}. \item \label{itm:triple_int_Ent}For  $ \nu\in(-1,0)$,  $ \alpha,\beta\in(\mathbb C\smallsetminus\mathbb R)\cup(0,1)$, we have the following integral identities: \begin{align}&\int_0^{\beta}\frac{[P_\nu(1-2t)]^{2}-[P_\nu(1-2\alpha)]^2}{t-\alpha}\D t\notag\\={}&\frac{\sin^3(\nu\pi)}{\pi^3} \int_0^{1}\int_0^{1}\int_0^{1}\frac{\log[1-\beta UW-\beta V(1-W)]}{1-\alpha UW-\alpha V(1-W)}\left[ \frac{U(1-V)(1-W)}{(1-U)VW} \right]^{\nu}\frac{\D U\D V\D W}{(1-U)VW}\notag\\=:{}&\frac{\sin^3(\nu\pi)}{\pi^3}\iiint\frac{\log[1-\beta UW-\beta V(1-W)]}{1-\alpha UW-\alpha V(1-W)}\mathbb D_{\nu,0}U\mathbb D_{-\nu-1,0}V\mathbb D_{-\nu-1,0}W,\label{eq:Pnu_sqr_KZ_int2trip_int}\intertext{and}
&\int_0^{\beta}\frac{P_\nu(1-2t)P_\nu(2t-1)-P_\nu(1-2\alpha)P_\nu(2\alpha-1)}{t-\alpha}\D t\notag\\={}&\frac{\sin^3(\nu\pi)}{\pi^3} \int_0^{1}\int_0^{1}\int_0^{1}\frac{\log\frac{1-\beta UW-(1-\beta) V(1-W)}{1- V(1-W)}}{1-\alpha UW-(1-\alpha) V(1-W)}\left[ \frac{U(1-V)(1-W)}{(1-U)VW} \right]^{\nu}\frac{\D U\D V\D W}{(1-U)VW}\notag\\=:{}&\frac{\sin^3(\nu\pi)}{\pi^3}\iiint\frac{\log\frac{1-\beta UW-(1-\beta) V(1-W)}{1- V(1-W)}}{1-\alpha UW-(1-\alpha) V(1-W)}\mathbb D_{\nu,0}U\mathbb D_{-\nu-1,0}V\mathbb D_{-\nu-1,0}W.\label{eq:Pnu_mir_KZ_int2trip_int}\end{align}
\end{enumerate}\end{proposition}\begin{proof}\begin{enumerate}[label=(\alph*),widest=a]\item We use the EFS formula (Eq.~\ref{eq:F_ab}) to combine the denominators in \begin{align}P_{\nu}(1-2t)P_{-\nu-1}(1-2t')=\frac{\sin^{2}(\nu\pi)}{\pi^{2}}\int_0^{1}\frac{U^{\nu}(1-U)^{-\nu-1}\D U}{(1-tU)^{-\nu}}\int_0^{1}\frac{V^{-\nu-1}(1-V)^{\nu}\D V}{(1-t'V)^{\nu+1}},\end{align}thereby confirming the triple integral representations in  Eqs.~\ref{eq:Pnu_sqr_F_trip_int} and \ref{eq:Pnu_mir_F_trip_int}. Bearing in mind that $ P_\nu(1)=1$, we can rephrase these triple integrals in terms of Legendre expectations (Eq.~\ref{eq:Legendre_expectation_defn}).   \item Integrate   Eqs.~\ref{eq:Pnu_sqr_F_trip_int} and \ref{eq:Pnu_mir_F_trip_int} in the variable $t$.\qedhere\end{enumerate}Later in \S\S\ref{subsec:bir_intn_entropy}--\ref{subsec:Legendre_add_intn_entropy}, we will use the notations $ \mathbb D_{\nu,\alpha}u$ to abbreviate integral formulae, as far as possible. Only occasionally will we spell out the underlying algebraic expression for clarification.
  \end{proof}
\subsection{Birational transformations for interaction entropies of weight 4 and degree $\nu$\label{subsec:bir_intn_entropy}}
In this subsection, we will  first apply Principles \ref{itm:KZ-1}--\ref{itm:KZ-2} in the Kontsevich--Zagier program to the triple integrals in Eqs.~\ref{eq:Pnu_sqr_KZ_int2trip_int} and \ref{eq:Pnu_mir_KZ_int2trip_int} for the $ \alpha=\beta$ scenario (self-interaction entropies), and then give a similar treatment to the   $ \alpha\neq\beta$ case (relative interaction entropies).

To handle the logarithmic factor in the integrands of multiple integrals, we need to compute some derivatives of hypergeometric functions with respect to their parameters, in the next lemma.
\begin{lemma}[Frobenius--Zagier Process]For $ \nu\in(-1,0)$ and $ \alpha\in(\mathbb C\smallsetminus\mathbb R)\cup(0,1)$, we have the following closed-form evaluations:\begin{align}&
\left.\frac{\partial}{\partial \varepsilon}\right|_{\varepsilon=0}{_2}F_1\left(\left. \begin{array}{c}
-\nu,\nu+1\ \\
1+\varepsilon\ \\
\end{array} \right| \alpha\right)\notag\\={}&\frac{\pi P_{\nu}(2\alpha-1)}{2\sin(\nu\pi)}+\frac{P_\nu(1-2\alpha)}{2}\left[ -2\gamma_{0}-\psi^{(0)}(-\nu)-\psi^{(0)}(\nu+1)+\log \frac{1-\alpha}{\alpha} \right],\label{eq:FZ_a}\intertext{}&\left.\frac{\partial}{\partial \varepsilon}\right|_{\varepsilon=0}{_2}F_1\left(\left. \begin{array}{c}
-\nu+\varepsilon,\nu+1+\varepsilon\ \\
1\ \\
\end{array} \right| \alpha\right)\notag\\={}&-P_\nu(1-2\alpha)\log(1-\alpha),\label{eq:FZ_b}
\end{align}which can be recast into the following integral formulae:\begin{align}&-\frac{\sin(\nu\pi)}{\pi} \int\log(1-U)(\mathbb D_{\nu,\alpha}U+\mathbb D_{-\nu-1,\alpha}U)\notag\\={}&\frac{\pi P_{\nu}(2\alpha-1)}{\sin(\nu\pi)}+P_\nu(1-2\alpha)\log \frac{1-\alpha}{\alpha},\label{eq:FZ_a'}\intertext{}&
\frac{\sin(\nu\pi)}{\pi} \int\log(1-\alpha U)(\mathbb D_{\nu,\alpha}U+\mathbb D_{-\nu-1,\alpha}U)\notag\\={}&-\frac{\sin(\nu\pi)}{\pi} \int\log\frac{U}{1-U}(\mathbb D_{\nu,\alpha}U+\mathbb D_{-\nu-1,\alpha}U)\notag\\={}&-P_\nu(1-2\alpha)\log(1-\alpha),\label{eq:FZ_b'}\intertext{}&\frac{\sin^{2} (\nu\pi   )}{\pi^{2}}\iint\frac{\log(1-\alpha UV)}{1-\alpha UV}\mathbb D_{\nu,0}U\mathbb D_{-\nu-1,0}V\notag\\={}&\frac{\pi P_{\nu}(2\alpha-1)}{2\sin(\nu\pi)}+\frac{P_\nu(1-2\alpha)}{2}\left[ -2\gamma_{0}-\psi^{(0)}(-\nu)-\psi^{(0)}(\nu+1)+\log \frac{1-\alpha}{\alpha} \right].\label{eq:FZ_a''}
\end{align}\end{lemma}\begin{proof}As we differentiate the hypergeometric equation \begin{align}
\left[\alpha(1-\alpha)\frac{\partial^2 }{\partial\alpha^2}+(1+\varepsilon-2\alpha)\frac{\partial }{\partial \alpha}+\nu(\nu+1)\right] {_2}F_1\left(\left. \begin{array}{c}
-\nu,\nu+1\ \\
1+\varepsilon\ \\
\end{array} \right| \alpha\right)=0
\end{align}with respect to $ \varepsilon$ and specialize to the case where $ \varepsilon=0$, we obtain an inhomogeneous Legendre differential equation:\begin{align}
\left[\alpha(1-\alpha)\frac{\partial^2 }{\partial\alpha^2}+(1-2\alpha)\frac{\partial }{\partial \alpha}+\nu(\nu+1)\right]\left.\frac{\partial}{\partial \varepsilon}\right|_{\varepsilon=0} {_2}F_1\left(\left. \begin{array}{c}
-\nu,\nu+1\ \\
1+\varepsilon\ \\
\end{array} \right| \alpha\right)=-\frac{\partial P_{\nu}(1-2\alpha) }{\partial \alpha}.
\end{align}One can then show that the difference between the two sides of Eq.~\ref{eq:FZ_a} is a solution to the homogeneous Legendre differential equation, and this solution is bounded at both boundary points $ \alpha\to0^+$ and $ \alpha\to1^-$. Such a solution must  be identically zero, so  Eq.~\ref{eq:FZ_a} is verified.

By the Frobenius process \cite[][Eq.~1.3.8]{Slater}, one sees that the expression\begin{align}
P_\nu(1-2\alpha)\log \alpha+\left.\frac{\partial}{\partial \varepsilon}\right|_{\varepsilon=0}{_2}F_1\left(\left. \begin{array}{c}
-\nu+\varepsilon,\nu+1+\varepsilon\ \\
1\ \\
\end{array} \right| \alpha\right)+2\left.\frac{\partial}{\partial \varepsilon}\right|_{\varepsilon=0}{_2}F_1\left(\left. \begin{array}{c}
-\nu,\nu+1\ \\
1+\varepsilon\ \\
\end{array} \right| \alpha\right)\label{eq:FP_ex}
\end{align}   solves the homogeneous Legendre differential equation. Thus,  Eq.~\ref{eq:FP_ex} is a linear combination of $ P_\nu(1-2\alpha)$ and $ P_\nu(2\alpha-1)$. The combination coefficients can be determined by asymptotic analysis near the boundary points  $ \alpha\to0^+$ and $ \alpha\to1^-$.
This shows how  Eq.~\ref{eq:FZ_b} follows from Eq.~\ref{eq:FZ_a}.

One can verify  Eqs.~\ref{eq:FZ_a'}--\ref{eq:FZ_b'} by differentiating Euler's integral representation for hypergeometric functions (Eq.~\ref{eq:Euler_int}) and referring to Eqs.~\ref{eq:FZ_a}--\ref{eq:FZ_b}. If one differentiates Erd\'elyi's  double integral representation for hypergeometric functions (Eq.~\ref{eq:Erdelyi_int}) in the parameter $c$, one obtains\begin{align}&
\left.\frac{\partial}{\partial \varepsilon}\right|_{\varepsilon=0}{_2}F_1\left(\left. \begin{array}{c}
-\nu,\nu+1\ \\
1+\varepsilon\ \\
\end{array} \right| \alpha\right)-P_\nu(1-2\alpha)[-2\gamma_{0}-\psi^{(0)}(-\nu)-\psi^{(0)}(\nu+1)]\notag\\={}&\frac{\sin^{2} (\nu\pi   )}{\pi^{2}}\iint\frac{\log\frac{(1-U)(1-V)}{1-\alpha UV}}{1-\alpha UV}\mathbb D_{\nu,0}U\mathbb D_{-\nu-1,0}V.\label{eq:FZ_a''_prep1}
\end{align}As one computes\begin{align}&
\frac{\sin^{2} (\nu\pi   )}{\pi^{2}}\iint\frac{\log[(1-U)(1-V)]}{1-\alpha UV}\mathbb D_{\nu,0}U\mathbb D_{-\nu-1,0}V\notag\\={}&-\frac{\sin(\nu\pi)}{\pi}\left[ \int\log(1-U)\mathbb D_{\nu,\alpha}U+  \int\log(1-V)\mathbb D_{-\nu-1,\alpha}V\right]\label{eq:FZ_a''_prep2}
\end{align} upon reference to Eq.~\ref{eq:Erdelyi_frac}, one may deduce Eq.~\ref{eq:FZ_a''} from Eqs.~\ref{eq:FZ_a'}, \ref{eq:FZ_a''_prep1} and \ref{eq:FZ_a''_prep2}.    \end{proof}\begin{remark}For the case where  $ \nu=-1/3$, Eqs.~\ref{eq:FZ_a} and \ref{eq:FZ_b} were derived by Zagier \cite{Zagier1998}, in order to prove a conjecture about the mirror symmetry of Schoen's Calabi--Yau 3-folds. Our proof of Eqs.~\ref{eq:FZ_a} and \ref{eq:FZ_b}  for $ \nu\in(-1,0)$ follows closely Zagier's argument in \cite{Zagier1998}, and makes use of the classical Frobenius process, hence the namesake.\eor\end{remark}\begin{remark}For $ \nu=-1/2$, one can also verify the integral formulae in   Eqs.~\ref{eq:FZ_a'}--\ref{eq:FZ_b'} by the Jacobi elliptic functions. See, for example, \cite[][\S22.5]{WhittakerWatson1927} or \cite[][Lemma 3.3.2]{AGF_PartI}.     \eor\end{remark}

In the next proposition, we will transform the triple integrals in  Eqs.~\ref{eq:Pnu_sqr_KZ_int2trip_int} and \ref{eq:Pnu_mir_KZ_int2trip_int}    with repeated use of  Jacobi involutions. Here, a Jacobi involution is a variable substitution in the form of \begin{align}
\mathscr U=\frac{1-U}{1-\lambda U}\quad \Longleftrightarrow\quad U=\frac{1-\mathscr U}{1-\lambda\mathscr U},\label{eq:Jacobi_involution_U}
\end{align} which occurs naturally in the ``complementary angle transformation'' of the Jacobi elliptic functions \cite[][item 122.03]{ByrdFriedman}:
\begin{align}
\sn^2(\mathbf K(\sqrt{\lambda})-u|\lambda)=\frac{1-\sn^2(u|\lambda)}{1-\lambda\sn^2(u|\lambda)}\quad \Longleftrightarrow\quad \sn^2(u|\lambda)=\frac{1-\sn^2(\mathbf K(\sqrt{\lambda})-u|\lambda)}{1-\lambda\sn^2(\mathbf K(\sqrt{\lambda})-u|\lambda)}.\label{eq:Jacobi_involution_sn}
\end{align}

\begin{proposition}[Self-Interaction Entropies of Weight 4 and Degree $\nu$]\label{prop:multi_integrals_Jacobi_involution}For   $ \nu\in(-1,0)$,  $ \alpha\in(\mathbb C\smallsetminus\mathbb R)\cup(0,1)$,  we have\begin{align}&
\frac{\sin^3(\nu\pi)}{2\pi^3}\iiint\frac{\log[1-\alpha UW-\alpha V(1-W)]}{1-\alpha UW-\alpha V(1-W)}(\mathbb D_{\nu,0}U\mathbb D_{-\nu-1,0}V\mathbb D_{-\nu-1,0}W+\mathbb D_{-\nu-1,0}U\mathbb D_{\nu,0}V\mathbb D_{\nu,0}W)\notag\\={}&-\frac{\pi P_\nu(1-2\alpha) P_{\nu}(2\alpha-1)}{2\sin(\nu\pi)}+[P_\nu(1-2\alpha)]^{2}\left[ \gamma_{0}+\frac{\psi^{(0)}(-\nu)+\psi^{(0)}(\nu+1)}{2}-\log \frac{1-\alpha}{\sqrt{\alpha}} \right]\notag\\{}&+\frac{\sin^2(\nu\pi)}{2\pi^2}\iint\log(1-\alpha UW)(\mathbb D_{\nu,\alpha}U\mathbb D_{\nu,\alpha}W+\mathbb D_{-\nu-1,\alpha}U\mathbb  D_{-\nu-1,\alpha}W),\label{eq:log_self_sqr} \\\intertext{and}&\frac{\sin^3(\nu\pi)}{2\pi^3}\iiint\frac{\log\frac{1-\alpha UW-(1-\alpha) V(1-W)}{1- V(1-W)}}{1-\alpha UW-(1-\alpha) V(1-W)}(\mathbb D_{\nu,0}U\mathbb D_{-\nu-1,0}V\mathbb D_{-\nu-1,0}W+\mathbb D_{-\nu-1,0}U\mathbb D_{\nu,0}V\mathbb D_{\nu,0}W)\notag\\={}&\frac{\pi[P_\nu(1-2\alpha)]^{2}}{2\sin(\nu\pi)}-\frac{\pi[P_\nu(2\alpha-1)]^{2}}{2\sin(\nu\pi)}+P_\nu(1-2\alpha)P_\nu(2\alpha-1)\left[ \gamma_{0}+\frac{\psi^{(0)}(-\nu)+\psi^{(0)}(\nu+1)}{2}-\log\frac{1- \alpha}{\sqrt{\alpha}} \right]\notag\\&+\frac{\sin^2(\nu\pi)}{2\pi^2}\iint\log\left(1+\frac{ \alpha V W}{1-V}\right)(\mathbb D_{\nu,1-\alpha}V\mathbb D_{\nu,\alpha}W+\mathbb D_{-\nu-1,1-\alpha}V\mathbb  D_{-\nu-1,\alpha}W).\label{eq:log_self_mir}\end{align}
\end{proposition}\begin{proof} Without loss of generality, we may assume that $ \alpha\in(0,1)$, and employ a variable substitution\begin{align}W=\frac{1-\mathscr W}{1-\alpha U\mathscr W},\quad\text{or equivalently,}\quad \frac{W}{1-W}=\frac{1-\mathscr W}{(1-\alpha U)\mathscr W}\label{eq:bir_sinsqr_0}\end{align}to transform \begin{align}&
\iiint\frac{\log[1-\alpha UW-\alpha V(1-W)]}{1-\alpha UW-\alpha V(1-W)}
\mathbb D_{\nu,0}U\mathbb D_{-\nu-1,0}V\mathbb D_{-\nu-1,0}W\notag\\={}&\iiint\frac{\log(1-\alpha U)+\log(1-\alpha V\mathscr W)-\log(1-\alpha U\mathscr W)}{1-\alpha V\mathscr W}\mathbb D_{\nu,\alpha}U\mathbb D_{-\nu-1,0}V\mathbb D_{\nu,0}\mathscr W.\label{eq:self_sqr_T}\end{align}In view of Eqs.~\ref{eq:FZ_b'} and \ref{eq:FZ_a''}, we can simplify{\allowdisplaybreaks\begin{align}
&-\frac{\sin^3(\nu\pi)}{\pi^3}\iiint\frac{\log(1-\alpha U)+\log(1-\alpha V\mathscr W)}{1-\alpha V\mathscr W}(\mathbb D_{\nu,\alpha}U\mathbb D_{-\nu-1,0}V\mathbb D_{\nu,0}\mathscr W+\mathbb D_{-\nu-1,\alpha}U\mathbb D_{\nu,0}V\mathbb D_{-\nu-1,0}\mathscr W)\notag\\={}&[P_\nu(1-2\alpha)]^{2}\log(1-\alpha)+2P_\nu(1-2\alpha)\times\notag\\{}&\times\left\{ \frac{\pi P_{\nu}(2\alpha-1)}{2\sin(\nu\pi)}+\frac{P_\nu(1-2\alpha)}{2}\left[ -2\gamma_{0}-\psi^{(0)}(-\nu)-\psi^{(0)}(\nu+1)+\log \frac{1-\alpha}{\alpha} \right] \right\}.
\end{align}}Meanwhile, Eq.~\ref{eq:Erdelyi_frac} allows us to put down\begin{align}&
-\frac{\sin(\nu\pi)}{\pi}\int\frac{\mathbb D_{-\nu-1,0}V}{1-\alpha V\mathscr W}=\frac{1}{(1-\alpha\mathscr W)^{-\nu}},
\end{align}and consequently, \begin{align}
-\frac{\sin(\nu\pi)}{\pi}\iint\frac{\log(1-\alpha U\mathscr W)\mathbb D_{-\nu-1,0}V\mathbb D_{\nu,0}\mathscr W}{1-\alpha V\mathscr W}=\int \log(1-\alpha U\mathscr W)\mathbb D_{\nu,\alpha}\mathscr W.
\end{align}
So far, we can verify Eq.~\ref{eq:log_self_sqr}  for $ \alpha\in(0,1)$, and the rest follows from analytic continuation.

In a similar vein, one may  use variable substitutions to  demonstrate  Eq.~\ref{eq:log_self_mir}  for $ \alpha\in(0,1)$, using the following transformation: \begin{align}
W=1-\frac{1-\mathscr W}{1-(1-\alpha) V\mathscr W},\quad\text{or equivalently,}\quad \frac{W}{1-W}=\frac{[1-(1-\alpha) V]\mathscr W}{1-\mathscr W},\label{eq:bir_sinsqr_0'}
\end{align}which is a variation on  Eq.~\ref{eq:bir_sinsqr_0}.     \end{proof}\begin{remark}For $ \nu=-1/2$, Eqs.~\ref{eq:Pnu_sqr_KZ_int2trip_int} and \ref{eq:log_self_sqr} bring us the following ``Jacobi self-interaction entropy formula'': \begin{align}
&\lim_{\mu\to\lambda-0^+}
\left\{\int_0^{\mu}\frac{[\mathbf K(\sqrt{t})]^{2}\D t}{t-\lambda}-[\mathbf K(\sqrt{\lambda})]^2\log\left( 1-\frac{\mu}{\lambda} \right)\right\}\notag\\
={}&-\frac{2}{\pi}\int_0^{\pi/2}\int_0^{\pi/2}\int_0^{\pi/2}\frac{\log(1-\lambda\sin^2\phi\sin^2\theta-\lambda\cos^2\phi\sin^2\psi)}{1-\lambda\sin^2\phi\sin^2\theta-\lambda\cos^2\phi\sin^2\psi}\D\phi\D\theta\D\psi\notag\\={}&\frac{\pi}{2}\mathbf{K}(\sqrt{\lambda})\mathbf{K}(\sqrt{1-\lambda})-[\mathbf{K}(\sqrt{\lambda})]^{2}\log \frac{4(1-\lambda)}{\sqrt{\lambda}}+\int_0^{\pi/2}\int_0^{\pi/2}\frac{\log(1-\lambda\sin^2\theta\sin^2\varphi)\D\theta\D\varphi}{\sqrt{1-\lambda\sin^2\theta}\sqrt{1-\lambda\sin^2\smash[b]{\varphi}}}\notag\\={}&\frac{\pi}{3}\mathbf K(\sqrt{\lambda})\mathbf K(\sqrt{1-\lambda})-\frac{2[\mathbf K(\sqrt{\lambda})]^2}{3}\log\frac{4(1-\lambda)}{\sqrt{\lambda}}\label{eq:log_self_sqr_K} \end{align}for $ 0<\lambda<1$. Here, the last double integral can be evaluated with the aid of elliptic functions \cite[][Eq.~3.3.12]{AGF_PartI}. Such an evaluation was essential in the proof of \cite[][Theorem 1.2.2(b)]{AGF_PartI} (see also Theorem~\ref{thm:app_HC}(a) of this article). See \S\ref{subsec:G2_Hecke4_GZ_renorm} for an alternative approach to Theorem~\ref{thm:app_HC}(a), without invoking elliptic functions. \eor\end{remark}
\begin{remark}One may  use Eq.~\ref{eq:log_self_sqr_K} to evaluate the entropy of a probability distribution supported on the cube $ [0,\pi/2]^3$ (Eq.~\ref{eq:Ent_rho_lambda}). Since the maximum entropy $ 3\log\frac{\pi}{2}$ is attained by the uniform distribution with the same support,  we obtain an inequality\begin{align}
\left[\frac{2\mathbf K(\sqrt{\lambda})}{\pi} \right]^6<{}&\frac{\lambda}{16(1-\lambda)^2}e^{\frac{\pi\mathbf K(\sqrt{1-\lambda})}{\mathbf K(\sqrt{\lambda})}}
\end{align}for $ 0<\lambda<1$. As pointed out by an anonymous referee, such an inequality can be derived by other means. For instance, using the relations \cite[cf.][Eqs.~3.0.1, 3.0.3 and 3.3.19a]{AGF_PartI}\begin{align}
\eta^{24}\left(i\frac{\mathbf K(\sqrt{1-\lambda})}{\mathbf K(\sqrt{\lambda})}\right) =\left[\frac{2\mathbf K(\sqrt{\lambda})}{\pi} \right]^{12}\frac{\lambda^2(1-\lambda)^2}{256},\quad \eta^{24}\left(2i\frac{\mathbf K(\sqrt{1-\lambda})}{\mathbf K(\sqrt{\lambda})}\right) =\left[\frac{2\mathbf K(\sqrt{\lambda})}{\pi} \right]^{12}\frac{\lambda^{4}(1-\lambda)}{65536},
\end{align}one can show that \begin{align}
\left[\frac{2\mathbf K(\sqrt{\lambda})}{\pi} \right]^6\frac{16(1-\lambda)^2}{\lambda}={}&\frac{}{}\frac{\eta^{36}\left(i\frac{\mathbf K(\sqrt{1-\lambda})}{\mathbf K(\sqrt{\lambda})}\right)}{\eta^{24}\left(2i\frac{\mathbf K(\sqrt{1-\lambda})}{\mathbf K(\sqrt{\lambda})}\right) } <e^{\frac{\pi\mathbf K(\sqrt{1-\lambda})}{\mathbf K(\sqrt{\lambda})}}
\end{align} holds for $ 0<\lambda<1$. Here, the last inequality follows from the infinite product expansion of the Dedekind eta function (Eq.~\ref{eq:Dedekind_eta_defn}).\eor\end{remark}
Now we will start handling the triple integrals in  Eqs.~\ref{eq:Pnu_sqr_KZ_int2trip_int} and \ref{eq:Pnu_mir_KZ_int2trip_int} in the situations where  $ \alpha\neq\beta$. Our attention will be   focused on triple integrals that belong to the type of ``relative entropies'', namely, \begin{align}&\iiint\frac{\log\frac{1-\beta UW-\beta V(1-W)}{1-\alpha UW-\alpha V(1-W)}}{1-\alpha UW-\alpha V(1-W)}\mathbb D_{\nu,0}U\mathbb D_{-\nu-1,0}V\mathbb D_{-\nu-1,0}W\\\text{and }&\iiint\frac{\log\frac{1-\beta UW-(1-\beta) V(1-W)}{1-\alpha UW-(1-\alpha) V(1-W)}}{1-\alpha UW-(1-\alpha) V(1-W)}\mathbb D_{\nu,0}U\mathbb D_{-\nu-1,0}V\mathbb D_{-\nu-1,0}W.\end{align} We shall refer to  such integrals as ``relative interaction entropies''.

The next proposition could be regarded as a continuation of Proposition~\ref{prop:multi_integrals_Jacobi_involution}, where we will again resort  to Jacobi involutions for the transformations of multiple integrals.
\begin{proposition}[Birational Transformations of Relative Interaction Entropies]\label{prop:multi_integrals_Jacobi_involution_again} For $ \nu\in(-1,0)$ and $ 0<\beta<\alpha<1$, we have {\allowdisplaybreaks\begin{align}
&\frac{\sin^3(\nu\pi)}{2\pi^3}\iiint\frac{\log\frac{1-\beta UW-\beta V(1-W)}{1-\alpha UW-\alpha V(1-W)}}{1-\alpha UW-\alpha V(1-W)}(\mathbb D_{\nu,0}U\mathbb D_{-\nu-1,0}V\mathbb D_{-\nu-1,0}W+\mathbb D_{-\nu-1,0}U\mathbb D_{\nu,0}V\mathbb D_{\nu,0}W)\notag\\={}&-[P_\nu(1-2\alpha)]^2\left[\log \left( 1-\frac{\beta}{\alpha} \right)+\gamma_{0}+\frac{\psi^{(0)}(-\nu)+\psi^{(0)}(\nu+1)}{2}\right]+\frac{[P_{\nu}(1-2\alpha)]^{2}}{2} \log(1-\alpha)\notag\\{}&+\frac{\sin^3(\nu\pi)}{2\pi^3}\iiint\frac{(1-\beta V)\log\frac{1-[\alpha W +\beta(1-W )]U }{1-\frac{\alpha W +\beta(1-W )}{\alpha} }\frac{}{}}{1-  [\alpha W +\beta(1-W )]V}(\mathbb D_{\nu,\alpha}U\mathbb D_{-\nu-1,\beta}V\mathbb D_{\nu,0}W+\mathbb D_{-\nu-1,\alpha}U\mathbb D_{\nu,\beta}V\mathbb D_{-\nu-1,0}W)
\label{eq:Pnu_sqr_intn_bir}\\\intertext{and}&\frac{\sin^3(\nu\pi)}{2\pi^3}\iiint\frac{\log\frac{1-\beta UW-(1-\beta) V(1-W)}{1-\alpha UW-(1-\alpha) V(1-W)}}{1-\alpha UW-(1-\alpha) V(1-W)}(\mathbb D_{\nu,0}U\mathbb D_{-\nu-1,0}V\mathbb D_{-\nu-1,0}W+\mathbb D_{-\nu-1,0}U\mathbb D_{\nu,0}V\mathbb D_{\nu,0}W)\notag\\={}&-P_{\nu}(1-2\alpha)P_{\nu}(2\alpha-1)\left[\log \left( 1-\frac{\beta}{\alpha} \right)+\gamma_{0}+\frac{\psi^{(0)}(-\nu)+\psi^{(0)}(\nu+1)}{2}\right]\notag\\{}&+\frac{P_{\nu}(1-2\alpha)P_{\nu}(2\alpha-1)}{2} \log(1-\alpha)+\frac{\sin^3(\nu\pi)}{2\pi^3}\iiint\frac{[1-(1-\beta)V]\log\frac{1-[\alpha W +\beta(1-W )]U }{1-\frac{\alpha W +\beta(1-W )}{\alpha} }\frac{}{}}{1-  [1-\alpha W -\beta(1-W )]V}\times\notag\\{}&\times(\mathbb D_{\nu,\alpha}U\mathbb D_{-\nu-1,1-\beta}V\mathbb D_{\nu,0}W+\mathbb D_{-\nu-1,\alpha}U\mathbb D_{\nu,1-\beta}V\mathbb D_{-\nu-1,0}W)
.\label{eq:Pnu_mir_intn_bir}\end{align}}\end{proposition}
\begin{proof} We may introduce a new variable $ \mathscr W$ satisfying\begin{align}
\frac{W}{1-W}=\frac{1-\beta V}{1-\alpha U}\frac{1-\mathscr W}{\mathscr W},\label{eq:arctan_bir_1}
\end{align}
and verify that \begin{align}&
\iiint\frac{\log\frac{1-\beta UW-\beta V(1-W)}{1-\alpha UW-\alpha V(1-W)}}{1-\alpha UW-\alpha V(1-W)}\mathbb D_{\nu,0}U\mathbb D_{-\nu-1,0}V\mathbb D_{-\nu-1,0}W\notag\\={}&\iiint\frac{(1-\beta V)\log\frac{(1-\beta  V)\{1-[\alpha\mathscr  W +\beta(1-\mathscr W )]U\}\ }{(1-\alpha  U)\{1-[\alpha\mathscr  W +\beta(1-\mathscr W )]V\}\ }\frac{}{}}{1-  [\alpha\mathscr  W +\beta(1-\mathscr W )]V}\mathbb D_{\nu,\alpha}U\mathbb D_{-\nu-1,\beta}V\mathbb D_{\nu,0}\mathscr W.\label{eq:rel_sqr_T}
\end{align}
Integrating over $ \mathscr W$ and  recalling Eq.~\ref{eq:FZ_b'}, we can show that\begin{align}&\iiint\frac{(1-\beta V)\log(1-\alpha  U)}{1-  [\alpha\mathscr  W +\beta(1-\mathscr W )]V}(\mathbb D_{\nu,\alpha}U\mathbb D_{-\nu-1,\beta}V\mathbb D_{\nu,0}\mathscr W+\mathbb D_{-\nu-1,\alpha}U\mathbb D_{\nu,\beta}V\mathbb D_{-\nu-1,0}\mathscr W)\notag\\={}&-\frac{\pi}{\sin(\nu\pi)}
\iint\log(1-\alpha  U)(\mathbb D_{\nu,\alpha}U\mathbb D_{-\nu-1,\alpha}V\mathbb +\mathbb D_{-\nu-1,\alpha}U\mathbb D_{\nu,\alpha}V)\notag\\={}&-\frac{\pi^{3}}{\sin^{3}(\nu\pi)}[P_{\nu}(1-2\alpha)]^{2} \log(1-\alpha).
\end{align}In the meantime, we may evaluate \begin{align}&
\iiint\frac{(1-\beta V)\log\frac{(1-\beta V)\left[ 1-\frac{\alpha\mathscr  W +\beta(1-\mathscr W )}{\alpha} \right] }{1-[\alpha\mathscr  W +\beta(1-\mathscr W )]V }{}}{1-[\alpha\mathscr  W +\beta(1-\mathscr W )]V}\mathbb D_{\nu,\alpha}U\mathbb D_{-\nu-1,\beta}V\mathbb D_{\nu,0}\mathscr W\notag\\={}&-\frac{\pi ^3}{ \sin ^3(\nu\pi   )}[P_\nu(1-2\alpha)]^2\left[\log \left( 1-\frac{\beta}{\alpha} \right)+\gamma_{0}+\psi^{(0)}(-\nu)\right],\label{eq:trip_int_redn}
\end{align}by setting $ u=\mathscr W, \xi=\frac{(\alpha-\beta)V}{1-\beta V}$ in  Eq.~\ref{eq:Erdelyi_log_deriv}.
So far, Eq.~\ref{eq:Pnu_sqr_intn_bir} is confirmed.

Similar to Eq.~\ref{eq:arctan_bir_1}, substituting\begin{align}
\frac{W}{1-W}=\frac{1-(1-\beta) V}{1-\alpha U}\frac{1-\mathscr W}{\mathscr W}\label{eq:arctan_bir_2}
\end{align} before specializing  Eq.~\ref{eq:Erdelyi_log_deriv} to $ u=\mathscr W, \xi=\frac{(\beta-\alpha)V}{1-(1-\beta) V}$, we can verify   Eq.~\ref{eq:Pnu_mir_intn_bir}. \end{proof}
 \subsection{Legendre addition formulae for  interaction entropies of weight $4$ and  degree $\nu$\label{subsec:Legendre_add_intn_entropy}}
Thus far, in our manipulations of multiple integrals, we have been mainly revolving around linear additivity and algebraic variable substitutions, that is, Principles \ref{itm:KZ-1}--\ref{itm:KZ-2}. In \S\ref{subsec:Legendre_add_intn_entropy}, we will exploit extensively the  Newton--Leibniz--Stokes formula   \ref{itm:KZ-3}, so as to achieve  further
reductions of relative interaction entropies. During the presentation below, we will   use both the notations $ \mathbb Y_{\nu,\alpha}(u)$ and $\mathbb{D}_{\nu,\alpha}u$ that were introduced in Eqs.~\ref{eq:Y_nu_alpha_defn} and \ref{eq:mathbb_D_nu_alpha_nu_defn}. \begin{lemma}[Some Addition Formulae of Legendre Type]\label{lm:Legendre_addition}\begin{enumerate}[label=\emph{(\alph*)}, ref=(\alph*), widest=a] \item For any  $ \nu\in(-1,0)$, $ 0<\beta<\alpha<1$ and $  W\in (0,1)$, we have an addition formula: \begin{align}&
\int\frac{P_{\nu}(2\beta-1)(1-\beta V)}{1-[\alpha W +\beta(1-W)]V}\mathbb D_{-\nu-1,\beta}V-\int\frac{P_{\nu}(1-2\beta)[1-(1-\beta) V]}{1-  [1-\alpha W-\beta(1-W )]V}\mathbb D_{-\nu-1,1-\beta}V\notag\\={}&(\alpha -\beta ) \left\{\frac{[\alpha W+\beta(1-W)][1-\alpha W -\beta(1-W)]}{W}\right\}^\nu\int_0^W\left\{\frac{\omega}{[\alpha\omega+\beta(1-\omega)][1-\alpha\omega -\beta(1-\omega)]}\right\}^{\nu+1}\frac{\D \omega}{\omega},\label{eq:3rd_kind_cancellation_Pnu}
\end{align}and the following identity:
\begin{align}&
\iiint\frac{P_{\nu}(2\beta-1)(1-\beta V)\log\frac{1-[\alpha W +\beta(1-W )]U }{1-\frac{\alpha W +\beta(1-W )}{\alpha} }\frac{}{}}{1-  [\alpha W +\beta(1-W )]V}\mathbb D_{\nu,\alpha}U\mathbb D_{-\nu-1,\beta}V\mathbb D_{\nu,0}W\notag\\{}&-\iiint\frac{P_{\nu}(1-2\beta)[1-(1-\beta) V]\log\frac{1-[\alpha W +\beta(1-W )]U }{1-\frac{\alpha W +\beta(1-W )}{\alpha} }\frac{}{}}{1-  [1-\alpha W -\beta(1-W )]V}\mathbb D_{\nu,\alpha}U\mathbb D_{-\nu-1,1-\beta}V\mathbb D_{\nu,0}W\notag\\={}&\int_0^{1}\frac{\D U}{\mathbb Y_{\nu,\alpha}(U)}\int^{1}_{\frac{1-\alpha}{1-\beta}}\frac{\D X}{\mathbb Y_{-\nu-1,1-\beta}(X)}\int_{\frac{1-(1-\beta)X}{\alpha}}^1\frac{\D V}{\mathbb Y_{\nu,\alpha}(V)}\log\frac{1-\alpha UV}{1-V }.\label{eq:rel_intn_add_0}\end{align}\item  For $ \nu\in(-1,0)$, $ 0<\beta<\alpha<1$, $ X\in \left(\frac{1-\alpha}{1-\beta} ,1 \right)$ and  $ \mathscr U\in(0,1)$, we have the following addition formulae{\allowdisplaybreaks\begin{align}&
\int_{\frac{1-(1-\beta)X}{\alpha}}^1 \frac{\alpha V}{1-\alpha\mathscr UV} \frac{\D V}{\mathbb Y_{\nu,\alpha}(V)}\notag\\={}&-\frac{(\nu+1)\alpha}{\mathbb Y_{-\nu-1,\alpha}(\mathscr U)}
\int_0^{\mathscr U}\frac{u\D u}{\mathbb Y_{\nu,\alpha}(u)}\int_{\frac{1-(1-\beta)X}{\alpha}}^1\frac{\D V}{\mathbb Y_{\nu,\alpha}(V)}+\frac{(\nu+1)\alpha}{\mathbb Y_{-\nu-1,\alpha}(\mathscr U)}\int_0^{\mathscr U}\frac{\D u}{\mathbb Y_{\nu,\alpha}(u)}\int_{\frac{1-(1-\beta)X}{\alpha}}^1\frac{V\D V}{\mathbb Y_{\nu,\alpha}(V)}\notag\\{}&-\frac{1}{\mathbb Y_{-\nu-1,\alpha}(\mathscr U)}\frac{X^{\nu+1}[1-(1-\beta)X]^{\nu+1}}{\left( X-\frac{1-\alpha}{1-\beta} \right)^\nu}\int_0^{\mathscr U}\frac{(1-\beta)u}{1-[1-(1-\beta)X]u}\frac{\D u}{\mathbb Y_{\nu,\alpha}(u)},\label{eq:3rd_kind_to_2nd_kind}
\\\intertext{}&
\frac{1}{\mathbb Y_{-\nu-1,\alpha}(\mathscr U)}
\int^1_{\mathscr U}\frac{u\D u}{\mathbb Y_{\nu,\alpha}(u)}\notag\\={}&\int_{0}^1 \frac{ {_2}F_1\left( \left.\begin{smallmatrix}-\nu-1,\nu+1\\1\end{smallmatrix}\right|1-\alpha \right)}{1-\alpha\mathscr UV} \frac{V\D V}{\mathbb Y_{\nu,\alpha}(V)}+\int_{0}^1 \frac{P_{\nu}(1-2\alpha)-{_2}F_1\left( \left.\begin{smallmatrix}-\nu-1,\nu+1\\1\end{smallmatrix}\right|\alpha \right)}{1-(1-\alpha\mathscr U)V} \frac{V\D V}{\mathbb Y_{\nu,1-\alpha}(V)},\label{eq:3rd_kind_add_1}
\\\intertext{}&
\frac{1}{\mathbb Y_{-\nu-1,\alpha}(\mathscr U)}\int^1_{\mathscr U}\frac{\D u}{\mathbb Y_{\nu,\alpha}(u)}\notag\\={}&\int_{0}^1 \frac{\alpha P_{\nu}(2\alpha-1)}{1-\alpha\mathscr UV} \frac{V\D V}{\mathbb Y_{\nu,\alpha}(V)}+\int_{0}^1 \frac{\alpha P_{\nu}(1-2\alpha)}{1-(1-\alpha\mathscr U)V} \frac{V\D V}{\mathbb Y_{\nu,1-\alpha}(V)}.\label{eq:3rd_kind_add_2}\end{align}} \end{enumerate}
\end{lemma}
\begin{proof}\begin{enumerate}[label=(\alph*),widest=a]\item We integrate an elementary identity \begin{align}& \frac{\nu(1-\beta V)}{\mathbb Y_{-\nu-1,\beta}(V)}\left[\frac{1}{\alpha\omega+\beta(1-\omega)}-\frac{1-\beta}{(\alpha -\beta )\omega\ }\frac{1}{1-\beta V}\right]+
\frac{\partial }{\partial V }\frac{\frac{V(1-V)(1-\beta V)}{\mathbb Y_{-\nu-1,\beta}(V)}}{1-[\alpha\omega+\beta(1-\omega)]V}\notag\\={}&\frac{1-\alpha\omega-\beta(1-\omega) }{\beta-\alpha}\left\{\frac{[\alpha\omega+\beta(1-\omega)][1-\alpha\omega-\beta(1-\omega) ]}{ \omega}\right\}^\nu\times\notag\\&\times \frac{\partial}{\partial\omega}\left( \frac{\frac{1-\beta V}{\mathbb Y_{-\nu-1,\beta}(V)}}{1-[\alpha\omega+\beta(1-\omega)]V} \left\{\frac{\omega}{[\alpha\omega+\beta(1-\omega)][1-\alpha\omega-\beta(1-\omega) ]}\right\}^\nu\right)\label{eq:third_kind_elem_id_1}
\end{align}over $ V\in(0,1)$, to obtain \begin{align}
&\frac{\partial}{\partial \omega}\left( \left\{\frac{\omega}{[\alpha\omega+\beta(1-\omega)][1-\alpha\omega-\beta(1-\omega) ]}\right\}^\nu\int\frac{1-\beta V}{1-[\alpha\omega +\beta(1-\omega)]V}\mathbb D_{-\nu-1,\beta}V\right)\notag\\={}&\frac{\nu\pi}{\sin (  \nu\pi )}\left\{\frac{\omega}{[\alpha\omega+\beta(1-\omega)][1-\alpha\omega-\beta(1-\omega) ]}\right\}^\nu\times\notag\\{}&\times\left[\frac{_2F_1\left( \left.\begin{smallmatrix}\nu,-\nu\\1\end{smallmatrix}\right|\beta \right)(\alpha -\beta )}{[\alpha\omega+\beta(1-\omega)][1-\alpha\omega-\beta(1-\omega) ]}-\frac{P_{\nu }(1-2 \beta )(1-\beta)}{[1-\alpha\omega-\beta(1-\omega)]\omega }\right].
\end{align}Here, in the last step, we have used the facts that (cf.~Eqs.~\ref{eq:Euler_int} and \ref{eq:Euler_int_Pnu})\begin{align}
P_\nu(1-2\beta)={_2}F_1\left( \left.\begin{array}{c}
\nu+1,-\nu \\
1
\end{array}\right|\beta \right)={}&-\frac{\sin(\nu\pi)}{\pi}\int\mathbb D_{-\nu-1,\beta}V,\notag\\{_2}F_1\left( \left.\begin{array}{c}
\nu,-\nu \\
1
\end{array}\right|\beta \right)={}&-\frac{\sin(\nu\pi)}{\pi}\int(1-\beta V)\mathbb D_{-\nu-1,\beta}V.
\end{align}
Similarly, we deduce
\begin{align}
&\frac{\partial}{\partial \omega}\left( \left\{\frac{\omega}{[\alpha\omega+\beta(1-\omega)][1-\alpha\omega-\beta(1-\omega) ]}\right\}^\nu\int\frac{1-(1-\beta)V}{1-[1-\alpha\omega -\beta(1-\omega)]V}\frac{}{}\mathbb D_{-\nu-1,1-\beta}V\right)\notag\\={}&\frac{\nu\pi}{\sin (  \nu\pi )}\left\{\frac{\omega}{[\alpha\omega+\beta(1-\omega)][1-\alpha\omega-\beta(1-\omega) ]}\right\}^\nu\times\notag\\{}&\times\left\{-\frac{_2F_1\left( \left.\begin{smallmatrix}\nu,-\nu\\1\end{smallmatrix}\right|1-\beta \right)(\alpha -\beta )}{[\alpha\omega+\beta(1-\omega)][1-\alpha\omega-\beta(1-\omega) ]}-\frac{P_{\nu }(2 \beta-1 )\beta}{[\alpha\omega+\beta(1-\omega)]\omega }\right\}.
\end{align}

 We have $
P_{\nu}(2\beta-1){_2}F_1\left( \left.\begin{smallmatrix}\nu,-\nu\\1\end{smallmatrix}\right|\beta \right)+P_{\nu}(1-2\beta){_2}F_1\left( \left.\begin{smallmatrix}\nu,-\nu\\1\end{smallmatrix}\right|1-\beta \right)=P_{\nu}(2\beta-1)P_{\nu}(1-2\beta)+\frac{\sin(\nu\pi)}{\nu\pi}
$ according to  Legendre's relation  \cite[][Theorem 3.2.8]{AAR}, so the foregoing computations combine into{\allowdisplaybreaks\begin{align}&
\frac{\partial}{\partial \omega}\Bigg(  \left\{\frac{\omega}{[\alpha\omega+\beta(1-\omega)][1-\alpha\omega-\beta(1-\omega) ]}\right\}^\nu\times\notag\\{}&\times\Bigg\{\int\frac{P_{\nu}(2\beta-1)(1-\beta V)}{1-[\alpha\omega\ +\beta(1-\omega)]V}\frac{}{}\mathbb D_{-\nu-1,\beta}V-\int\frac{P_{\nu}(1-2\beta)[1-(1-\beta) V]}{1-  [1-\alpha\omega-\beta(1-\omega )]V}\mathbb D_{-\nu-1,1-\beta}V\Bigg\}\Bigg)\notag\\={}&\frac{\alpha -\beta}{\omega}\left\{\frac{\omega}{[\alpha\omega+\beta(1-\omega)][1-\alpha\omega-\beta(1-\omega) ]}\right\}^{\nu+1}.
\end{align}}Integrating over $ \omega\in(0,W)$, we arrive at Eq.~\ref{eq:3rd_kind_cancellation_Pnu}.

Applying the addition formula in  Eq.~\ref{eq:3rd_kind_cancellation_Pnu} to the left-hand side of Eq.~\ref{eq:rel_intn_add_0}, we reformulate it as\begin{align}{}&(\alpha -\beta ) \iint\left\{\frac{[\alpha W+\beta(1-W)][1-\alpha W -\beta(1-W)]}{W}\right\}^\nu\times\notag\\{}&\times\left(\int_0^W\left\{\frac{\omega}{[\alpha\omega+\beta(1-\omega)][1-\alpha\omega -\beta(1-\omega)]}\right\}^{\nu+1}\frac{\D \omega}{\omega}\right)\log\frac{1-[\alpha W +\beta(1-W )]U }{1-\frac{\alpha W +\beta(1-W )}{\alpha} }\mathbb D_{\nu,\alpha}U\mathbb D_{\nu,0}W.
\end{align}In the expression above,  we make the following variable substitutions:\begin{align}
\omega=\frac{(1-\beta)(1-X)}{\alpha-\beta},\qquad W=\frac{\alpha V-\beta}{\alpha-\beta},
\end{align}which turns   it into\begin{align}&\int_0^{1}\frac{\D U}{\mathbb{Y}_{\nu,\alpha}(U)}\int_{\frac{\beta}{\alpha}{}}^{1}\frac{\D V}{\mathbb{Y}_{\nu,\alpha}(V)}\int^{1}_{\frac{1-\alpha V}{1-\beta}{}}\frac{\D X}{\mathbb{Y}_{-\nu-1,1-\beta}(X)}\log\frac{1-\alpha UV}{1-V }.\label{eq:rel_intn_form_1}
\end{align}Switching the order of integrations over $ V$ and $ X$,   we arrive at the right-hand side of  Eq.~\ref{eq:rel_intn_add_0}.   \item Similar to Eq.~\ref{eq:third_kind_elem_id_1}, we put down\begin{align}&
\frac{\alpha u}{\mathbb{Y}_{ \nu,\alpha}(u)}\left[-\left( 1-\frac{V }{u} \right)\frac{\nu+1}{\mathbb{Y}_{ \nu,\alpha}(V)}+
\frac{\partial }{\partial V }\frac{\mathbb{Y}_{ -\nu-1,\alpha}(V)}{1-\alpha uV}\right]=\frac{\partial}{\partial u}\left[ \frac{\alpha V}{1-\alpha uV} \frac{\mathbb{Y}_{- \nu-1,\alpha}(u)}{\mathbb{Y}_{ \nu,\alpha}(V)}\right]\label{eq:incomp_3rd_kind_nu_diff}
\end{align} for $ \alpha,u,V\in(0,1)$ and $ \nu\in(-1,0)$.
Integrating the equation above over $ V\in\left(\frac{1-(1-\beta)X}{\alpha},1\right)$ where $ X\in\left( \frac{1-\alpha}{1-\beta},1 \right)$,  before integrating over $ u\in(0,\mathscr U)$, we are able to  establish Eq.~\ref{eq:3rd_kind_to_2nd_kind} after brief rearrangements.

 Integrating Eq.~\ref{eq:incomp_3rd_kind_nu_diff}  over $ V\in(0,1)$,   before integrating over  $ u\in(\mathscr U,1)$, we can show that \begin{align}&
-\frac{\nu+1}{ \mathbb{Y}_{- \nu-1,\alpha}(\mathscr U)}P_{\nu}(1-2\alpha)\int^1_{\mathscr U}\frac{u\D u}{\mathbb Y_{\nu,\alpha}(u)}\notag\\{}&+\frac{\nu+1}{\alpha\mathbb{Y}_{- \nu-1,\alpha}(\mathscr U)}\left[ P_{\nu}(1-2\alpha) -{_2}F_1\left( \left.\begin{array}{c}
-\nu-1,\nu+1 \\
1
\end{array}\right|\alpha \right)\right]\int^1_{\mathscr U}\frac{\D u}{\mathbb Y_{\nu,\alpha}(u)}\notag\\={}&\frac{\sin(\nu \pi)}{\pi}\int_{0}^1 \frac{ V}{1-\alpha \mathscr UV} \frac{\D V}{\mathbb Y_{\nu,\alpha}(V)}.\label{eq:3rd_kind_breakdown_a}
\end{align}Meanwhile, instead of working on  Eq.~\ref{eq:incomp_3rd_kind_nu_diff}, we may integrate another identity\begin{align}&
\frac{(1-\alpha)u }{\mathbb{Y}_{\nu,\alpha}(u)}\left\{\left[ 1-\frac{1-(1-\alpha)V }{\alpha u} \right]\frac{\nu+1}{\mathbb{Y}_{ \nu,1-\alpha}(V)}+\frac{1-\alpha u}{\alpha u}
\frac{\partial }{\partial V }\frac{\mathbb{Y}_{ -\nu-1,1-\alpha}(V)}{1-(1-\alpha u)V}\right\}= \frac{\partial}{\partial u}\left[\frac{(1-\alpha) V}{1-(1-\alpha u)V} \frac{\mathbb{Y}_{- \nu-1,\alpha}(u)}{\mathbb{Y}_{ \nu,1-\alpha}(V)}\right]
\end{align} to produce an analog of Eq.~\ref{eq:3rd_kind_breakdown_a}:\begin{align}&
\frac{\nu+1}{ \mathbb{Y}_{- \nu-1,\alpha}(\mathscr U)}P_{\nu}(2\alpha-1)\int^1_{\mathscr U}\frac{u\D u}{\mathbb Y_{\nu,\alpha}(u)}-\frac{\nu+1}{ \alpha\mathbb{Y}_{- \nu-1,\alpha}(\mathscr U)}{_2}F_1\left( \left.\begin{array}{c}
-\nu-1,\nu+1 \\
1
\end{array}\right|1-\alpha \right)\int^1_{\mathscr U}\frac{\D u}{\mathbb Y_{\nu,\alpha}(u)}\notag\\={}&\frac{\sin(\nu \pi)}{\pi}\int_{0}^1 \frac{ V}{1-(1-\alpha \mathscr U)V} \frac{\D V}{\mathbb Y_{\nu,1-\alpha}(V)}.\label{eq:3rd_kind_breakdown_b}
\end{align}Consequently,  we can solve Eqs.~\ref{eq:3rd_kind_add_1} and \ref{eq:3rd_kind_add_2} from Eqs.~\ref{eq:3rd_kind_breakdown_a} and \ref{eq:3rd_kind_breakdown_b}.
\qedhere\end{enumerate}

 \end{proof}\begin{remark}For $ \nu=-1/2$, Eq.~\ref{eq:3rd_kind_cancellation_Pnu} becomes an addition formula for the complete elliptic integrals of the third kind \cite[][items 117.02 and 117.05]{ByrdFriedman}. Our proof of Eq.~\ref{eq:3rd_kind_cancellation_Pnu} is a modest extension of    the classical theory for the complete elliptic integrals of the third kind \cite[][\S31]{Enneper}, traceable to the original work of Legendre  \cite[][Chap.~XXIII]{LegendreTomeI}. 
 \eor\end{remark}\begin{remark}One key observation of Legendre is that  complete elliptic integrals of the third kind are always expressible in terms of  incomplete elliptic integrals of the first and second kinds. Some generalizations of Legendre's observation to Legendre functions of fractional degrees $ \nu\in(-1,0)$ are given by Eqs.~\ref{eq:3rd_kind_breakdown_a} and \ref{eq:3rd_kind_breakdown_b}. \eor\end{remark}\begin{proposition}[Legendre Addition Formulae for Relative Interaction Entropies]\label{prop:Legendre_add_rel_intn_ent}\begin{enumerate}[label=\emph{(\alph*)}, ref=(\alph*), widest=a] \item We have   a pair of  integral evaluations for  $ \nu\in(-1,0)$, $ 0<\beta<\alpha<1$, $ X\in \left(\frac{1-\alpha}{1-\beta} ,1 \right)$ and $ 0<u<1$:\begin{align}\int_{\frac{1-\alpha}{1-\beta}}^1
\frac{(1-\beta)u}{1-[1-(1-\beta)X]u}\frac{(1-X)^{\nu}}{\left( X-\frac{1-\alpha}{1-\beta} \right)^{\nu}}\D X={}&-\frac{\pi}{\sin(\nu\pi)}\left[ 1-\frac{(1-\beta u)^\nu}{(1-\alpha u)^\nu} \right],\label{eq:nu_modulus_change}\\\int_{\frac{1-\alpha}{1-\beta}}^1
\frac{(1-\beta)u}{1-(1-\beta)Xu}\frac{(1-X)^{\nu}}{\left( X-\frac{1-\alpha}{1-\beta} \right)^{\nu}}\D X={}&\frac{\pi}{\sin(\nu\pi)}\left\{ 1-\frac{[1-(1-\beta) u]^\nu}{[1-(1-\alpha )u]^\nu} \right\}.\label{eq:nu_modulus_change'}
\end{align}\item For $ \nu\in(-1,0)$ and  $ 0<\beta<\alpha<1$, we have the following integral identity:
{\allowdisplaybreaks\begin{align}&\frac{\sin(\nu\pi)}{\pi}\int^{1}_{\frac{1-\alpha}{1-\beta}}\frac{\D X}{\mathbb Y_{-\nu-1,1-\beta}(X)}\int_{\frac{1-(1-\beta)X}{\alpha}}^1\frac{\D V}{\mathbb Y_{\nu,\alpha}(V)}\frac{\alpha V}{1-\alpha\mathscr  UV}\notag\\={}&-\int_{0}^1 \frac{\alpha P_{\nu}(2\beta-1)}{1-\alpha\mathscr UV} \frac{V\D V}{\mathbb Y_{\nu,\alpha}(V)}-\int_{0}^1 \frac{\alpha P_{\nu}(1-2\beta)}{1-(1-\alpha\mathscr U)V} \frac{V\D V}{\mathbb Y_{\nu,1-\alpha}(V)}\notag\\{}&+\frac{1}{\mathscr U^{\nu+1}(1-\alpha\mathscr  U)^{\nu+1}}\int^1_{\mathscr U}\frac{u^{\nu}}{(1-\beta u)^{-\nu}}\left( \frac{1-\mathscr U}{1-u} \right)^\nu\frac{\D u}{1-u},\quad \text{where }0<\mathscr U<1/\alpha,\ \label{eq:3rd_kind_add_comb}
\end{align}as well as a  reformulation of Eq.~\ref{eq:rel_intn_add_0}:
\begin{align}&
\frac{\sin(\nu\pi)}{\pi}\int_0^{1}\frac{\D U}{\mathbb Y_{\nu,\alpha}(U)}\int^{1}_{\frac{1-\alpha}{1-\beta}}\frac{\D X}{\mathbb Y_{-\nu-1,1-\beta}(X)}\int_{\frac{1-(1-\beta)X}{\alpha}}^1\frac{\D V}{\mathbb Y_{\nu,\alpha}(V)}\log\frac{1-\alpha UV}{1-V }\notag\\={}&-P_{\nu}(2\beta-1)\iint\log\frac{1-\alpha UW}{1-W}\mathbb D_{\nu,\alpha}U\mathbb D_{\nu,\alpha}W+P_{\nu}(1-2\beta)\iint\log(1-V+\alpha  VW)\mathbb D_{\nu,1-\alpha}V\mathbb D_{\nu,\alpha}W\notag\\{}&-\int_0^1\left\{\int_0^U\left[\int_W^1\frac{\D V}{\mathbb Y_{\nu,\alpha}(V)}\right]\frac{\D W}{\mathbb Y_{-\nu-1,\alpha}(W)}\right\}\frac{\D U}{\mathbb Y_{\nu,\beta}(U)}\notag\\{}&-\frac{\pi P_{\nu}(1-2\alpha)}{\sin(\nu\pi)}\int^{1/\alpha}_0\left[\frac{1}{\mathscr U^{\nu+1}(1-\alpha\mathscr  U)^{\nu+1}}\int^1_{\mathscr U}\frac{u^{\nu}}{(1-\beta u)^{-\nu}}\left( \frac{1-\mathscr U}{1-u} \right)^\nu\frac{\D u}{1-u}\right]\D\mathscr U.\label{eq:rel_intn_self_intn}\end{align}}
 \end{enumerate}\end{proposition}\begin{proof}\begin{enumerate}[label=(\alph*),widest=a]\item

With a variable transformation $ X=\frac{(1-\alpha)t}{1-\beta}+1-t$, we  compute\begin{align}
\int_{\frac{1-\alpha}{1-\beta}}^1
\frac{(1-\beta)u}{1-[1-(1-\beta)X]u}\frac{(1-X)^{\nu}}{\left( X-\frac{1-\alpha}{1-\beta} \right)^{\nu}}\D X={}&\int_0^1\left[1-\frac{1-\alpha u}{(1-\alpha u)t+(1-\beta u)(1-t)}\right]t^\nu(1-t)^{-\nu-1}\D t\notag\\={}&-\frac{\pi}{\sin(\nu\pi)}\left[ 1-\frac{(1-\beta u)^\nu}{(1-\alpha u)^\nu} \right],\label{eq:EFS_app}
\end{align}with the help of the EFS formula (Eq.~\ref{eq:F_ab}).
 This confirms Eq.~\ref{eq:nu_modulus_change}.
The proof of Eq.~\ref{eq:nu_modulus_change'} is similar.
\item On the left-hand sides of Eqs.~\ref{eq:3rd_kind_add_1} and \ref{eq:3rd_kind_add_2}, we perform a variable substitution $ \mathscr U=[1-(1-\beta)X]/\alpha$, before applying
 Eqs.~\ref{eq:nu_modulus_change} and \ref{eq:nu_modulus_change'} to them. These operations lead us to closed-form evaluations of the following two  double integrals:\begin{align}\int^{1}_{\frac{1-\alpha}{1-\beta}}
\frac{\D X}{\mathbb Y_{-\nu-1,1-\beta}(X)}\int_{\frac{1-(1-\beta)X}{\alpha}}^1\frac{\D V}{\mathbb Y_{\nu,\alpha}(V)}={}&\frac{\pi^2}{\sin^2(\nu\pi)}[P_\nu(1-2\alpha)P_\nu(2\beta-1)-P_\nu(2\alpha-1)P_\nu(1-2\beta)],\intertext{and}
\int^{1}_{\frac{1-\alpha}{1-\beta}}
\frac{\D X}{\mathbb Y_{-\nu-1,1-\beta}(X)}\int_{\frac{1-(1-\beta)X}{\alpha}}^1\frac{V\D V}{\mathbb Y_{\nu,\alpha}(V)}={}&-\frac{1}{(\nu+1)\alpha}\frac{\pi}{\sin(\nu\pi)}-\frac{\pi^2}{\sin^2(\nu\pi)}\frac{{_2}F_1\left( \left.\begin{smallmatrix}-\nu-1,\nu+1\\1\end{smallmatrix}\right|1-\alpha \right)}{\alpha}P_\nu(1-2\beta)\notag\\{}&+\frac{\pi^2}{\sin^2(\nu\pi)}\frac{P_{\nu}(1-2\alpha)-{_2}F_1\left( \left.\begin{smallmatrix}-\nu-1,\nu+1\\1\end{smallmatrix}\right|\alpha \right)}{\alpha}P_\nu(2\beta-1).\end{align}These formulae, together with a direct consequence of   Eq.~\ref{eq:nu_modulus_change}: \begin{align}&
\int^{1}_{\frac{1-\alpha}{1-\beta}}
\frac{X^{\nu+1}[1-(1-\beta)X]^{\nu+1}}{\left( X-\frac{1-\alpha}{1-\beta} \right)^\nu}\frac{\D X}{\mathbb Y_{-\nu-1,1-\beta}(X)}\int_0^{\mathscr U}\frac{(1-\beta)u}{1-[1-(1-\beta)X]u}\frac{\D u}{\mathbb Y_{\nu,\alpha}(u)}\notag\\={}&-\frac{\pi}{\sin(\nu\pi)}\left[\int_0^{\mathscr U}\frac{\D u}{\mathbb Y_{\nu,\alpha}(u)}-\int_0^{\mathscr U}\frac{\D u}{\mathbb Y_{\nu,\beta}(u)}\right],
\end{align}further bring us a reduction of Eq.~\ref{eq:3rd_kind_to_2nd_kind}:
 \begin{align}
&\frac{\sin(\nu\pi)}{\pi}\int^{1}_{\frac{1-\alpha}{1-\beta}}\frac{\D X}{\mathbb Y_{-\nu-1,1-\beta}(X)}\int_{\frac{1-(1-\beta)X}{\alpha}}^1\frac{\D V}{\mathbb Y_{\nu,\alpha}(V)}\frac{\alpha V}{1-\alpha\mathscr  UV}\notag\\={}&-\int_{0}^1 \frac{\alpha P_{\nu}(2\beta-1)}{1-\alpha\mathscr UV} \frac{V\D V}{\mathbb Y_{\nu,\alpha}(V)}-\int_{0}^1 \frac{\alpha P_{\nu}(1-2\beta)}{1-(1-\alpha\mathscr U)V} \frac{V\D V}{\mathbb Y_{\nu,1-\alpha}(V)}+\frac{1}{\mathbb Y_{-\nu-1,\alpha}(\mathscr U)}\int^1_{\mathscr U}\frac{\D u}{\mathbb Y_{\nu,\beta}(u)},\label{eq:3rd_kind_add_comb_prep}
\end{align} where we have twice invoked  Legendre's relation in the form of  $
P_{\nu}(2\alpha-1){_2}F_1\left( \left.\begin{smallmatrix}-\nu-1,\nu+1\\1\end{smallmatrix}\right|\alpha \right)+P_{\nu}(1-2\alpha){_2}F_1\left( \left.\begin{smallmatrix}-\nu-1,\nu+1\\1\end{smallmatrix}\right|1-\alpha \right)=P_{\nu}(2\alpha-1)P_{\nu}(1-2\alpha)-\frac{\sin(\nu\pi)}{(\nu+1)\pi}
$. The expression in  Eq.~\ref{eq:3rd_kind_add_comb_prep} agrees with that in Eq.~\ref{eq:3rd_kind_add_comb} (restricted to the range $ 0<\mathscr U<1$), up to an elementary rearrangement of the factors in $ \mathbb Y_{-\nu-1,\alpha}(\mathscr U)\mathbb Y_{\nu,\beta}(u)$.

 Although we have just derived Eq.~\ref{eq:3rd_kind_add_comb} under the restriction that $ \mathscr U\in(0,1)$, its both sides remain equal for  $ \mathscr U\in(0,1/\alpha)$, by analytic continuation. Integrating  Eq.~\ref{eq:3rd_kind_add_comb}  over $ \mathscr U\in(U,1/\alpha)$ for any $U\in(0,1) $, we obtain\begin{align}&\frac{\sin(\nu\pi)}{\pi}\int^{1}_{\frac{1-\alpha}{1-\beta}}\frac{\D X}{\mathbb Y_{-\nu-1,1-\beta}(X)}\int_{\frac{1-(1-\beta)X}{\alpha}}^1\frac{\D V}{\mathbb Y_{\nu,\alpha}(V)}\log\frac{1-\alpha  UV}{1-V}\notag\\={}&-P_{\nu}(2\beta-1)\int_{0}^1  \frac{\log\frac{1-\alpha  UV}{1-V}\D V}{\mathbb Y_{\nu,\alpha}(V)}+ P_{\nu}(1-2\beta)\int_{0}^1  \frac{\log(1-V+\alpha  UV)\D V}{\mathbb Y_{\nu,1-\alpha}(V)}\notag\\{}&+\int^{1/\alpha}_U\left[\frac{1}{\mathscr U^{\nu+1}(1-\alpha\mathscr  U)^{\nu+1}}\int^1_{\mathscr U}\frac{u^{\nu}}{(1-\beta u)^{-\nu}}\left( \frac{1-\mathscr U}{1-u} \right)^\nu\frac{\D u}{1-u}\right]\D\mathscr U.\label{eq:int_U_1_over_alpha}
\end{align}Clearly, the  identity above results in a transformation of the last line in Eq.~\ref{eq:rel_intn_add_0}:{\allowdisplaybreaks\begin{align}&
\frac{\sin(\nu\pi)}{\pi}\int_0^{1}\frac{\D U}{\mathbb Y_{\nu,\alpha}(U)}\int^{1}_{\frac{1-\alpha}{1-\beta}}\frac{\D X}{\mathbb Y_{-\nu-1,1-\beta}(X)}\int_{\frac{1-(1-\beta)X}{\alpha}}^1\frac{\D V}{\mathbb Y_{\nu,\alpha}(V)}\log\frac{1-\alpha UV}{1-V }\notag\\={}&-P_{\nu}(2\beta-1)\int_0^1\int_{0}^1\frac{\D U}{\mathbb Y_{\nu,\alpha}(U)}  \frac{\D W}{\mathbb Y_{\nu,\alpha}(W)}\log\frac{1-\alpha  UW}{1-W}\notag\\{}&+P_{\nu}(1-2\beta)\int_0^1\int_{0}^1  \frac{\D V}{\mathbb Y_{\nu,1-\alpha}(V)}\frac{\D W}{\mathbb Y_{\nu,\alpha}(W)}\log(1-V+\alpha  VW)\notag\\{}&+\int_0^1\left\{\int^{1/\alpha}_U\left[\frac{1}{\mathscr U^{\nu+1}(1-\alpha\mathscr  U)^{\nu+1}}\int^1_{\mathscr U}\frac{u^{\nu}}{(1-\beta u)^{-\nu}}\left( \frac{1-\mathscr U}{1-u} \right)^\nu\frac{\D u}{1-u}\right]\D\mathscr U\right\}\frac{\D U}{\mathbb Y_{\nu,\alpha}(U)}.
\end{align}We now reformulate the last triple integral with integration by parts:\begin{align}
&\int_0^1\left\{\int^{1/\alpha}_U\left[\frac{1}{\mathscr U^{\nu+1}(1-\alpha\mathscr  U)^{\nu+1}}\int^1_{\mathscr U}\frac{u^{\nu}}{(1-\beta u)^{-\nu}}\left( \frac{1-\mathscr U}{1-u} \right)^\nu\frac{\D u}{1-u}\right]\D\mathscr U\right\}\frac{\D U}{\mathbb Y_{\nu,\alpha}(U)}\notag\\={}&-\frac{\pi P_{\nu}(1-2\alpha)}{\sin(\nu\pi)}\int^{1/\alpha}_0\left[\frac{1}{\mathscr U^{\nu+1}(1-\alpha\mathscr  U)^{\nu+1}}\int^1_{\mathscr U}\frac{u^{\nu}}{(1-\beta u)^{-\nu}}\left( \frac{1-\mathscr U}{1-u} \right)^\nu\frac{\D u}{1-u}\right]\D\mathscr U\notag\\{}&-\int_0^1\left[\int_U^1\frac{\D V}{\mathbb Y_{\nu,\alpha}(V)}\right]\left[\int^1_{U}\frac{\D u}{\mathbb Y_{\nu,\beta}(u)}\right]\frac{\D U}{\mathbb Y_{-\nu-1,\alpha}(U)}\notag\\={}&-\frac{\pi P_{\nu}(1-2\alpha)}{\sin(\nu\pi)}\int^{1/\alpha}_0\left[\frac{1}{\mathscr U^{\nu+1}(1-\alpha\mathscr  U)^{\nu+1}}\int^1_{\mathscr U}\frac{u^{\nu}}{(1-\beta u)^{-\nu}}\left( \frac{1-\mathscr U}{1-u} \right)^\nu\frac{\D u}{1-u}\right]\D\mathscr U\notag\\{}&-\int_0^1\left\{\int_0^U\left[\int_W^1\frac{\D V}{\mathbb Y_{\nu,\alpha}(V)}\right]\frac{\D W}{\mathbb Y_{-\nu-1,\alpha}(W)}\right\}\frac{\D U}{\mathbb Y_{\nu,\beta}(U)},\label{eq:trip_int_part0}
\end{align}}which completes the verification of Eq.~\ref{eq:rel_intn_self_intn}.
\qedhere\end{enumerate}\end{proof}Finally, combining the results from Eqs.~\ref{eq:Pnu_sqr_KZ_int2trip_int}, \ref{eq:Pnu_mir_KZ_int2trip_int}, \ref{eq:FZ_a'}, \ref{eq:FZ_b'}, \ref{eq:log_self_sqr}, \ref{eq:log_self_mir}, \ref{eq:Pnu_sqr_intn_bir}, \ref{eq:Pnu_mir_intn_bir}, \ref{eq:rel_intn_add_0} and  \ref{eq:rel_intn_self_intn},
we arrive at Eq.~\ref{eq:precursor_int_id}, thereby  completing  the task stated in   Theorem \ref{thm:KZ_precursor}, at the beginning of this section.
\section{Entropy formulae for automorphic Green's functions and their applications\label{sec:Green_HC}}
\setcounter{equation}{0}\setcounter{theorem}{0}
The purpose of \S\ref{subsec:Ent_HC} is to combine the analysis in \S\ref{sec:Interaction_Entropies_Transformations} into transformations for the interaction entropy (see  Proposition \ref{prop:S_nu_alpha_beta_recip}) \begin{align*}S_{\nu}(\alpha\Vert\beta)={}&[P_{\nu }(2\beta-1)]^2\int_0^{\beta}\frac{[P_\nu(1-2t)]^2\D t}{t-\alpha}-P_\nu(1-2\beta)P_\nu(2\beta-1)\int_0^{\beta}\frac{P_\nu(1-2t)P_\nu(2t-1)\D t}{t-\alpha}\notag\\{}&-P_\nu(2\beta-1)P_\nu(1-2\beta)\int_0^{1-\beta}\frac{P_\nu(1-2t)P_\nu(2t-1)\D t}{t-1+\alpha}+[P_{\nu }(1-2\beta)]^2\int_0^{1-\beta}\frac{[P_\nu(1-2t)]^2\D t}{t-1+\alpha}\end{align*}defined for generic moduli parameters $ \alpha$ and $ \beta$, thereby geometrically interpreting the  Kontsevich--Zagier integral representations for automorphic Green's functions
(Eqs.~\ref{eq:G2_z_z'_arb1} and \ref{eq:G2Hecke234_Pnu1}) as entropy couplings of Legendre--Ramanujan  curves $ Y^n=(1-X)^{n-1}X(1-\alpha X)$ ($ n\in\{6,4,3,2\}$).

By ``entropy coupling'', we are referring to  certain types of double integrals that generalize abelian integrals on algebraic curves, as  defined below.\begin{definition}[Entropy Couplings $ H_\nu(\alpha\Vert\beta)$ and $ h_\nu(\alpha\Vert\beta)$]\label{defn:HC}For $ \nu\in\{-1/6,-1/4,-1/3,-1/2\}$ and $ \alpha,\beta\in(\mathbb C\smallsetminus\mathbb R)\cup(0,1)$, we define entropy couplings {\allowdisplaybreaks\begin{align}
H_\nu(\alpha\Vert\beta):={}&\frac{1}{2}\left\{\left( \mathbb E_{\nu,\alpha\vphantom{\beta}}^{U} \mathbb E_{\nu,\beta}^{V}+\mathbb E_{-\nu-1,\alpha\vphantom{\beta}}^{U} \mathbb E_{-\nu-1,\beta}^{V}\right)\log\frac{1-\beta UV}{1-\beta V}\right.\notag\\{}&\left.-\left( \mathbb E_{-\nu-1,1-\alpha\vphantom{\beta}}^{U} \mathbb E_{\nu,\beta}^{V}+\mathbb E_{\nu,1-\alpha\vphantom{\beta}}^{U} \mathbb E_{-\nu-1,\beta}^{V} \right)\log\left[ 1-\frac{1-\alpha}{\alpha}\frac{\beta}{1-\beta}(1-U)(1-V)\right]\right\},\label{eq:H_nu_alpha_beta_defn}
\end{align}}under the constraints $\alpha,1-\alpha,\beta,1-\beta,\frac{\beta(1-\alpha)}{\alpha(1-\beta)}\in\mathbb C\smallsetminus[1,+\infty)$;  we define $ H_\nu(\alpha\Vert\beta)$ for distinct moduli parameters $ \alpha,\beta$ not fulfilling the foregoing constraints as an analytic continuation of Eq.~\ref{eq:H_nu_alpha_beta_defn}.
Here, the Legendre expectations $ \mathbb E_{\nu,\alpha}^{u}$ are defined in Eq.~\ref{eq:Legendre_expectation_defn}, whose integration paths are always  straight-line segments.

The ``regular part'' of  $ H_\nu(\alpha\Vert\beta)$, or the regularized entropy coupling $ h_\nu(\alpha\Vert\beta)$, is defined as {\allowdisplaybreaks\begin{align}
h_\nu(\alpha\Vert\beta):={}&\frac{1}{2}\left[\left( \mathbb E_{\nu,\alpha\vphantom{\beta}}^{U} \mathbb E_{\nu,\beta}^{V}+\mathbb E_{-\nu-1,\alpha\vphantom{\beta}}^{U} \mathbb E_{-\nu-1,\beta}^{V}\right)\log(1-\alpha UV)\right.\notag\\{}&\left.-\left( \mathbb E_{\nu,1-\alpha\vphantom{\beta}}^{U} \mathbb E_{\nu,\beta}^{V}+\mathbb E_{-\nu-1,1-\alpha\vphantom{\beta}}^{U} \mathbb E_{-\nu-1,\beta}^{V} \right)\log\left( 1+\frac{\alpha UV}{1-U} \right)\right].\label{eq:reg_h_nu_alpha_beta_defn}
\end{align}}Such a definition of $ h_\nu(\alpha\Vert\beta)$ applies to  $ \nu\in\{-1/6,-1/4,-1/3,-1/2\}$ and $ \alpha,\beta\in(\mathbb C\smallsetminus\mathbb R)\cup(0,1)$.\eor\end{definition}
\begin{remark}
When  $ \nu\in\{-1/6,-1/4,-1/3,-1/2\}$ and $ \alpha,\beta\in(\mathbb C\smallsetminus\mathbb R)\cup(0,1)$, we have $P_{\nu}(1-2\alpha) P_{\nu}(1-2\beta)\neq0  $ and $P_{\nu}(2\alpha-1) P_{\nu}(1-2\beta) \neq0$  according to Ramanujan's elliptic function theory to alternative bases \cite[][Eq.~2.1.8]{AGF_PartI}. Thus, the Legendre expectations in Eqs.~\ref{eq:H_nu_alpha_beta_defn} and \ref{eq:reg_h_nu_alpha_beta_defn} are indeed well-defined.\eor\end{remark}\begin{remark}Owing to the extra symmetry (Lemma \ref{lm:nu_reflection_symm})\begin{align}
\left( \mathbb E_{\nu,\alpha\vphantom{\beta}}^{U} \mathbb E_{\nu,\beta}^{V}-\mathbb E_{-\nu-1,\alpha\vphantom{\beta}}^{U} \mathbb E_{-\nu-1,\beta}^{V}\right)\log\frac{1-\beta UV}{1-\beta V}={}&0,\\\left( \mathbb E_{-\nu-1,1-\alpha\vphantom{\beta}}^{U} \mathbb E_{\nu,\beta}^{V}-\mathbb E_{\nu,1-\alpha\vphantom{\beta}}^{U} \mathbb E_{-\nu-1,\beta}^{V} \right)\log\left[ 1-\frac{1-\alpha}{\alpha}\frac{\beta}{1-\beta}(1-U)(1-V)\right]={}&0,
\end{align}the definition of the entropy coupling $ H_\nu(\alpha\Vert\beta)$ can be actually reduced to half as many operations of Legendre expectations. (Therefore, the formulations of $ H_\nu(\alpha\Vert\beta)$ in Eqs.~\ref{eq:HC_defn_intro} and \ref{eq:H_nu_alpha_beta_defn} are in fact  equivalent.) Such a symmetric reduction, however, does not apply to the expression for $ h_\nu(\alpha\Vert\beta)$. \eor\end{remark}\begin{remark}The letter $H$ chosen for $ H_\nu(\alpha\Vert\beta)$ represents both entropy (as in information theory) and height (as in arithmetic geometry). For $ 0<\beta<\alpha<1$, one may  verify the following identity (cf.~Eq.~\ref{eq:H2_double_refl}):\begin{align}
H_\nu(\alpha\Vert\beta)=\frac{\displaystyle\int_0^1\frac{\mathbb E_{\nu,\beta}^{V}\log(1-\beta UV)\D U}{U^{-\nu}(1-U)^{\nu+1}(1-\alpha U)^{-\nu}}}{\displaystyle\int_0^{1}\frac{ \D U}{U^{-\nu}(1-U)^{\nu+1}(1-\alpha U)^{-\nu}}}-\frac{\displaystyle\int_1^{1/\alpha}\frac{\mathbb E_{\nu,\beta}^{V}\log(1-\beta UV)\D U}{U^{-\nu}(U-1)^{\nu+1}(1-\alpha U)^{-\nu}}}{\displaystyle\int_1^{1/\alpha}\frac{ \D U}{U^{-\nu}(U-1)^{\nu+1}(1-\alpha U)^{-\nu}}},
\end{align}which can be converted into \begin{align}
H_\nu(\alpha\Vert\beta)={}&\frac{\displaystyle\int_0^1\frac{\D U}{U^{-\nu}(1-U)^{\nu+1}(1-\alpha U)^{-\nu}}\int_{1/\beta}^{\infty}\frac{ \D\mathscr  V}{\mathscr V^{\nu+1}(\mathscr V-1)^{-\nu}(\beta\mathscr  V-1)^{\nu+1}}\log(\mathscr V-U)}{\displaystyle\int_0^{1}\frac{ \D U}{U^{-\nu}(1-U)^{\nu+1}(1-\alpha U)^{-\nu}}\int_{1/\beta}^{\infty}\frac{ \D\mathscr  V}{\mathscr V^{\nu+1}(\mathscr V-1)^{-\nu}(\beta\mathscr  V-1)^{\nu+1}}}\notag\\{}&-\frac{\displaystyle\int_1^{1/\alpha}\frac{\D U}{U^{-\nu}(U-1)^{\nu+1}(1-\alpha U)^{-\nu}}\int_{1/\beta}^{\infty}\frac{ \D\mathscr  V}{\mathscr V^{\nu+1}(\mathscr V-1)^{-\nu}(\beta\mathscr  V-1)^{\nu+1}}\log(\mathscr V-U)}{\displaystyle\int_1^{1/\alpha}\frac{ \D U}{U^{-\nu}(U-1)^{\nu+1}(1-\alpha U)^{-\nu}}\int_{1/\beta}^{\infty}\frac{ \D\mathscr  V}{\mathscr V^{\nu+1}(\mathscr V-1)^{-\nu}(\beta\mathscr  V-1)^{\nu+1}}}
\end{align}by $ V=1/(\beta\mathscr V)$. In particular, for $ 0<\mu<\lambda<1$, we have {\allowdisplaybreaks\begin{align}
H_{-1/2}(\lambda\Vert\mu)={}&\int_{0}^{\mathbf K(\sqrt{\lambda})}\frac{\D u}{\mathbf K(\sqrt{\lambda})}\int_{0}^{ \mathbf K(\sqrt{\vphantom1\smash[b]{\mu }})}\frac{\D v}{\mathbf K(\sqrt{\vphantom1\smash[b]{\mu }})}\,\log[1-\mu\sn^2(u|\lambda)\sn^2(v|\mu)]\notag\\{}&-\int_{\mathbf K(\sqrt{\lambda})}^{\mathbf K(\sqrt{\lambda})+i\mathbf K(\sqrt{1-\lambda})}\frac{\D u}{i\mathbf K(\sqrt{1-\lambda})}\int_{0}^{ \mathbf K(\sqrt{\vphantom1\smash[b]{\mu }})}\frac{\D v}{\mathbf K(\sqrt{\vphantom1\smash[b]{\mu }})}\,\log[1-\mu\sn^2(u|\lambda)\sn^2(v|\mu)]\notag\\={}&\int_{0}^{\mathbf K(\sqrt{\lambda})}\frac{\D u}{\mathbf K(\sqrt{\lambda})}\int_{i\mathbf K(\sqrt{\vphantom1\smash[b]{1-\mu }})}^{ \mathbf K(\sqrt{\vphantom1\smash[b]{\mu }})+i\mathbf K(\sqrt{\vphantom1\smash[b]{1-\mu }})}\frac{\D v}{\mathbf K(\sqrt{\vphantom1\smash[b]{\mu }})}\,\log[\sn^2(v|\mu)-\sn^2(u|\lambda)]\notag\\{}&-\int_{\mathbf K(\sqrt{\lambda})}^{\mathbf K(\sqrt{\lambda})+i\mathbf K(\sqrt{1-\lambda})}\frac{\D u}{i\mathbf K(\sqrt{1-\lambda})}\int_{i\mathbf K(\sqrt{\vphantom1\smash[b]{1-\mu }})}^{ \mathbf K(\sqrt{\vphantom1\smash[b]{\mu }})+i\mathbf K(\sqrt{\vphantom1\smash[b]{1-\mu }})}\frac{\D v}{\mathbf K(\sqrt{\vphantom1\smash[b]{\mu }})}\,\log[\sn^2(v|\mu)-\sn^2(u|\lambda)],
\end{align}}which measures the ``expected'' value of a ``height function'' $\log(X_2-X_1) $ on certain cycles of two coupled elliptic curves $ E_\lambda(\mathbb C):Y_1^2=X_1(1-X_1)(1-\lambda X_1)$ and $ E_\mu(\mathbb C):Y_2^2=X_2(1-X_2)(1-\mu X_2)$. In a sequel to the current work, we shall see that the entropy coupling can be regarded as localizations of the global height pairings in the work of Gross--Zagier \cite{GrossZagierI}, Gross--Kohnen--Zagier \cite{GrossZagierII} and Zhang \cite{SWZhang1997}.   \eor\end{remark}

The connections from  the entropy coupling $ H_\nu(\alpha\Vert\beta)$ to Eq.~\ref{eq:KZ_int_precursor} will be eventually revealed in Proposition \ref{prop:HC_int_Ent}. In short, the interaction entropy in Legendre--Ramanujan form (defined in  Proposition \ref{prop:S_nu_alpha_beta_recip}) can be recast into the  entropy-coupling form (spelt out in Proposition \ref{prop:HC_int_Ent}). We will see that\begin{align}
S_{\nu}(\alpha\Vert\beta)\equiv{}& P_\nu(1-2\alpha)P_\nu(2\alpha-1)P_\nu(1-2\beta)P_\nu(2\beta-1)[H_\nu(\alpha\Vert\beta)+H_\nu(1-\alpha\Vert1-\beta)]
\notag\\&\pmod{2\pi i\Lambda_\nu(\alpha,\beta)}\end{align}follows from Eq.~\ref{eq:S_nu_H_nu} (with $ \Lambda_\nu(\alpha,\beta)$ being defined in Eq.~\ref{eq:Lambda_nu_alpha_beta_defn}), which completes the proof of Theorem \ref{thm:HC_wt4_Green}.

In \S\ref{subsec:lim_ent}, we will analyze    $ h_\nu(\alpha\Vert\alpha)$, $ H_{-1/2}(\lambda\Vert1/(1-\lambda))$ and $ H_{-1/2}(1-\lambda\Vert\lambda/(\lambda-1))$, where the moduli parameters $\alpha$ and $\lambda$ tend to certain extreme values. The asymptotic properties of these special entropy formulae will lead us to a proof of Theorem \ref{thm:app_HC} in  \S\ref{subsec:G2_Hecke4_GZ_renorm}.

\subsection{Interaction entropies and Kontsevich--Zagier integrals for weight-4 automorphic Green's functions\label{subsec:Ent_HC}}
In \S\ref{subsec:Legendre_add_intn_entropy}, the constraint $ 0<\beta<\alpha<1$ relieved us of  concerns over the branch cuts of the fractional powers and the logarithms in the integrands.  In the next lemma, we will turn the last two multiple integrals on the right-hand side of Eq.~\ref{eq:rel_intn_self_intn} into definite integrals on the unit square $ (0,1)\times(0,1)\subset\mathbb R^2$,  paving way for their analytic continuation to generic values of $ \alpha$ and $ \beta$.

\begin{lemma}[Some Addition Formulae for Multiple Integrals]\begin{enumerate}[label=\emph{(\alph*)}, ref=(\alph*), widest=a] \item For $ -1<\nu<0,0<\beta<\alpha<1$, we have \begin{align}&
-\int_0^1\left\{\int_0^U\left[\int_W^1\frac{\D V}{\mathbb Y_{\nu,\alpha}(V)}\right]\frac{\D W}{\mathbb Y_{-\nu-1,\alpha}(W)}\right\}\frac{\D U}{\mathbb Y_{\nu,\beta}(U)}\notag\\={}& P_{\nu}(2\alpha-1)\iint\log(1-\alpha UV)\mathbb D_{\nu,\alpha}U \mathbb D_{\nu,\beta} V-P_{\nu}(1-2\alpha)\iint\log\left( 1+\frac{\alpha VW}{1-V} \right)\mathbb D_{\nu,1-\alpha}V \mathbb D_{\nu,\beta} W.\label{eq:trip_indef_to_double_log_def}
\end{align}\item For $ -1<\nu<0,0<\beta<\alpha<1$, the following identity holds:\begin{align}&\frac{\pi}{\sin(\nu\pi)}
\int^{1/\alpha}_0\left[\frac{1}{\mathscr U^{\nu+1}(1-\alpha\mathscr  U)^{\nu+1}}\int^1_{\mathscr U}\frac{u^{\nu}}{(1-\beta u)^{-\nu}}\left( \frac{1-\mathscr U}{1-u} \right)^\nu\frac{\D u}{1-u}\right]\D\mathscr U\notag\\={}&\iint\log(1-V)\mathbb D_{\nu,1-\alpha}U \mathbb D_{\nu,\beta} V-\iint\log\left( 1+\frac{\alpha VW}{1-V} \right)\mathbb D_{\nu,1-\alpha}V \mathbb D_{\nu,\beta} W\notag\\{}&+\iint\log\left[ 1-\frac{(1-\alpha)U(1-\beta V)}{1-\beta} \right]\mathbb D_{\nu,1-\alpha}U \mathbb D_{\nu,\beta} V.\label{eq:double_indef_to_double_log_def}
\end{align}\item For  $ -1<\nu<0,0<\beta<\alpha<1$, we have {\allowdisplaybreaks\begin{align}&-\int_0^1\left\{\int_0^U\left[\int_W^1\frac{\D V}{\mathbb Y_{\nu,\alpha}(V)}\right]\frac{\D W}{\mathbb Y_{-\nu-1,\alpha}(W)}\right\}\frac{\D U}{\mathbb Y_{\nu,\beta}(U)}\notag\\{}&-\frac{\pi P_{\nu}(1-2\alpha)}{\sin(\nu\pi)}\int^{1/\alpha}_0\left[\frac{1}{\mathscr U^{\nu+1}(1-\alpha\mathscr  U)^{\nu+1}}\int^1_{\mathscr U}\frac{u^{\nu}}{(1-\beta u)^{-\nu}}\left( \frac{1-\mathscr U}{1-u} \right)^\nu\frac{\D u}{1-u}\right]\D\mathscr U\notag\\={}& \frac{\pi P_{\nu}(1-2\alpha)}{\sin(\nu\pi)}\left[P_{\nu}(2\alpha-1)\int\log(1-V) \mathbb D_{\nu,\beta} V-P_{\nu}(1-2\beta)\int\log\frac{1-(1-\alpha )U}{\alpha}\mathbb D_{\nu,1-\alpha}U\right]\notag\\{}&+P_{\nu}(2\alpha-1)\iint\log(1-\alpha UV)\mathbb D_{\nu,\alpha}U \mathbb D_{\nu,\beta} V\notag\\{}&- P_{\nu}(1-2\alpha)\iint\log\left[ 1-\frac{1-\alpha}{\alpha}\frac{\beta}{1-\beta}(1-U)(1-V)\right]\mathbb D_{-\nu-1,1-\alpha}U \mathbb D_{\nu,\beta} V.\label{eq:rel_ent_princip}\end{align}}Furthermore, the last double integral satisfies\begin{align}
{}&\iint\log\left[ 1-\frac{1-\alpha}{\alpha}\frac{\beta}{1-\beta}(1-U)(1-V)\right]\mathbb D_{-\nu-1,1-\alpha}U \mathbb D_{\nu,\beta} V\notag\\={}&\int_1^{1/\alpha}\left(\int\log  \frac{1-\beta\mathscr U\mathscr V}{1-\beta\mathscr V} \mathbb D_{-\nu-1,\beta}\mathscr  V\right)\frac{\D \mathscr U}{\mathscr U^{\nu+1}(\mathscr U-1)^{-\nu}(1-\alpha \mathscr U)^{\nu+1}}.\label{eq:H2_double_refl}
\end{align}\end{enumerate}\end{lemma}\begin{proof}\begin{enumerate}[label=(\alph*),widest=a]\item
A direct  integration of Eq.~\ref{eq:3rd_kind_add_2} gives rise to \begin{align}&
 \int_0^U\left[\int_W^1\frac{\D V}{\mathbb Y_{\nu,\alpha}(V)}\right]\frac{\D W}{\mathbb Y_{-\nu-1,\alpha}(W)}\notag\\={}& -P_{\nu}(2\alpha-1)\int_{0}^1 \frac{\log(1-\alpha UV)\D V}{\mathbb Y_{\nu,\alpha}(V)}+ P_{\nu}(1-2\alpha)\int_{0}^1  \frac{\log\left( 1+\frac{\alpha UV}{1-V} \right)\D V}{\mathbb Y_{\nu,1-\alpha}(V)},\label{eq:log_int_add_F_sqr}
\end{align} which reduces the  triple integral on the left-hand side of Eq.~\ref{eq:trip_indef_to_double_log_def} into a double integral over the unit square $0<U<1,0<V<1 $. We may subsequently recast this into the form stated on the right-hand side of  Eq.~\ref{eq:trip_indef_to_double_log_def}.\item We start by splitting the integral over $\mathscr U\in(0,1/\alpha) $  into two parts:\begin{align}&
\int^{1/\alpha}_0\left[\frac{1}{\mathscr U^{\nu+1}(1-\alpha\mathscr  U)^{\nu+1}}\int^1_{\mathscr U}\frac{u^{\nu}}{(1-\beta u)^{-\nu}}\left( \frac{1-\mathscr U}{1-u} \right)^\nu\frac{\D u}{1-u}\right]\D\mathscr U\notag\\={}&\int^{1}_0\left[\frac{(1-\mathscr U)^\nu}{\mathscr U^{\nu+1}(1-\alpha\mathscr  U)^{\nu+1}}\int^1_{\mathscr U}\frac{u^{\nu}(1-u)^{-\nu-1}}{(1-\beta u)^{-\nu}}\D u\right]\D\mathscr U\notag\\{}&+\int_1^{1/\alpha}\left[\frac{(\mathscr U-1)^\nu}{\mathscr U^{\nu+1}(1-\alpha\mathscr  U)^{\nu+1}}\int_1^{\mathscr U}\frac{u^{\nu}(u-1)^{-\nu-1}}{(1-\beta u)^{-\nu}}\D u\right]\D\mathscr U.
\end{align}

On the right-hand side of the equation above, we rewrite the first double integral involving   $ \mathscr U\in(0,1)$ as \begin{align}
\int_0^1\frac{\D \mathscr U}{\mathbb Y_{-\nu-1,\alpha}(\mathscr U)}\int^1_{\mathscr U}\frac{\D u}{\mathbb Y_{\nu,\beta}(u)}=\int^1_{0}\frac{\D u}{\mathbb Y_{\nu,\beta}(u)}\int_0^u\frac{\D \mathscr U}{\mathbb Y_{-\nu-1,\alpha}(\mathscr U)},
\end{align}where the remaining integration over $\mathscr  U\in(0,u)$ can be  reformulated as\begin{align}
\int_0^u\frac{\D \mathscr U}{\mathbb Y_{-\nu-1,\alpha}(\mathscr U)}=&\frac{\sin(\nu \pi)}{\pi}\int\left\{\log[1-u+(1-\alpha)u V]-\log\frac{1-V+\alpha uV}{1-V}\right\}\mathbb D_{\nu,1-\alpha}V,\label{eq:gen_incomp_ell_int_3rd_add_a}
\end{align} which we will explain in the  paragraph below.

We point out that the following elementary identity \begin{align}&
\frac{\alpha u}{\mathbb{Y}_{ \nu,\alpha}(u)}\left[-\left( 1-\frac{V }{u} \right)\frac{\nu+1}{V^{-\nu}(1-\alpha V)^{-\nu}(V-1)^{\nu+1}}-
\frac{\partial }{\partial V }\frac{V^{\nu+1}(1-\alpha V)^{\nu+1}(V-1)^{-\nu}}{1-\alpha uV}\right]\notag\\={}& \frac{\partial}{\partial u}\left[ \frac{\alpha V}{1-\alpha uV} \frac{\mathbb{Y}_{- \nu-1,\alpha}(u)}{V^{-\nu}(1-\alpha V)^{-\nu}(V-1)^{\nu+1}}\right]\label{eq:incomp_3rd_kind_nu_diff'}\tag{\ref{eq:incomp_3rd_kind_nu_diff}$'$}
\end{align} for $ \alpha,u\in(0,1)$, $ \nu\in(-1,0)$ and $ V\in(1,1/\alpha)$ can be regarded as an analytic continuation of Eq.~\ref{eq:incomp_3rd_kind_nu_diff}. Integrating Eq.~\ref{eq:incomp_3rd_kind_nu_diff'} over $ V=t+\frac{1-t}{\alpha}\in(1,1/\alpha)$ and $ u\in(0,\mathscr U)$, we arrive at an analog of Eq.~\ref{eq:3rd_kind_breakdown_b}: \begin{align}
&\frac{\nu+1}{ \mathbb{Y}_{- \nu-1,\alpha}(\mathscr U)}P_{\nu}(2\alpha-1)\int_0^{\mathscr U}\frac{u\D u}{\mathbb Y_{\nu,\alpha}(u)}-\frac{\nu+1}{ \alpha\mathbb{Y}_{- \nu-1,\alpha}(\mathscr U)}{_2}F_1\left( \left.\begin{array}{c}
-\nu-1,\nu+1 \\
1
\end{array}\right|1-\alpha \right)\int_0^{\mathscr U}\frac{\D u}{\mathbb Y_{\nu,\alpha}(u)}\notag\\={}&\frac{\sin(\nu \pi)}{\alpha\pi}\int_{0}^1 \frac{1-(1-\alpha)t}{1-\mathscr U+(1-\alpha)\mathscr Ut} \frac{\D t}{\mathbb Y_{\nu,1-\alpha}(t)}.\label{eq:3rd_kind_breakdown_b_ImT}
\end{align}Adding up  Eqs.~\ref{eq:3rd_kind_breakdown_b} and  \ref{eq:3rd_kind_breakdown_b_ImT}, we are able to establish an addition formula:  \begin{align}
-\frac{\sin(\nu \pi)}{\pi}\int_{0}^1 \left[\frac{\alpha V}{1-(1-\alpha\mathscr U)V}+\frac{1-(1-\alpha)V}{1-\mathscr U+(1-\alpha)\mathscr UV}\right] \frac{\D V}{\mathbb Y_{\nu,1-\alpha}(V)}=\frac{1}{\mathbb{Y}_{- \nu-1,\alpha}(\mathscr U)},\label{eq:gen_incomp_ell_int_3rd_add_a_prep}
\end{align} after an invocation of Legendre's relation   \cite[][Theorem 3.2.8]{AAR} in the following manner:\begin{align}&
(\nu+1)\left[\alpha P_{\nu}(2\alpha-1)\int_0^{1}\frac{u\D u}{\mathbb Y_{\nu,\alpha}(u)}-{_2}F_1\left( \left.\begin{array}{c}
-\nu-1,\nu+1 \\
1
\end{array}\right|1-\alpha \right)\int_0^{1}\frac{\D u}{\mathbb Y_{\nu,\alpha}(u)}\right]\notag\\={}&-\frac{(\nu+1)\pi}{\sin(\nu\pi)}\left[ P_\nu(2\alpha-1)P_\nu(1-2\alpha)-P_\nu(2\alpha-1) {_2}F_1\left( \left.\begin{array}{c}
-\nu-1,\nu+1 \\
1
\end{array}\right|\alpha \right)\right.\notag\\{}&\left.-{_2}F_1\left( \left.\begin{array}{c}
-\nu-1,\nu+1 \\
1
\end{array}\right|1-\alpha \right)P_\nu(1-2\alpha)\right]=-1.
\end{align} Integrating Eq.~\ref{eq:gen_incomp_ell_int_3rd_add_a_prep} over $ \mathscr U\in(0,u)$, we arrive at Eq.~\ref{eq:gen_incomp_ell_int_3rd_add_a}.

Now we  turn  our attention to the proof of \begin{align}&
\int_1^{1/\alpha}\left[\frac{(\mathscr U-1)^\nu}{\mathscr U^{\nu+1}(1-\alpha\mathscr  U)^{\nu+1}}\int_1^{\mathscr U}\frac{u^{\nu}(u-1)^{-\nu-1}}{(1-\beta u)^{-\nu}}\D u\right]\D\mathscr U\notag\\={}&-\frac{\sin(\nu \pi)}{\pi}\int_0^1\frac{\D U}{\mathbb Y_{\nu,1-\alpha}(U)}\int_{0}^1  \frac{\D V}{\mathbb Y_{\nu,\beta}(V)}\log\frac{1-V+(1-\alpha)UV}{1-V}\notag\\{}&+\frac{\sin(\nu \pi)}{\pi}\int_0^1\frac{\D U}{\mathbb Y_{\nu,1-\alpha}(U)}\int_{0}^1  \frac{\D V}{\mathbb Y_{\nu,\beta}(V)}\log\left[ 1-\frac{(1-\alpha)U(1-\beta V)}{1-\beta} \right].\label{eq:1_to_1_over_alpha_add_form_log}
\end{align} To demonstrate the identity above, we compute\begin{align}&
\int_1^{\mathscr U}\frac{u^{\nu}(u-1)^{-\nu-1}}{(1-\beta u)^{-\nu}}\D u\overset{u=\frac{1}{1-(1-\beta)t}}{=\!\!=\!\!=\!\!=\!\!=
\!\!=\!\!=}\int^{\frac{\mathscr U-1}{(1-\beta)\mathscr U}}_0\frac{\D t}{\mathbb Y_{-\nu-1,1-\beta}(t)}\notag\\={}&-\frac{\sin(\nu \pi)}{\pi}\int_{0}^1  \frac{\D V}{\mathbb Y_{\nu,\beta}(V)}\log\frac{1-V+\frac{\mathscr U-1}{\mathscr U}V}{1-V}\notag\\{}&+\frac{\sin(\nu \pi)}{\pi}\int_{0}^1  \frac{\D V}{\mathbb Y_{\nu,\beta}(V)}\log\left[1-\frac{(\mathscr U-1)(1-\beta V)}{(1-\beta)\mathscr U}\right],
\end{align} with the help of Eq.~\ref{eq:gen_incomp_ell_int_3rd_add_a}, before noting the fact that\begin{align}
\int_1^{1/\alpha}\frac{(\mathscr U-1)^\nu f( \mathscr U) }{\mathscr U^{\nu+1}(1-\alpha\mathscr  U)^{\nu+1}}\D\mathscr U=\int_0^1\frac{f\left( \frac{1}{1-(1-\alpha)U} \right)\D U}{\mathbb Y_{\nu,1-\alpha}(U)}
\end{align}holds for any suitably regular function $f(\mathscr U),1<\mathscr U<1/\alpha$.

Clearly, the efforts above can be combined into Eq.~\ref{eq:double_indef_to_double_log_def}.\item The equality in  Eq.~\ref{eq:rel_ent_princip} descends from Eqs.~\ref{eq:trip_indef_to_double_log_def}--\ref{eq:double_indef_to_double_log_def} and an application of the following Jacobi involution\begin{align}
\int f(u)\mathbb D_{\nu,1-\alpha}u:=\int_0^1\frac{f(u)\D u}{\mathbb Y_{\nu,1-\alpha}(u)}=\int_0^1\frac{f\left( \frac{1-U}{1-(1-\alpha)U} \right)\D U}{\mathbb Y_{-\nu-1,1-\alpha}(U)}=:\int f\left( \frac{1-U}{1-(1-\alpha)U} \right)\mathbb D_{-\nu-1,1-\alpha}U\label{eq:Jac_inv_int}
\end{align}to the last integral in Eq.~\ref{eq:double_indef_to_double_log_def}.

The equality in Eq.~\ref{eq:H2_double_refl} originates from the birational variable substitutions $U=\frac{1-\alpha\mathscr U}{1-\alpha},V=\frac{1-\mathscr V}{1-\beta\mathscr V} $.
\qedhere\end{enumerate}
\end{proof}

Combining  Eqs.~\ref{eq:precursor_int_id} and  \ref{eq:rel_ent_princip},  we see that   the following formula holds for  $ 0<\beta<\alpha<1$:{\allowdisplaybreaks\begin{align}
&
\frac{2\pi^2}{\sin^2(\nu\pi)}\left\{[P_{\nu }(2\beta-1)]^2\int_0^{\beta}\frac{[P_\nu(1-2t)]^2\D t}{t-\alpha}-P_\nu(1-2\beta)P_\nu(2\beta-1)\int_0^{\beta}\frac{P_\nu(1-2t)P_\nu(2t-1)\D t}{t-\alpha}\right\}\notag\\={}&\frac{\pi^3[P_\nu(2\alpha-1)]^{2}P_\nu(1-2\beta)P_\nu(2\beta-1)}{\sin^3(\nu\pi)}-\frac{\pi^3P_\nu(1-2\alpha)P_\nu(2\alpha-1)[P_\nu(2\beta-1)]^{2}}{\sin^3(\nu\pi)}\notag\\{}&+\frac{\pi^{2}P_\nu(1-2\alpha)P_\nu(2\alpha-1)P_\nu(1-2\beta)P_\nu(2\beta-1)}{\sin^2(\nu\pi)}\left(\log\frac{\beta}{1-\beta}-\log\frac{\alpha}{1-\alpha}\right)\notag\\{}&+P_\nu(2\beta-1)\left\{P_{\nu}(2\alpha-1)\iint\log\frac{1-\alpha UV}{1-\alpha U}(\mathbb D_{\nu,\alpha}U\mathbb D_{\nu,\beta}V+\mathbb D_{-\nu-1,\alpha}U\mathbb D_{-\nu-1,\beta}V)\right.\notag\\{}&\left.- P_{\nu}(1-2\alpha)\iint\log\left[ 1-\frac{\beta(1-\alpha)}{\alpha(1-\beta)}(1-U)(1- V)\right](\mathbb D_{-\nu-1,1-\alpha}U\mathbb D_{\nu,\beta}V+\mathbb D_{\nu,1-\alpha}U\mathbb D_{-\nu-1,\beta}V)\right\},\label{eq:intn_ent_log_form'}
\end{align}}with the understanding that   the integration paths on both the left- and right-hand sides are straight-line segments.  In order to relate the right-hand side of Eq.~\ref{eq:intn_ent_log_form'} to the entropy coupling $ H_\nu(\alpha\Vert\beta)$ (Eq.~\ref{eq:H_nu_alpha_beta_defn}), we need  further simplifications.

The following self-explanatory integral identity (cf.~\cite[][Proposition 4.3]{Zhou2013Pnu} and \cite[][Lemma 3.3.1]{AGF_PartI}) \begin{align}&
\int_0^1\int_0^1\frac{f(U,V)\D U\D V}{\sqrt{U(1-U)(1-\lambda U)V(1-V)(1-\lambda V)}}\notag\\={}&\int_0^1\int_0^1\frac{f\left(\frac{1-\mathscr U}{1-\lambda  \mathscr U \mathscr V},\frac{(1-\mathscr V) (1-\lambda  \mathscr U \mathscr V)}{1-\lambda\mathscr V  }\right)\D \mathscr U\D \mathscr V}{\sqrt{\mathscr U(1-\mathscr U)\mathscr V(1-\mathscr V)[1-\lambda+ \lambda^{2}\mathscr U\mathscr V(1-\mathscr V)]}},\quad 0<\lambda<1\label{eq:Ramanujan_rotation}
\end{align}  is true for any sufficiently regular bivariate function $ f(U,V),0<U<1,0<V<1$. We (re)discovered this formula in an attempt to reconstruct Ramanujan's thoughts behind Entry 7(x) in Chapter 17 of his second notebook \cite[][pp.~110--111]{RN3}. We have referred to this trick as ``Ramanujan's rotations''
due to the spherical rotations that were originally employed to introduce these variable substitutions. From a different perspective, the birational maps in Eq.~\ref{eq:Ramanujan_rotation} can be broken down into two successive Jacobi involutions (cf.~Eq.~\ref{eq:Jacobi_involution_U}): $
V=\frac{1-\mathscr V}{1-\lambda   U \mathscr V}
$ and $ U=\frac{1-\mathscr U}{1-\lambda  \mathscr U \mathscr V}$.

As an  application of Ramanujan's rotations (Eq.~\ref{eq:Ramanujan_rotation}), we demonstrate a reciprocity relation
in the next lemma.\begin{lemma}[Ramanujan Reciprocity]\label{lm:Ramanujan_recip} For  $ \alpha,\beta\in(\mathbb C\smallsetminus\mathbb R)\cup(0,1)$ and $ \nu\in(-1,0)$, we have \begin{align}
\iint\log\frac{1-\alpha UV}{\frac{1-\alpha U}{1-U}}\mathbb D_{\nu,\alpha}U\mathbb D_{\nu,\beta}V=\iint \log\frac{1-\beta UV}{\frac{1-\beta V}{1-V}}\mathbb{D}_{\nu,\alpha}U\mathbb D_{\nu,\beta}V.\label{eq:Ramanujan_recip}
\end{align} \end{lemma}\begin{proof}In what follows, we prove Eq.~\ref{eq:Ramanujan_recip} for $ \alpha,\beta\in(0,1)$ and $\nu\in(-1,0)$, using Ramanujan's rotations.
 The generic cases for  $ \alpha,\beta\in(\mathbb C\smallsetminus\mathbb R)\cup(0,1)$ and $ \nu\in(-1,0)$ would then arise from analytic continuation.

By Ramanujan's rotations $U= \frac{1-\mathscr U}{1-\alpha  \mathscr U \mathscr V},V=\frac{(1-\mathscr V) (1-\alpha  \mathscr U \mathscr V)}{1-\alpha\mathscr V  }$, the left-hand side of  Eq.~\ref{eq:Ramanujan_recip} is equal to\begin{align}
\int_0^1\int_0^1\frac{\mathscr U^{-\nu-1}(1-\mathscr U)^\nu\mathscr V^{-\nu-1}(1-\mathscr V)^\nu\log \mathscr U\D \mathscr U\D \mathscr V}{(1-\alpha \mathscr V)^{2\nu+1}[1-\alpha\mathscr V-\beta(1-\mathscr V)+\alpha\beta\mathscr U\mathscr V(1-\mathscr V)]^{-\nu}}.
\end{align}As we follow this up  with a Jacobi involution $ \mathscr V=\frac{1-\mathscr W}{1-\alpha\mathscr W}$, we can identify the left-hand side of  Eq.~\ref{eq:Ramanujan_recip}  with \begin{align}
\int_0^1\int_0^1\frac{\mathscr U^{-\nu-1}(1-\mathscr U)^\nu(1-\mathscr W)^{-\nu-1}\mathscr W^\nu\log \mathscr U\D \mathscr U\D \mathscr W}{[(1-\alpha\mathscr W)(1-\beta\mathscr W)+\alpha\beta\mathscr U\mathscr W(1-\mathscr W)]^{-\nu}},\label{eq:symm_H_recip}
\end{align}an expression that is symmetric in the variables $ \alpha$ and $ \beta$. Meanwhile, if we apply  Ramanujan's rotations $U= \frac{1-\mathscr U}{1-\beta  \mathscr U \mathscr V},V=\frac{(1-\mathscr V) (1-\beta  \mathscr U \mathscr V)}{1-\beta\mathscr V  }$ to  the right-hand side of  Eq.~\ref{eq:Ramanujan_recip}, and follow up with a Jacobi involution $ \mathscr V=\frac{1-\mathscr W}{1-\beta\mathscr W}$, we also arrive at the expression displayed in Eq.~\ref{eq:symm_H_recip}.

This proves the reciprocity relation stated in Eq.~\ref{eq:Ramanujan_recip}.
 \end{proof}

 In view of Eqs.~\ref{eq:FZ_a'} and \ref{eq:Ramanujan_recip}, one may demonstrate the following identity\begin{align}
&\frac{\pi^3[P_\nu(2\alpha-1)P_\nu(1-2\beta)-P_\nu(1-2\alpha)P_\nu(2\beta-1)]}{\sin^3(\nu\pi)}+\frac{\pi^2P_\nu(1-2\alpha)P_\nu(1-2\beta)}{\sin^2(\nu\pi)}\left(\log\frac{\beta}{1-\beta}-\log\frac{\alpha}{1-\alpha}\right)\notag\\{}&+\iint\log\frac{1-\alpha UV}{1-\alpha U}(\mathbb D_{\nu,\alpha}U\mathbb D_{\nu,\beta}V+\mathbb D_{-\nu-1,\alpha}U\mathbb D_{-\nu-1,\beta}V)\notag\\={}&\iint\log\frac{1-\beta UV}{1-\beta V}(\mathbb D_{\nu,\alpha}U\mathbb D_{\nu,\beta}V+\mathbb D_{-\nu-1,\alpha}U\mathbb D_{-\nu-1,\beta}V)\label{eq:H_nu_recip}
\end{align}for $ \alpha,\beta\in(\mathbb C\smallsetminus\mathbb R)\cup(0,1)$, $ \nu\in(-1,0)$, with all the integration paths being straight-line segments. Therefore, one can condense  Eq.~\ref{eq:intn_ent_log_form'}   into the following form: \begin{align}&
\frac{2\pi^2}{\sin^2(\nu\pi)}\left\{[P_{\nu }(2\beta-1)]^2\int_0^{\beta}\frac{[P_\nu(1-2t)]^2\D t}{t-\alpha}-P_\nu(1-2\beta)P_\nu(2\beta-1)\int_0^{\beta}\frac{P_\nu(1-2t)P_\nu(2t-1)\D t}{t-\alpha}\right\}\notag\\={}&P_\nu(2\beta-1)\left\{P_{\nu}(2\alpha-1)\iint\log\frac{1-\beta UV}{1-\beta V}(\mathbb D_{\nu,\alpha}U\mathbb D_{\nu,\beta}V+\mathbb D_{-\nu-1,\alpha}U\mathbb D_{-\nu-1,\beta}V)\right.\notag\\{}&\left.- P_{\nu}(1-2\alpha)\iint\log\left[ 1-\frac{\beta(1-\alpha)}{\alpha(1-\beta)}(1-U)(1- V)\right](\mathbb D_{-\nu-1,1-\alpha}U\mathbb D_{\nu,\beta}V+\mathbb D_{\nu,1-\alpha}U\mathbb D_{-\nu-1,\beta}V)\right\},\label{eq:intn_ent_log_form}\end{align}for $ -1<\nu<0$ and $0<\beta<\alpha<1$.

As another  application of Lemma \ref{lm:Ramanujan_recip}, we construct a functional equation involving certain double integrals, in order to  show that the sum  $ h_\nu(\alpha\Vert\beta)+h_\nu(1-\alpha\Vert1-\beta)$ remains unchanged when the moduli parameters $ \alpha$ and $\beta$ trade their places. \begin{lemma}[A Reflection Formula for Ramanujan Entropy Coupling]
For  $ \nu\in\{-1/6$, $-1/4$, $-1/3$, $-1/2\}$ and $ \alpha,\beta\in(\mathbb C\smallsetminus\mathbb R)\cup(0,1)$, define the Ramanujan entropy coupling\begin{align}&
\mathscr R_\nu(\alpha\Vert\beta):=\frac{1}{2}\left( \mathbb E^U_{\nu,\alpha}\mathbb E^V_{\nu,\beta}+ \mathbb E^U_{-\nu-1,\alpha}\mathbb E^V_{-\nu-1,\beta}\right)\log\frac{1-\alpha UV}{\frac{1-\alpha U}{1-U}},\label{eq:Ramanujan_HC_defn}
\end{align} then $ \mathscr R_\nu(\alpha\Vert\beta)=\mathscr R_\nu(\beta\Vert\alpha)$ (according to Eq.~\ref{eq:Ramanujan_recip}), and we have the following exact identities:\begin{align}&
 \mathscr R_\nu(\alpha\Vert\beta)- \mathscr R_\nu(1-\beta\Vert1-\alpha)\notag\\={}&h_\nu(\alpha\Vert\beta)-h_\nu(1-\beta\Vert1-\alpha)-\frac{1}{2}\left(\log\frac{\alpha}{1-\alpha}+\log\frac{\beta}{1-\beta}\right)\notag\\{}&+\frac{\pi}{2\sin(\nu\pi)}\left[ \frac{P_\nu(2\alpha-1)}{P_\nu(1-2\alpha)}-\frac{P_\nu(1-2\alpha)}{P_\nu(2\alpha-1)}+ \frac{P_\nu(2\beta-1)}{P_\nu(1-2\beta)}-\frac{P_\nu(1-2\beta)}{P_\nu(2\beta-1)}\right],\label{eq:R_nu_alpha_beta_refl_form}
\intertext{}&h_\nu(\alpha\Vert\beta)+h_\nu(1-\alpha\Vert1-\beta)=h_\nu(\beta\Vert\alpha)+h_\nu(1-\beta\Vert1-\alpha).\label{eq:hReg_recip_refl}\end{align}\end{lemma}\begin{proof}By definitions of $ h_\nu(\cdot\Vert\cdot)$ (Eq.~\ref{eq:reg_h_nu_alpha_beta_defn}) and $ \mathscr R_\nu(\cdot\Vert\cdot)$ (Eq.~\ref{eq:Ramanujan_HC_defn}), as well as applications of Eqs.~\ref{eq:FZ_a'} and \ref{eq:FZ_b'}, we have \begin{align}&
h_\nu(\alpha\Vert\beta)-
 \mathscr R_\nu(\alpha\Vert\beta)-\frac{1}{2}\log\frac{\alpha}{1-\alpha}+\frac{\pi}{2\sin(\nu\pi)}\left[ \frac{P_\nu(2\alpha-1)}{P_\nu(1-2\alpha)}-\frac{P_\nu(1-2\alpha)}{P_\nu(2\alpha-1)} \right]\notag\\={}&-\frac{1}{2}\left( \mathbb E^U_{\nu,1-\alpha}\mathbb E^V_{\nu,\beta}+\mathbb E^U_{-\nu-1,1-\alpha}\mathbb E^V_{-\nu-1,\beta} \right)\log\left( 1+\frac{\alpha UV}{1-U} \right)+\frac{\log(1-\alpha)}{2}-\frac{\pi}{2\sin(\nu\pi)}\frac{P_\nu(1-2\alpha)}{P_\nu(2\alpha-1)}.
\end{align}By a Jacobi involution $ U=\frac{1-u}{1-(1-\alpha) u}$ together with applications of Eqs.~\ref{eq:FZ_a'} and \ref{eq:FZ_b'},  we convert the expression above into\begin{align}
-\frac{1}{2}\left( \mathbb E^u_{-\nu-1,1-\alpha}\mathbb E^V_{\nu,\beta}+\mathbb E^u_{\nu,1-\alpha}\mathbb E^V_{-\nu-1,\beta} \right)\log(u+V-uV),
\end{align} which is manifestly invariant under the transformations $ (\alpha,\beta)\mapsto(1-\beta,1-\alpha)$. This proves    Eq.~\ref{eq:R_nu_alpha_beta_refl_form}.

 Now that the left-hand side of    Eq.~\ref{eq:R_nu_alpha_beta_refl_form} remains intact under the interchange of moduli parameters $ \alpha\leftrightarrow\beta$, so must its right-hand side. Spelling out this invariance property for the right-hand side of Eq.~\ref{eq:R_nu_alpha_beta_refl_form}, and rearranging, we reach  Eq.~\ref{eq:hReg_recip_refl}.\end{proof}

While the Legendre function satisfies $ P_\nu=P_{-\nu-1}$, such a ``$\nu$-reflection symmetry'' cannot be taken for granted for many multiple integrals we have encountered so far. Fortunately, owing to the Ramanujan reciprocity (Lemma \ref{lm:Ramanujan_recip}),  some important double integrals do remain invariant  when   $ \nu$ and $ -\nu-1$ switch their places, as explained in the next lemma.
\begin{lemma}[$ \nu$-Reflection Symmetry for Certain Double Integrals]\label{lm:nu_reflection_symm}Let $ \nu$ be a rational number from the set $\{-1/6,-1/4,-1/3,-1/2\} $. For  $ \alpha,\beta\in(\mathbb C\smallsetminus\mathbb R)\cup(0,1)$, we have the following identity:\begin{align}
\left( \mathbb E_{\nu,\alpha\vphantom{\beta}}^{U} \mathbb E_{\nu,\beta}^{V}-\mathbb E_{-\nu-1,\alpha\vphantom{\beta}}^{U} \mathbb E_{-\nu-1,\beta}^{V}\right)\log\frac{1-\beta UV}{1-\beta V}={}&0.\label{eq:H1_nu_refl}
\end{align}For  $\alpha,1-\alpha,\beta,1-\beta,\frac{\beta(1-\alpha)}{\alpha(1-\beta)}\in\mathbb C\smallsetminus[1,+\infty)$, the following relation is true:\begin{align}
\left( \mathbb E_{-\nu-1,1-\alpha\vphantom{\beta}}^{U} \mathbb E_{\nu,\beta}^{V}-\mathbb E_{\nu,1-\alpha\vphantom{\beta}}^{U} \mathbb E_{-\nu-1,\beta}^{V} \right)\log\left[ 1-\frac{1-\alpha}{\alpha}\frac{\beta}{1-\beta}(1-U)(1-V)\right]=0.\label{eq:H2_nu_refl}
\end{align}   \end{lemma}\begin{proof}Without loss of generality, we may demonstrate the claimed identities for $ 0<\beta<\alpha<1$, and perform analytic continuation later afterwards.

To show   Eq.~\ref{eq:H1_nu_refl}, we first put down\begin{align}
\iint\log\frac{1-\beta UV}{1-\beta V}\mathbb D_{\nu,\alpha }U\mathbb D_{\nu,\beta}V=\iint\mathbb E^W_{-\nu-1,0}\frac{\log\frac{1-\beta UV}{1-\beta V}}{1-\alpha UW}\mathbb D_{\nu,0 }U\mathbb D_{\nu,\beta}V
\end{align} and then argue that \begin{align}
&\iint\mathbb E^W_{-\nu-1,0}\frac{\log\frac{1-\beta UV}{1-\beta V}-\log(1-\alpha UW)}{1-\alpha UW}\mathbb D_{\nu,0 }U\mathbb D_{\nu,\beta}V\notag\\={}&\iint\mathbb E^{V}_{\nu,0}\frac{\log\frac{1-\alpha\mathscr  UW}{1-\alpha W}-\log(1-\beta\mathscr  UV)}{1-\beta\mathscr  UV}\mathbb D_{-\nu-1,0 }\mathscr U\mathbb D_{-\nu-1,\alpha}W\notag\\={}&\iint\log\frac{1-\alpha\mathscr  UW}{1-\alpha W}\mathbb D_{-\nu-1,\beta }\mathscr U\mathbb D_{-\nu-1,\alpha}W-\iint\mathbb E^{V}_{\nu,0}\frac{\log(1-\beta\mathscr  UV)}{1-\beta\mathscr  UV}\mathbb D_{-\nu-1,0 }\mathscr U\mathbb D_{-\nu-1,\alpha}W\label{eq:H1_nu_refl_prep}
\end{align} follows from a birational transformation\begin{align}
\frac{U}{1-U}=\frac{1}{1-\beta V}\frac{1}{1-\alpha W}\frac{1-\mathscr U}{\mathscr U}
\end{align} and an integration over $V$. By  the Ramanujan reciprocity, we have \begin{align}{}&
\iint\log\frac{1-\alpha\mathscr  UW}{1-\alpha W}\mathbb D_{-\nu-1,\beta }\mathscr U\mathbb D_{-\nu-1,\alpha}W\notag\\={}&\iint\log\frac{1-\beta\mathscr  UW}{{1-\beta\mathscr U}}\mathbb D_{-\nu-1,\beta }\mathscr U\mathbb D_{-\nu-1,\alpha}W+\iint\log\frac{1-\mathscr U}{1-W}\mathbb D_{-\nu-1,\beta }\mathscr U\mathbb D_{-\nu-1,\alpha}W,
\end{align}while Eq.~\ref{eq:Erdelyi_log_deriv} allows us to deduce\begin{align}&
\iint\log\frac{1-\mathscr U}{1-W}\mathbb D_{-\nu-1,\beta }\mathscr U\mathbb D_{-\nu-1,\alpha}W\notag\\={}&\iint\mathbb E^{V}_{\nu,0}\frac{\log(1-\beta\mathscr  UV)}{1-\beta\mathscr  UV}\mathbb D_{-\nu-1,0 }\mathscr U\mathbb D_{-\nu-1,\alpha}W-\iint\mathbb E^{V}_{\nu,0}\frac{\log(1-\alpha VW)}{1-\alpha VW}\mathbb D_{-\nu-1,\beta }\mathscr U\mathbb D_{-\nu-1,0}W,
\end{align}so Eq.~\ref{eq:H1_nu_refl_prep} leads us to     Eq.~\ref{eq:H1_nu_refl}.

Had we refrained from pairing up $ \nu$ and $ -\nu-1$ in Propositions~\ref{prop:multi_integrals_Jacobi_involution}--\ref{prop:multi_integrals_Jacobi_involution_again}, and summarized the computations in Eqs.~\ref{eq:rel_intn_add_0}, \ref{eq:rel_intn_self_intn}, \ref{eq:rel_ent_princip}, before invoking the Ramanujan reciprocity from Eq.~\ref{eq:Ramanujan_recip}, we would have obtained the following identity for  $ -1<\nu<0$ and $0<\beta<\alpha<1$: \begin{align}&
\frac{\pi^2}{\sin^2(\nu\pi)}\left\{[P_{\nu }(2\beta-1)]^2\int_0^{\beta}\frac{[P_\nu(1-2t)]^2\D t}{t-\alpha}-P_\nu(1-2\beta)P_\nu(2\beta-1)\int_0^{\beta}\frac{P_\nu(1-2t)P_\nu(2t-1)\D t}{t-\alpha}\right\}\notag\\={}&P_\nu(2\beta-1)\left\{P_{\nu}(2\alpha-1)\iint\log\frac{1-\beta UV}{1-\beta V}\mathbb D_{\nu,\alpha}U\mathbb D_{\nu,\beta}V\right.\notag\\{}&\left.- P_{\nu}(1-2\alpha)\iint\log\left[ 1-\frac{\beta(1-\alpha)}{\alpha(1-\beta)}(1-U)(1- V)\right]\mathbb D_{-\nu-1,1-\alpha}U\mathbb D_{\nu,\beta}V\right\},\label{eq:intn_ent_log_form''}\tag{\ref{eq:intn_ent_log_form}$'$}\end{align}
in place of Eq.~\ref{eq:intn_ent_log_form}. From the equation above, we see that the symmetry relation in Eq.~\ref{eq:H2_nu_refl} descends from Eq.~\ref{eq:H1_nu_refl}.
\end{proof}

\begin{proposition}[Kontsevich--Zagier Integrals as Entropy Couplings]\label{prop:HC_int_Ent}\begin{enumerate}[label=\emph{(\alph*)}, ref=(\alph*), widest=a]
\item When $ \nu\in\{-1/6$, $-1/4$, $-1/3$, $-1/2\}$ and $\alpha,1-\alpha,\beta,1-\beta,\frac{\beta(1-\alpha)}{\alpha(1-\beta)}\in\mathbb C\smallsetminus[1,+\infty)$, we have an integral  identity:\begin{align}
&
[P_{\nu }(2\beta-1)]^2\int_0^{\beta}\frac{[P_\nu(1-2t)]^2\D t}{t-\alpha}-P_\nu(1-2\beta)P_\nu(2\beta-1)\int_0^{\beta}\frac{P_\nu(1-2t)P_\nu(2t-1)\D t}{t-\alpha}\notag\\\equiv{}&P_\nu(1-2\alpha)P_\nu(2\alpha-1)P_\nu(1-2\beta)P_\nu(2\beta-1)H_\nu(\alpha\Vert\beta),\pmod{2\pi  i\Lambda_\nu(\alpha,\beta)},\label{eq:S_nu_H_nu}
\end{align}where the integration paths on the left-hand side circumvent singularities of the  integrands. The notations $H_\nu(\alpha\Vert\beta) $ and  $ \Lambda_\nu(\alpha,\beta)$ were defined, respectively,  in Eqs.~\ref{eq:HC_defn_intro} and \ref{eq:Lambda_nu_alpha_beta_defn}.
\item For $ \nu\in(-1,0),\beta\in(\mathbb C\smallsetminus\mathbb R)\cup(0,1)$, the following integral identity holds:
\begin{align}&\frac{\pi^2}{\sin^2(\nu\pi)}\left\{P_{\nu }(2\beta-1)\int_0^{\beta}[P_\nu(1-2t)]^2\D t-P_\nu(1-2\beta)\int_0^{\beta}P_\nu(1-2t)P_\nu(2t-1)\D t\right\}\notag\\={}&
\frac{\sqrt{\pi }}{2^{2\nu+1}\Gamma (-\nu ) \Gamma \left(\nu +\frac{3}{2}\right)}\iint U^\nu\log(1-\beta UV)\mathbb D_{\nu,0}U\mathbb D_{\nu,\beta}V,\label{eq:Pnu_Pnu_degen_coupling}\end{align}with all the integration paths being straight-line segments.
\end{enumerate}\end{proposition}
\begin{proof}\begin{enumerate}[label=(\alph*),widest=a]\item This is an analytic continuation of Eq.~\ref{eq:intn_ent_log_form}.

\item
For $ 0<\beta<\alpha<1$, we may apply a $ \nu$-reflection $ \nu\mapsto-\nu-1$ and a  Jacobi involution $\mathscr  V=\frac{1-V}{1-\beta V}$ to the last  integral in Eq.~\ref{eq:intn_ent_log_form''}, which results in an exact equality:\begin{align}
&\frac{\pi^2}{\sin^2(\nu\pi)}\left\{[P_{\nu }(2\beta-1)]^2\int_0^{\beta}\frac{[P_\nu(1-2t)]^2\D t}{t-\alpha}-P_\nu(1-2\beta)P_\nu(2\beta-1)\int_0^{\beta}\frac{P_\nu(1-2t)P_\nu(2t-1)\D t}{t-\alpha}\right\}\notag\\={}&P_{\nu}(2\alpha-1)P_\nu(2\beta-1)\iint\log(1-\beta UV)\mathbb D_{\nu,\alpha}U\mathbb D_{\nu,\beta}V\notag\\{}&- P_{\nu}(1-2\alpha)P_\nu(2\beta-1)\iint\log\left[1-\frac{1-(1-\alpha)U}{\alpha}\beta\mathscr  V\right]\mathbb D_{\nu,1-\alpha}U\mathbb D_{\nu,\beta}\mathscr V.\label{eq:H_nu_Jac_inv}\end{align} This extends to a valid identity for $ \beta\in(0,1)$ and $ \alpha\in(\mathbb C\smallsetminus\mathbb R)\cup(\beta,1)$. Multiplying both sides of the equation above by $ -\alpha/P_\nu(2\beta-1)$ and taking the $ \alpha\to i\infty$ limit, we arrive at Eq.~\ref{eq:Pnu_Pnu_degen_coupling} for $ \beta\in(0,1)$. The rest follows from analytic continuation.
 \qedhere\end{enumerate}\end{proof}\begin{remark}We note that the sum of the     last two integrals in  Eq.~\ref{eq:intn_ent_log_form} is invariant under the variable substitutions $ (\alpha,\beta)\mapsto(1-\beta,1-\alpha)$. Thus, a comparison between   Eqs.~\ref{eq:intn_ent_log_form} and  \ref{eq:intn_ent_log_form'}  also results in an independent verification of the reciprocity relation for $ S_\nu(\alpha\Vert\beta)-S_\nu(\beta\Vert\alpha)$, which was mentioned in Eq.~\ref{eq:S_recip}.\eor\end{remark}\begin{remark}For $ \nu\in\{-1/6,-1/4,-1/3\}$,  Eq.~\ref{eq:Pnu_Pnu_degen_coupling} has appeared in the Kontsevich--Zagier integral representations for $ G_2^{\mathfrak H/\overline{\varGamma}_0(N)}(z_{1},z_{2}),N=4\sin^2(\nu\pi)\in\{1,2,3\}$, where $\alpha_N(z)=\infty,\alpha_N(z')=\beta $ \cite[][\S2.2]{AGF_PartI}; for $ \nu=-1/2$, Eq.~\ref{eq:Pnu_Pnu_degen_coupling} has appeared an integral representation for the Epstein zeta function \cite[][\S4]{EZF}:\begin{align}&
E^{\varGamma_0(4)}\left( -\frac{1}{4z},2 \right):=-\frac{3}{4\pi}\lim_{\I z'\to+\infty}G_{2}^{\mathfrak H/\overline{\varGamma}_0(4)}(z,z')\I z'\notag\\={}&\frac{21 \zeta (3)}{8 \pi ^3\I z}+\frac{3}{4\pi^{3}}\R\int_0^{\lambda(2z+1)}\frac{[\mathbf K(\sqrt{t})]}{\I z}^2\left[ \frac{i\mathbf K(\sqrt{1-t})}{\mathbf K(\sqrt{t})} -2z-1\right]\left[ \frac{i\mathbf K(\sqrt{1-t})}{\mathbf K(\sqrt{t})} -2\overline{z}-1\right]\D t,\label{eq:Epstein_2_Eichler}
\end{align}where $ \zeta(3)=\sum_{n=1}^\infty n^{-3}=\frac{2}{7}\int_0^1[\mathbf K(\sqrt{t})]^2\D t$ is Ap\'ery's constant, and $ z+\frac12\in\Int\mathfrak D_4$. Specializing Eq.~\ref{eq:Pnu_Pnu_degen_coupling} to the case where $ \nu=-1/2$, we obtain \begin{align}
&[\mathbf K(\sqrt{\vphantom1\smash[b]{1-\mu}})]^2\int_0^\mu[\mathbf K(\sqrt{t})]^2\D t-\mathbf K(\sqrt{\vphantom1\smash[b]{\mu }})\mathbf K(\sqrt{1-\smash[b]{\mu}})\int_0^\mu\mathbf K(\sqrt{t})\mathbf K(\sqrt{1-t})\D t\notag\\={}&\frac{\pi}{2} \mathbf K(\sqrt{1-\smash[b]{\mu}})\int_0^{\pi/2}\int_0^{\pi/2}\frac{\log(1-\mu\sin^2\theta\sin^2\varphi)\D\theta\D\varphi}{\sin\theta\sqrt{\smash[b]{1-\mu\sin^2\varphi}}},\quad \mu\in(\mathbb C\smallsetminus\mathbb R)\cup(0,1)\label{eq:deg_height_coupling_prep}.
\end{align}A limit scenario of the reflection formula for  Ramanujan entropy coupling (Eq.~\ref{eq:R_nu_alpha_beta_refl_form}) leads us to  \begin{align}&
\frac{1}{\mathbf  K(\sqrt{\vphantom1\smash[b]{\mu }})}
\int_0^{\pi/2}\int_0^{\pi/2}\frac{\log(1-\mu\sin^2\theta\sin^2\varphi)\D\theta\D\varphi}{\sin\theta\sqrt{\smash[b]{1-\mu\sin^2\varphi}}}\notag\\{}&-\frac{1}{\mathbf  K(\sqrt{\vphantom1\smash[b]{1-\mu }})}
\int_0^{\pi/2}\int_0^{\pi/2}\frac{\log[1-(1-\mu)\sin^2\theta\sin^2\varphi]\D\theta\D\varphi}{\sin\theta\sqrt{\smash[b]{1-(1-\mu)\sin^2\varphi}}}\notag\\={}&\frac{\pi}{2\mathbf K(\sqrt{\vphantom1\smash[b]{\mu }})\mathbf K(\sqrt{\vphantom1\smash[b]{1-\mu }})}\int_0^{\pi/2}\frac{F(\theta,1-\mu)[2\mathbf K(\sqrt{\vphantom1\smash[b]{1-\mu }})-F(\theta,1-\mu)]\D\theta}{\sin\theta}-\frac{\pi^{2}}{4}.
\end{align}Setting $ \mu=1-\lambda(2z+1)$, and parametrizing the last integral with elliptic functions, we arrive at  \cite[cf.][Theorem 1.1(b)]{EZF}\begin{align}
&E^{\varGamma_0(4)}\left( -\frac{1}{4z},2 \right)\notag\\={}&\frac{3\I z}{4\pi^{2}}\R\frac{\partial}{\partial\I z}\frac{2z+1}{i\I z}\int_0^{\mathbf K(\sqrt{\lambda(2z+1)})}\frac{\dn(u|\lambda(2z+1))}{\sn(u|\lambda(2z+1))}\frac{\pi u[u-2\mathbf K(\sqrt{\lambda(2z+1)})]\D u}{4\mathbf K(\sqrt{\lambda(2z+1)})\mathbf K(\sqrt{1-\lambda(2z+1)})}\notag\\={}&\frac{21\zeta(3)}{8\pi^3\I z}+\frac{6}{\pi^3}\R\frac{\partial}{\partial\I z}\frac{1}{\I z}\sum_{n=0}^\infty\frac{1}{(2n+1)^{3}}\frac{1}{e^{(2n+1)\pi \frac{2z+1}{i}}+1}
\end{align}  for $ z+\frac12\in\Int\mathfrak D_4$, where the last step involves a standard Fourier expansion  \cite[][item 908.06]{ByrdFriedman}\begin{align}
\frac{\dn(u|\lambda(2z+1))}{\sn(u|\lambda(2z+1))}={}&\frac{\pi}{2\mathbf K(\sqrt{\lambda(2z+1)})}\csc\frac{\pi u}{2\mathbf K(\sqrt{\lambda(2z+1)})}\notag\\{}&-\frac{2\pi}{\mathbf K(\sqrt{\lambda(2z+1)})}\sum^\infty_{m=0}\frac{e^{(2n+1)\pi i(2z+1)}}{1+e^{(2n+1)\pi i(2z+1)}}\sin\frac{(2n+1)\pi u}{2\mathbf K(\sqrt{\lambda(2z+1)})},
\end{align} and termwise integration.  \eor\end{remark}

\subsection{Limit behavior of certain entropy formulae\label{subsec:lim_ent}}To prepare for the proof of Theorem~\ref{thm:app_HC} in \S\ref{subsec:G2_Hecke4_GZ_renorm}, we need to analyze some special entropy formulae.

First, we discuss entropy formulae that arise from the asymptotic behavior of $ G_2^{\mathfrak H/\overline{\varGamma}_0(N)}(z,z')=\mathscr I_N(z,z'),N\in\{2,3,4\}$ (Eq.~\ref{eq:G2_Hecke234_via_I}) as $ z'$ approaches $z$. If we set \begin{align}f_{\nu}(\alpha):={}&\int_0^{\alpha}\frac{[P_\nu(1-2t)]^2-[P_\nu(1-2\alpha)]^2}{(t-\alpha)P_\nu(1-2\alpha)P_\nu(2\alpha-1)}\D t,\label{eq:f_nu_alpha_defn}\\g_\nu(\alpha):={}&\int_0^{\alpha}\frac{P_\nu(1-2t)P_\nu(2t-1)-P_\nu(1-2\alpha)P_\nu(2\alpha-1)}{(t-\alpha)P_\nu(1-2\alpha)P_\nu(2\alpha-1)}\D t,\label{eq:g_nu_alpha_defn}\end{align} and  (cf.~Eq.~\ref{eq:trip_int_part0})\begin{align}
\mathfrak Q_{\nu}({\alpha}):={}&\int_0^1\left\{\int_0^U\left[\int_W^1\frac{\D V}{\mathbb Y_{\nu,\alpha}(V)}\right]\frac{\D W}{\mathbb Y_{-\nu-1,\alpha}(W)}\right\}\frac{\D U}{\mathbb Y_{\nu,\alpha}(U)}\notag\\{}&+\int_0^1\left\{\int_0^U\left[\int_W^1\frac{\D V}{\mathbb Y_{-\nu-1,\alpha}(V)}\right]\frac{\D W}{\mathbb Y_{\nu,\alpha}(W)}\right\}\frac{\D U}{\mathbb Y_{-\nu-1,\alpha}(U)}\notag\\={}&\int_0^1\left[\int_W^1\frac{\D V}{\mathbb Y_{\nu,\alpha}(V)}\right]^2\frac{\D W}{\mathbb Y_{-\nu-1,\alpha}(W)}+\int_0^1\left[\int_W^1\frac{\D V}{\mathbb Y_{-\nu-1,\alpha}(V)}\right]^2\frac{\D W}{\mathbb Y_{\nu,\alpha}(W)},\label{eq:frakQ_nu_alpha_defn}
\end{align} then we may combine Eqs.~\ref{eq:Pnu_sqr_KZ_int2trip_int}, \ref{eq:Pnu_mir_KZ_int2trip_int}, \ref{eq:log_self_sqr}, \ref{eq:log_self_mir},  and  \ref{eq:trip_indef_to_double_log_def} into the following identity:\begin{align}&
\frac{P_\nu(2\alpha-1)}{P_\nu(1-2\alpha)}f_{\nu}(\alpha)-g_\nu(\alpha)=-\frac{\pi}{2\sin(\nu\pi)}\frac{P_\nu(1-2\alpha)}{P_\nu(2\alpha-1)}-\frac{\sin^{2}(\nu\pi)\mathfrak Q_{\nu}({\alpha)}}{2\pi^2[P_\nu(1-2\alpha)]^{2}P_\nu(2\alpha-1)}.\label{eq:f-g}
\end{align}Furthermore, \cite[][Eq.~57]{Zhou2013Pnu} can be readily generalized into\begin{align}
g_\nu(\alpha)-g_\nu(1-\alpha)=-\frac{\pi}{2\sin(\nu\pi)}\left[ \frac{P_\nu(2\alpha-1)}{P_\nu(1-2\alpha)} -\frac{P_\nu(1-2\alpha)}{P_\nu(2\alpha-1)}\right]+\log\frac{\alpha}{1-\alpha}.\label{eq:g-g}
\end{align}

In the following lemma, we shall investigate the limit behavior of $f_\nu(\alpha)$, $g_\nu(\alpha)$ and $ \mathfrak Q_\nu(\alpha)$ as $\alpha(1-\alpha)$ tends to zero. \begin{lemma}[Limit Behavior of Self-Interaction Entropies as $\alpha\to0$ and $ \alpha\to1$]\label{lm:cusp_limits}\begin{enumerate}[label=\emph{(\alph*)}, ref=(\alph*), widest=a]
\item As $ \alpha\to0$, we have the following expansions:\begin{align}\frac{\sin^{2}(\nu\pi)\mathfrak Q_{\nu}({\alpha)}}{2\pi^2[P_\nu(1-2\alpha)]^{2}P_\nu(2\alpha-1)}={}&o(1),\label{eq:Q_nu_alpha0}\\\widehat D_\alpha\frac{\sin^{2}(\nu\pi)\mathfrak Q_{\nu}({\alpha)}}{2\pi^2[P_\nu(1-2\alpha)]^{2}P_\nu(2\alpha-1)}={}&o(1),\label{eq:D_Q_nu_alpha0}\\\widehat D^2_\alpha\frac{\sin^{2}(\nu\pi)\mathfrak Q_{\nu}({\alpha)}}{2\pi^2[P_\nu(1-2\alpha)]^{2}P_\nu(2\alpha-1)}={}&o(1),\label{eq:DD_Q_nu_alpha0}\\
\frac{P_\nu(2\alpha-1)}{P_\nu(1-2\alpha)}f_{\nu}(\alpha)={}&o(1),\label{eq:f_nu_alpha0}\\g_\nu(\alpha)={}&o(1),\label{eq:g_nu_alpha0}
\end{align}where\begin{align}\widehat D_\alpha:=
\frac{\pi\alpha(1-\alpha)}{\sin(\nu\pi)}P_\nu(1-2\alpha)P_\nu(2\alpha-1) \frac{\partial}{\partial\alpha},\label{eq:op_D_alpha}
\end{align}and $o(1)$ represents an infinitesimal quantity in the limit procedure, during which the  branch cut of $ P_\nu(2\alpha-1),\alpha\in\mathbb C\smallsetminus(-\infty,0]$ is avoided. \item As $ \alpha\to1$, without interfering with the branch cut of  $ P_\nu(1-2\alpha),\alpha\in\mathbb C\smallsetminus[1,+\infty)$, we have the following expansions:\begin{align}\frac{\sin^{2}(\nu\pi)\mathfrak Q_{\nu}({\alpha)}}{2\pi^2[P_\nu(1-2\alpha)]^{2}P_\nu(2\alpha-1)}={}&-\frac{\pi}{3\sin(\nu\pi)}\frac{P_\nu(1-2\alpha)}{P_\nu(2\alpha-1)}+o(1),\label{eq:Q_nu_alpha1}\\\widehat D_\alpha\frac{\sin^{2}(\nu\pi)\mathfrak Q_{\nu}({\alpha)}}{2\pi^2[P_\nu(1-2\alpha)]^{2}P_\nu(2\alpha-1)}={}&+\frac{\pi}{3\sin(\nu\pi)}\frac{P_\nu(1-2\alpha)}{P_\nu(2\alpha-1)}+o(1),\label{eq:D_Q_nu_alpha1}\\\widehat D^2_\alpha\frac{\sin^{2}(\nu\pi)\mathfrak Q_{\nu}({\alpha)}}{2\pi^2[P_\nu(1-2\alpha)]^{2}P_\nu(2\alpha-1)}={}&-\frac{\pi}{3\sin(\nu\pi)}\frac{P_\nu(1-2\alpha)}{P_\nu(2\alpha-1)}+o(1),\label{eq:DD_Q_nu_alpha1}\\
\frac{P_\nu(2\alpha-1)}{P_\nu(1-2\alpha)}f_{\nu}(\alpha)={}&+\frac{\pi}{3\sin(\nu\pi)}\frac{P_\nu(1-2\alpha)}{P_\nu(2\alpha-1)}-\log(1-\alpha)+o(1),\label{eq:f_nu_alpha1}\\g_\nu(\alpha)={}&+\frac{\pi}{2\sin(\nu\pi)}\frac{P_\nu(1-2\alpha)}{P_\nu(2\alpha-1)}-\log(1-\alpha)+o(1).\label{eq:g_nu_alpha1}
\end{align} \end{enumerate}\end{lemma}
\begin{proof}\begin{enumerate}[label=(\alph*), ref=(\alph*), widest=a]
\item Judging from Eq.~\ref{eq:frakQ_nu_alpha_defn}, we know that $ \mathfrak Q_{\nu}({\alpha)}$ admits a convergent Taylor series in an open neighborhood of the origin in the complex $ \alpha$-plane, so Eqs.~\ref{eq:Q_nu_alpha0}--\ref{eq:DD_Q_nu_alpha0} immediately follow. Using the definition of $f_\nu$ in Eq.~\ref{eq:f_nu_alpha_defn}, one can justify Eq.~\ref{eq:f_nu_alpha0}. By Eq.~\ref{eq:f-g}, one subsequently verifies Eq.~\ref{eq:g_nu_alpha0}.\item We first need to  show that \begin{align}
-\frac{\sin^3(\nu\pi)}{\pi^3}\frac{\mathfrak Q_{\nu}({\alpha})}{[P_\nu(1-2\alpha)]^{3}}={}&\frac{2}{3}+O\left( \left[\frac{P_\nu(2\alpha-1)}{P_{\nu}(1-2\alpha)}\right]^2 \right),&\alpha\to{}&1-0^{+}.\label{eq:two_thirds_lim}
\end{align}By Eq.~\ref{eq:frakQ_nu_alpha_defn}, we may confirm \begin{align}{}&
-\frac{\sin^3(\nu\pi)}{\pi^3}\frac{\mathfrak Q_{\nu}({\alpha})}{[P_\nu(1-2\alpha)]^{3}}-\frac{2}{3}\notag\\={}&-\frac{\sin^3(\nu\pi)}{\pi^3[P_\nu(1-2\alpha)]^{3}}\int_0^1\left\{\left[\int_W^1\frac{\D V}{\mathbb Y_{\nu,\alpha}(V)}\right]^2-\left[\int_W^1\frac{\D V}{\mathbb Y_{-\nu-1,\alpha}(V)}\right]^2\right\}\left[\frac{1}{\mathbb Y_{-\nu-1,\alpha}(W)}-\frac{1}{\mathbb Y_{\nu,\alpha}(W)}\right]\D W\notag\\={}&-\frac{\sin^3(\nu\pi)}{2\pi^3[P_\nu(1-2\alpha)]^{3}}\int_0^1\left[\int_W^1\frac{\D V}{\mathbb Y_{\nu,\alpha}(V)}-\int_W^1\frac{\D V}{\mathbb Y_{-\nu-1,\alpha}(V)}\right]^2\left[\frac{1}{\mathbb Y_{-\nu-1,\alpha}(W)}+\frac{1}{\mathbb Y_{\nu,\alpha}(W)}\right]\D W.\label{eq:two_thirds_diff}
\end{align}after integration by parts. Then, it would suffice to show that\begin{align}
\left[\int_W^1\frac{\D V}{\mathbb Y_{\nu,\alpha}(V)}-\int_W^1\frac{\D V}{\mathbb Y_{-\nu-1,\alpha}(V)}\right]^2,\quad W\in(0,1)
\end{align}is a bounded quantity as $ \alpha\to1-0^{+}$. By a variable substitution $ W=w/(w-1)$ (an analog of  Jacobi's imaginary  transformation in elliptic function theory), we see that \begin{align}{}&
\left[\int_W^1\frac{\D V}{\mathbb Y_{\nu,\alpha}(V)}-\int_W^1\frac{\D V}{\mathbb Y_{-\nu-1,\alpha}(V)}\right]^2\notag\\={}&\left\{\int_0^w\frac{(-u)^\nu[1-(1-\alpha)u]^\nu\D u}{(1-u)^{\nu+1}}-\int_0^w\frac{(-u)^{-\nu-1}[1-(1-\alpha)u]^{-\nu-1}\D u}{(1-u)^{-\nu}}\right\}^2\label{eq:int_diff_W_to_w}
\end{align} for $ w<0$, and the right-hand side of the equation above is invariant as one trades $w$ for $ \frac{1}{(1-\alpha)w}$. Therefore, it would suffice to bound the right-hand side of Eq.~\ref{eq:int_diff_W_to_w} for $ -1/\sqrt{1-\alpha}<w<0$. Suppose that $ -2<w<0$, then one can construct bounds like \begin{align}
\left\vert \int_0^w\frac{(-u)^\nu[1-(1-\alpha)u]^\nu\D u}{(1-u)^{\nu+1}}\right\vert\leq \left\vert \int_0^w(-u)^\nu\D u\right|\leq 2^{\nu+1}.
\end{align}For $ -1/\sqrt{1-\alpha}<w<-2$, we use the Taylor expansion of the integrands to show that \begin{align}&
\int_{-2}^w\frac{(-u)^\nu[1-(1-\alpha)u]^\nu\D u}{(1-u)^{\nu+1}}-\int_{-2}^w\frac{(-u)^{-\nu-1}[1-(1-\alpha)u]^{-\nu-1}\D u}{(1-u)^{-\nu}}\notag\\={}&O(\sqrt{1-\alpha})+O\left(\int^w_{-2}\frac{\D u}{(1-u)u}\right)
\end{align}  is bounded.   This completes the verification of the claim in Eq.~\ref{eq:two_thirds_lim}.

With cosmetic changes to the foregoing arguments, one can show that,  for generic $ \alpha\to1$ limits, the expression\begin{align}
\left\{-\frac{\sin^3(\nu\pi)}{\pi^3}\frac{\mathfrak Q_{\nu}({\alpha})}{[P_\nu(1-2\alpha)]^{3}}-\frac{2}{3}\right\} \left[\frac{P_\nu(1-2\alpha)}{P_{\nu}(2\alpha-1)}\right]^2
\end{align}is bounded, and so are its derivatives with respect to $\alpha$. This proves Eqs.~\ref{eq:Q_nu_alpha1}--\ref{eq:DD_Q_nu_alpha1}. Meanwhile, one sees that Eq.~\ref{eq:g_nu_alpha1} descends from Eqs.~\ref{eq:g-g} and \ref{eq:g_nu_alpha0}. Back substituting into Eq.~\ref{eq:f-g}, we can then derive Eq.~\ref{eq:f_nu_alpha1} from Eq.~\ref{eq:Q_nu_alpha1}.   \qedhere
\end{enumerate}\end{proof}

Another topic of  \S\ref{subsec:G2_Hecke4_GZ_renorm} will be  the asymptotic analysis of $ G_2^{\mathfrak H/\overline{\varGamma}_0(1)}(z,z'),z\to z'$. To prepare for this,  we need a quantitative understanding of the special automorphic Green's function\begin{align}\mathfrak G(z)
:=G_2^{\mathfrak H/\overline{\varGamma}(2)}\left(-\frac{1}{z+1},z\right)=G_2^{\mathfrak H/\overline{\varGamma}(2)}\left(-\frac{z+1}{z},z\right),\quad \text{for }j(z)\neq0,\label{eq:frakG_defn}
\end{align} where the weight-4 automorphic Green's function on $ \overline{\varGamma}(2)$ \cite[][Eq.~2.2.1]{AGF_PartI}\begin{align}
G_{2}^{\mathfrak H/\overline{\varGamma}(2)}(z,z')=G_2^{\mathfrak H/\overline{\varGamma}_{0}(4)}\left(\frac{z}{2},\frac{z'}{2}\right),\quad \text{a.e. }z,z'\in\mathfrak H
\end{align}   has involution symmetries $ G_{2}^{\mathfrak H/\overline{\varGamma}(2)}(z,z')=G_{2}^{\mathfrak H/\overline{\varGamma}(2)}(-1/z,-1/z')=G_{2}^{\mathfrak H/\overline{\varGamma}(2)}(z+1,z'+1)$ (see \cite[][Eq.~2.1.19]{AGF_PartI} and \cite[][Eq.~3.8]{EZF}). Using these symmetries, one can check that  $ \mathfrak G(z)=\mathfrak G(-1/z)=\mathfrak G(z+1)$, so the function $ \mathfrak G(z),z\in\mathfrak H\smallsetminus\left(SL(2,\mathbb Z)\frac{1+i\sqrt{3}}{2}\right)$ is in fact $SL(2,\mathbb Z)$-invariant.

Noting that $ \lambda(-1/(z+1))=1/[1-\lambda(z)]$ \cite[cf.][Eq.~2.3.8]{AGF_PartI}, we will be interested in the asymptotic analysis of \begin{align}
\mathscr J\left( \lambda(z)\left\Vert\frac{1}{1-\lambda(z)} \right.\right)\quad\text{and}\quad\mathscr J\left( 1-\lambda(z)\left\Vert\frac{\lambda(z)}{\lambda(z)-1} \right.\right)
\end{align}for \begin{align}
\mathscr J(\lambda\Vert\mu):=\int_0^{\mu}\frac{\frac{ \mathbf K(\sqrt{1-\smash[b]{\mu}})}{\mathbf K(\sqrt{\vphantom1\smash[b]{\mu }})}[\mathbf K(\sqrt{t})]^2-\mathbf K(\sqrt{t})\mathbf K(\sqrt{1-t})}{\mathbf K(\sqrt{\lambda} )\mathbf K(\sqrt{1-\lambda})}\frac{\D t}{t-\lambda},\label{eq:J_lambda_mu}
\end{align}as $z$ approaches the infinite cusp $ i\infty$ of $SL(2,\mathbb Z)$, so that  $ \lambda(z)\to0$. (We note that the difference $ H_{-1/2}(\lambda\Vert\mu)-\mathscr J(\lambda\Vert\mu)$ does not affect the evaluation of automorphic Green's functions.)
\begin{lemma}[Limit Behavior of Certain Level-4 Entropy Formulae as $ \lambda\to0$]We have the following asymptotic formulae\begin{align}\mathscr J\left( \lambda \left\Vert\frac{1}{1-\lambda} \right.\right)={}&-\frac{i\pi\I\lambda}{|\I\lambda|}-\frac{\pi}{2}\frac{\mathbf K(\sqrt{1-\lambda})}{\mathbf K(\sqrt{\lambda})}+\frac{3\pi}{2}\frac{\mathbf K(\sqrt{\lambda})}{\mathbf K(\sqrt{1-\lambda})}+o\left( \frac{\mathbf K(\sqrt{\lambda})}{\mathbf K(\sqrt{1-\lambda})} \right),\label{eq:J_lambda_0}\\
\mathscr J\left( 1-\lambda \left\Vert\frac{\lambda}{\lambda-1} \right.\right)={}&O(\lambda),\label{eq:J_lambda_1}
\end{align} as $ \lambda\to0$ and $ \I \lambda\neq0$.\end{lemma}\begin{proof}In the limit of $ \lambda\to0$, the estimates \begin{align}
\frac{\mathbf K\left(\sqrt\frac{\lambda}{\lambda-1}{}\right)}{\mathbf K\left(\sqrt{\frac{1}{1-\lambda}}\right)}\int_{\frac{1}{1-\lambda}}^{1}\frac{[\mathbf K(\sqrt{t})]^2}{\mathbf K(\sqrt{\lambda} )\mathbf K(\sqrt{1-\lambda})}\frac{\D t}{t-\lambda}={}&O(\lambda)\intertext{and}\int_{\frac{1}{1-\lambda}}^{1}\frac{\mathbf K(\sqrt{t})\mathbf K(\sqrt{1-t})}{\mathbf K(\sqrt{\lambda} )\mathbf K(\sqrt{1-\lambda})}\frac{\D t}{t-\lambda}={}&O(\lambda)
\end{align}follow directly from expansions of the integrands.

Therefore, the asymptotic analysis in the last paragraph allows us to deduce \begin{align}
\mathscr J\left( \lambda \left\Vert\frac{1}{1-\lambda} \right.\right)=\int_0^{1}\frac{\frac{\mathbf K\left(\sqrt\frac{\lambda}{\lambda-1}{}\right)}{\mathbf K\left(\sqrt{\frac{1}{1-\lambda}}\right)}[\mathbf K(\sqrt{t})]^2-\mathbf K(\sqrt{t})\mathbf K(\sqrt{1-t})}{\mathbf K(\sqrt{\lambda} )\mathbf K(\sqrt{1-\lambda})}\frac{\D t}{t-\lambda}+O(\lambda).
\end{align}Here, the closed-form evaluation \begin{align}
\int_0^{1}\frac{\mathbf K(\sqrt{t})\mathbf K(\sqrt{1-t})}{\mathbf K(\sqrt{\lambda} )\mathbf K(\sqrt{1-\lambda})}\frac{\D t}{t-\lambda}=\frac{i\pi\I\lambda}{|\I\lambda|}+\frac{\pi}{2}\left[\frac{\mathbf K(\sqrt{1-\lambda})}{\mathbf K(\sqrt{\lambda})}-\frac{\mathbf K(\sqrt{\lambda})}{\mathbf K(\sqrt{1-\lambda})}\right]\label{eq:KKmir_Tricomi_eval}
\end{align}follows from \cite[][Eq.~51]{Zhou2013Pnu}. Meanwhile, it is clear that \begin{align}
\frac{\mathbf K\left(\sqrt\frac{\lambda}{\lambda-1}{}\right)}{\mathbf K\left(\sqrt{\frac{1}{1-\lambda}}\right)}\int_0^{1}\frac{[\mathbf K(\sqrt{t})]^2-[\mathbf K(\sqrt{\lambda})]^2}{\mathbf K(\sqrt{\lambda} )\mathbf K(\sqrt{1-\lambda})}\frac{\D t}{t-\lambda}=O\left( \left[\frac{\mathbf K(\sqrt{\lambda})}{\mathbf K(\sqrt{1-\lambda})}\right]^2 \right),
\end{align}so \begin{align}
\frac{\mathbf K\left(\sqrt\frac{\lambda}{\lambda-1}{}\right)}{\mathbf K\left(\sqrt{\frac{1}{1-\lambda}}\right)}\int_0^{1}\frac{[\mathbf K(\sqrt{\lambda})]^2}{\mathbf K(\sqrt{\lambda} )\mathbf K(\sqrt{1-\lambda})}\frac{\D t}{t-\lambda}=\pi\frac{\mathbf K(\sqrt{\lambda})}{\mathbf K(\sqrt{1-\lambda})}+o\left( \frac{\mathbf K(\sqrt{\lambda})}{\mathbf K(\sqrt{1-\lambda})} \right)
\end{align}can be combined with Eq.~\ref{eq:KKmir_Tricomi_eval} to yield Eq.~\ref{eq:J_lambda_0}.

The derivation of Eq.~\ref{eq:J_lambda_1} is similar to the first paragraph of this proof. \end{proof}
\subsection{Automorphic  self-energies $G_2^{\mathfrak H/\overline{\varGamma}_0(1)}(z )$, $G_2^{\mathfrak H/\overline{\varGamma}_0(2)}(z )$, $ G_2^{\mathfrak H/\overline{\varGamma}_0(3)}(z )$ and $G_2^{\mathfrak H/\overline{\varGamma}_0(4)}(z ) $\label{subsec:G2_Hecke4_GZ_renorm}}For $N\in\{1$, $2$, $3$, $4\}$, the Gross--Zagier renormalized Green's function (also known as ``automorphic self-energy'') is defined by
 ~(see \cite[][Chap.~II,  Eq.~5.7]{GrossZagierI} or \cite[][Eq.~3.2.1]{AGF_PartI})

\begin{align}
G_2^{\mathfrak H/\overline{\varGamma}_0(N)}(z):={}&-2\sum_{\hat  \gamma\in\overline {\varGamma}_0(N),\hat \gamma z\neq z}Q_{1}
\left( 1+\frac{\vert z -\hat  \gamma z\vert ^{2}}{2\I z\I(\hat\gamma z)} \right)-2\left\{ \log\left\vert 4\pi\eta^{4}(z)  \I z\right|-1 \right\},\label{eq:G_2_HeckeN_GZ_renorm}
\end{align} so long as  $ z$ corresponds to  a non-elliptic point of $ \varGamma_0(N)$. As $\varGamma_0(4)$ contains no elliptic points, Eq.~\ref{eq:G_2_HeckeN_GZ_renorm} is applicable to all $ z\in\mathfrak H$ for $N=4$. For $ N\in\{2,3\}$, Eq.~\ref{eq:G_2_HeckeN_GZ_renorm} fails when $ z$ is an elliptic point, where $ \alpha_N(z)=\infty$; for $N=1$,  Eq.~\ref{eq:G_2_HeckeN_GZ_renorm} breaks down when $ z$ is an elliptic point satisfying $j(z)[j(z)-1728]=0$. The automorphic self-energies at these elliptic points have already been computed in \cite[][Theorem 1.2.2(a)]{AGF_PartI}.

Hereafter,  we will   only consider  automorphic self-energies at non-elliptic points, where Eq.~\ref{eq:G_2_HeckeN_GZ_renorm} can be reformulated as \cite[][Eq.~3.2.3]{AGF_PartI} \begin{align}
G_2^{\mathfrak H/\overline{\varGamma}_0(N)}(z)=\lim_{z'\to z}\left[G_2^{\mathfrak H/\overline{\varGamma}_0(N)}(z,z')-2\log|z-z'|\right]-2\log\left|2\pi\eta^{4}(z)  \right|.\label{eq:G_2_HeckeN_GZ_renorm_alt}
\end{align}
To facilitate quantitative analysis of the weight-4 automorphic self-energies of levels 2, 3 and 4, we rewrite Eq.~\ref{eq:G_2_HeckeN_GZ_renorm_alt} using the entropy formulae (Eqs.~\ref{eq:G2_Hecke234_via_I} and \ref{eq:I_N_via_H_nu}), in the next lemma.

\begin{lemma}[Integral Representations of  Automorphic Self-Energies $G_2^{\mathfrak H/\overline{\varGamma}_0(2)}(z )$, $ G_2^{\mathfrak H/\overline{\varGamma}_0(3)}(z )$ and $G_2^{\mathfrak H/\overline{\varGamma}_0(4)}(z )$] \label{lm:G_2_Hecke234_int_repn}

For $z\in\Int\mathfrak D_N$ where $ N\in\{2,3,4\}$,  we have \begin{align}
G_2^{\mathfrak H/\overline{\varGamma}_0(N)}(z)={}&-\frac{\sin^{2}(\nu\pi)}{2\pi^2}\R\left\{\frac{(\I z)^{2}}{3}
\left(\frac{\partial}{\partial\I z}\frac{1}{\I z}\right)^2\frac{z^{2}\left[\frac{\mathfrak Q_{\nu}({\alpha_N(z))}}{P_\nu(1-2\alpha_{N}(z))}+\frac{\mathfrak Q_{\nu}({1-\alpha_N(z))}}{P_\nu(2\alpha_{N}(z)-1)}\right]}{P_\nu(1-2\alpha_{N}(z))P_\nu(2\alpha_{N}(z)-1)}\right\}\notag\\{}&+2\log\left\vert \frac{\partial\alpha_{N}(z)/\partial z}{2\pi\eta^{4}(z)} \right\vert-\frac{4}{3}\log|\alpha_{N}(z)[1-\alpha_{N}(z)]|+\frac{2}{3}[2\gamma_0+\psi^{(0)}(-\nu)+\psi^{(0)}(\nu+1)],\label{eq:G_2_Hecke234_renorm_int_repn}
\end{align}where $ \mathfrak Q_\nu(\alpha)$ is defined in Eq.~\ref{eq:frakQ_nu_alpha_defn} and $ N=4\sin^2(\nu\pi)$ for $ \nu\in\{-1/4,-1/3,-1/2\}$.
\end{lemma}
\begin{proof}
In view of Eq.~\ref{eq:I_N_via_H_nu}, we need to analyze the asymptotic behavior of $H_\nu(\alpha_N(z)\Vert\alpha_N(z'))+H_\nu(1-\alpha_N(z)\Vert1-\alpha_N(z'))$ as $ z'$ approaches $z$. By Eq.~\ref{eq:S_nu_H_nu}, we have the following identity for  $\alpha,1-\alpha,\beta,1-\beta,\frac{\beta(1-\alpha)}{\alpha(1-\beta)}\in\mathbb C\smallsetminus[1,+\infty)$:\begin{align}
&
P_\nu(1-2\alpha)P_\nu(2\alpha-1)P_\nu(1-2\beta)P_\nu(2\beta-1)H_\nu(\alpha\Vert\beta)\equiv A_\nu(\alpha\Vert\beta)+M_\nu(\alpha\Vert\beta)\pmod{2\pi  i\Lambda_\nu(\alpha,\beta)},\label{eq:S_nu_H_nu_pm_log}
\end{align}where\begin{align}
A_\nu(\alpha\Vert\beta):={}&[P_{\nu }(2\beta-1)]^2\int_0^{\beta}\frac{[P_\nu(1-2t)]^2-[P_\nu(1-2\alpha)]^2}{t-\alpha}\D t\notag\\{}&-P_\nu(1-2\beta)P_\nu(2\beta-1)\int_0^{\beta}\frac{P_\nu(1-2t)P_\nu(2t-1)-P_\nu(1-2\alpha)P_\nu(2\alpha-1)}{t-\alpha}\D t
\end{align}remains analytic as $\beta$ approaches $\alpha$, and\begin{align}
M_\nu(\alpha\Vert\beta):=\left\{[P_{\nu }(2\beta-1)]^2[P_\nu(1-2\alpha)]^2-P_\nu(1-2\beta)P_\nu(2\beta-1)P_\nu(1-2\alpha)P_\nu(2\alpha-1)\right\}\log\left( 1-\frac{\beta}{\alpha} \right)
\end{align}has a mild singularity of order $ O(|\beta-\alpha|\log|\beta-\alpha|)$, in the $ \beta\to\alpha$ limit.
No matter which  logarithmic branch is chosen  in the definition of $M_\nu(\alpha\Vert\beta)$, the  identity in  Eq.~\ref{eq:S_nu_H_nu_pm_log} remains valid, modulo $ 2\pi  i\Lambda_\nu(\alpha,\beta)$.

First we consider the  mildly singular portion $ M_\nu(\alpha\Vert\beta)$. As we have (cf.~Eq.~\ref{eq:Pnu_ratio_to_z})\begin{align}&
\frac{M_\nu({\alpha_{N}(z)\|\alpha_{N}(z'))+M_\nu(1-{\alpha_{N}(z)\|1-\alpha_{N}(z'))}}}{P_\nu(1-2\alpha_{N}(z))P_\nu(2\alpha_{N}(z)-1)P_\nu(1-2\alpha_{N}(z'))P_\nu(2\alpha_{N}(z')-1)}\notag\\={}&\left( \frac{z'}{z}-1 \right)\log\left( 1-\frac{\alpha_{N}(z')}{\alpha_{N}(z)} \right)+\left( \frac{z}{z'}-1 \right)\log\left( 1-\frac{1-\alpha_{N}(z')}{1-\alpha_{N}(z)} \right)
\end{align}for $N=4\sin^2(\nu\pi)$ and $ z,z'\in\Int\mathfrak D_N$, the singular part contributes to the asymptotic behavior $ G_2^{\mathfrak H/\overline{\varGamma}_0(N)}(z,z'),z'\to z$ as follows (cf.~Eq.~\ref{eq:I_N_via_H_nu}):
{\allowdisplaybreaks\begin{align}
{}&\R\left\{\I z\I z'
\frac{\partial}{\partial\I z}\frac{\partial}{\partial\I z'}\frac{zz'\left[ \left( \frac{z'}{z}-1 \right)\log\left( 1-\frac{\alpha_{N}(z')}{\alpha_{N}(z)} \right )+\left( \frac{z}{z'}-1 \right)\log\left( 1-\frac{1-\alpha_{N}(z')}{1-\alpha_{N}(z)} \right)\right]}{\I z\I z'}\right\}\notag\\={}&3+\R\left\{\frac{z\partial\alpha_{N}(z)/\partial z}{\alpha_{N}(z)[1-\alpha_{N}(z)]}\right\}+2\log|z-z'|+\log\left\vert \frac{[\partial\alpha_{N}(z)/\partial z]^2}{\alpha_{N}(z)[1-\alpha_{N}(z)]} \right\vert+O(|z-z'|\log|z-z'|).\label{eq:L_nu_contrib}
\end{align}}

Then we move on to the treatment of the analytic part  $ A_\nu(\alpha\Vert\beta)$.
With Taylor expansions in the form of   \begin{align}
\int_0^{\beta}\frac{\varphi(t)-\varphi(\alpha)}{t-\alpha}\D t=\int_0^{\alpha}\frac{\varphi(t)-\varphi(\alpha)}{t-\alpha}\D t+\varphi'(\alpha)(\beta-\alpha)+\frac{\varphi''(\alpha)}{4}(\beta-\alpha)^2+O((\beta-\alpha)^3)
\end{align}and  the following relation \cite[cf.][Eqs.~2.1.4 and 2.2.17]{AGF_PartI}:\begin{align}
\frac{\partial \alpha_{N}(z)}{\partial z}=2\pi i\alpha_N(z)[1-\alpha_N(z)][P_\nu(1-2\alpha_{N}(z))]^{2},\quad \forall z\in\Int\mathfrak D_N,
\end{align} we see that  the net contribution from the analytic part \begin{align}
&\frac{A_\nu({\alpha_{N}(z)\|\alpha_{N}(z'))+A_\nu(1-{\alpha_{N}(z)\|1-\alpha_{N}(z'))}}}{P_\nu(1-2\alpha_{N}(z))P_\nu(2\alpha_{N}(z)-1)P_\nu(1-2\alpha_{N}(z'))P_\nu(2\alpha_{N}(z')-1)}
\end{align}   to $ G_2^{\mathfrak H/\overline{\varGamma}_0(N)}(z,z'),z'\to z$ amounts to{\allowdisplaybreaks
\begin{align}
{}&\R\left\{\I z\I z'
\frac{\partial}{\partial\I z}\frac{\partial}{\partial\I z'}\frac{zz'\left[\frac{\sqrt{N}z'}{i}f_{\nu}(\alpha_{N}(z))-g_{\nu}(\alpha_{N}(z)) +\frac{i}{\sqrt{N}z'}f_{\nu}(1-\alpha_{N}(z))-g_{\nu}(1-\alpha_{N}(z))\right]}{\I z\I z'}\right\}\notag\\{}&+\frac{3\sin^2(\nu\pi)}{2\pi^{2}}\R\left\{ zz'\frac{\partial^{2}}{\partial z\partial z'} \frac{[\alpha_{N}(z)-\alpha_{N}(z')]^{2}}{[\alpha_{N}(z)]^{2}[1-\alpha_{N}(z)]^{2}[P_\nu(1-2\alpha)]^2[P_\nu(2\alpha-1)]^{2}}\right\}+O(z-z')
\notag\\={}&\frac{(\I z)^{2}}{3}
\left(\frac{\partial}{\partial\I z}\frac{1}{\I z}\right)^2\R\left\{z^{2}\left[\frac{\sqrt{N}z}{i}f_{\nu}(\alpha_{N}(z))-g_{\nu}(\alpha_{N}(z)) +\frac{i}{\sqrt{N}z}f_{\nu}(1-\alpha_{N}(z))-g_{\nu}(1-\alpha_{N}(z))\right]\right\}\notag\\{}&-\frac{1}{3}\R\left\{\left[ 1-\left(z \frac{\partial}{\partial z} \right)^2 \right]\left[\frac{\sqrt{N}z}{i}f_{\nu}(\alpha_{N}(z))-g_{\nu}(\alpha_{N}(z)) +\frac{i}{\sqrt{N}z}f_{\nu}(1-\alpha_{N}(z))-g_{\nu}(1-\alpha_{N}(z))\right]\right\}\notag\\{}&-\R\left\{ z\frac{\partial}{\partial z} \left[ \frac{\sqrt{N}z}{i}f_{\nu}(\alpha_{N}(z))- \frac{i}{\sqrt{N}z}f_{\nu}(1-\alpha_{N}(z))\right]\right\}\notag\\{}&+\R\left[\frac{\sqrt{N}z}{i}f_{\nu}(\alpha_{N}(z))+ \frac{i}{\sqrt{N}z}f_{\nu}(1-\alpha_{N}(z))\right]-3+O(z-z'),\label{eq:A_nu_contrib}
\end{align}}where $ f_\nu$ and $ g_\nu$ were defined in  Eqs.~\ref{eq:f_nu_alpha_defn} and \ref{eq:g_nu_alpha_defn}.

Adding up  Eqs.~\ref{eq:L_nu_contrib} and \ref{eq:A_nu_contrib} in the $z'\to z$ limit, while referring back to   Eqs.~\ref{eq:f-g} and  \ref{eq:g-g}, we obtain\begin{align}
G_2^{\mathfrak H/\overline{\varGamma}_0(N)}(z)={}&-\frac{\sin^{2}(\nu\pi)}{6\pi^2}\R\left\{(\I z)^{2}
\left(\frac{\partial}{\partial\I z}\frac{1}{\I z}\right)^2\frac{z^{2}\left[\frac{\mathfrak Q_{\nu}({\alpha_N(z))}}{P_\nu(1-2\alpha_{N}(z))}+\frac{\mathfrak Q_{\nu}({1-\alpha_N(z))}}{P_\nu(2\alpha_{N}(z)-1)}\right]}{P_\nu(1-2\alpha_{N}(z))P_\nu(2\alpha_{N}(z)-1)}\right\}\notag\\{}&+2\log\left\vert \frac{\partial\alpha_{N}(z)/\partial z}{2\pi\eta^{4}(z)} \right|-\log|\alpha_{N}(z)[1-\alpha_{N}(z)]\vert+\mathfrak L_\nu(\alpha_N(z)),\label{eq:G2RN_via_frakL}\end{align}for \begin{align} \mathscr L_\nu(\alpha):={}&\frac{\sin^{2}(\nu\pi)}{6\pi^2}( 1- \widehat D_\alpha^2)\frac{\frac{\mathfrak Q_{\nu}({\alpha})}{P_\nu(1-2\alpha)}+\frac{\mathfrak Q_{\nu}({1-\alpha})}{P_\nu(2\alpha-1)}}{P_\nu(1-2\alpha)P_\nu(2\alpha-1)}+\frac{\sin^{2}(\nu\pi)}{2\pi^2} \widehat D_\alpha\frac{\frac{\mathfrak Q_{\nu}({\alpha})}{P_\nu(1-2\alpha)}-\frac{\mathfrak Q_{\nu}({1-\alpha})}{P_\nu(2\alpha-1)}}{P_\nu(1-2\alpha)P_\nu(2\alpha-1)}\notag\\{}&+\frac{P_\nu(2\alpha-1)}{P_\nu(1-2\alpha)}f_{\nu}(\alpha)+ \frac{P_\nu(1-2\alpha)}{P_\nu(2\alpha-1)}f_{\nu}(1-\alpha),\quad \text{and}\quad \mathfrak L_\nu(\alpha):=\R \mathscr L_\nu(\alpha).\label{eq:frakL_nu_defn}
\end{align}(Here, we refer the readers to   Eq.~\ref{eq:op_D_alpha} for the differential operator $ \widehat D_\alpha$.) Clearly, $ \mathfrak L_\nu(\alpha)$ is well-defined for $ \alpha\in(\mathbb C\smallsetminus\mathbb R)\cup(0,1)$ and $ \nu\in\{-1/4,-1/3,-1/2\}$, as the Legendre functions in the  denominators never assume zero values. One can also continuously extend $ \mathfrak L_\nu(\alpha)$ to $ \alpha\in\mathbb C\smallsetminus\{0,1\}$, because $ \mathscr L_\nu(x+i0^+)= \overline{\mathscr L_\nu(x-i0^+)}$ for $ x\in(-\infty,0)\cup(1,+\infty)$.

Finally, we point out  that \begin{align}
 \mathfrak L_\nu(\alpha)-\frac{2}{3}[2\gamma_0+\psi^{(0)}(-\nu)+\psi^{(0)}(\nu+1)]+\frac{1}{3}\log|\alpha(1-\alpha)|
\end{align}vanishes in the limits of $ \alpha\to0$ and $1$  (which can be checked by the asymptotic formulae in Lemma~\ref{lm:cusp_limits}), and the same expression remains bounded as $ \alpha\to\infty$ (which can be shown by a modest variation on \cite[][Proposition 2.1.1]{AGF_PartI}).
 This leads to a closed-form evaluation of $  \mathfrak L_\nu(\alpha)$, because a bounded harmonic function in the entire plane that vanishes at  certain points must vanish  identically.
\end{proof}
\begin{remark}For  $ N=4\sin^2(\nu\pi)$ with $ \nu\in\{-1/4,-1/3,-1/2\}$, one may employ \cite[][Eqs.~2.1.8 and 2.1.28]{AGF_PartI} to further simplify the following expression occurring in Eq.~\ref{eq:G_2_Hecke234_renorm_int_repn}:\begin{align}
2\log\left\vert \frac{\partial\alpha_{N}(z)/\partial z}{2\pi\eta^{4}(z)} \right\vert-\frac{4}{3}\log|\alpha_{N}(z)[1-\alpha_{N}(z)]|+\frac{2}{3}[2\gamma_0+\psi^{(0)}(-\nu)+\psi^{(0)}(\nu+1)].
\end{align}The results are\begin{align}
\quad \frac{1}{3}\log\frac{|\alpha_2(z)|}{2^6} ,\quad \frac{1}{3}\log\frac{|\alpha_3(z)|}{3^{3}|1-\alpha_3(z)|},\quad\text{and}\quad -\frac{1}{3}\log\frac{2^4|1-\alpha_{4}(z)|^{2}}{|\alpha_{4}(z)|}\label{eq:harmonic_redn}
\end{align}for $ N=2,3$ and $4$, respectively.\eor\end{remark}
With the lemma above, we will be able to reprove one of   the main results in Part I  \cite[][Theorem 1.2.2(b)]{AGF_PartI}, as recapitulated in the next proposition.
\begin{proposition}[Automorphic Self-Energy of Weight 4 and Level 4] \label{prop:self_intn_Hecke4}The weight-4, level-4 automorphic self-energy $ G_2^{\mathfrak H/\overline{\varGamma}_0(4)}(z)$ can be evaluated in closed form:\begin{align}
G_2^{\mathfrak H/\overline{\varGamma}_0(4)}(z)={}&-\frac{1}{3}\log\frac{2^4|1-\alpha_{4}(z)|^{2}}{|\alpha_{4}(z)|}
=-8\log\left\vert \frac{\eta(z)}{\eta(2z)} \right\vert,\quad \forall z\in\mathfrak H.\label{eq:G_2_z_Hecke4_eval}\end{align}
\end{proposition}\begin{proof}One may directly compute that\begin{align}\mathfrak Q_{-1/2}({\alpha}):={}&2\int_0^1\left[\int_W^1\frac{\D V}{\mathbb Y_{-1/2,\alpha}(V)}\right]^2\frac{\D W}{\mathbb Y_{-1/2,\alpha}(W)}=\frac{2\pi^{3}}{3}[P_{-1/2}(1-2\alpha)]^3,
\end{align}so the first term on the right-hand side of Eq.~\ref{eq:G_2_Hecke234_renorm_int_repn} vanishes. Using the last item in Eq.~\ref{eq:harmonic_redn}, we subsequently turn  Eq.~\ref{eq:G_2_Hecke234_renorm_int_repn} into Eq.~\ref{eq:G_2_z_Hecke4_eval}.\end{proof}
The  evaluation of $ G_2^{\mathfrak H/\overline{\varGamma}_0(4)}(z)$ in the proposition above will become useful in the following analysis of $  G_2^{\mathfrak H/\overline{\varGamma}_0(1)}(z)$.
\begin{lemma}[Integral Representation for $ G_2^{\mathfrak H/\overline{\varGamma}_0(1)}(z)$]\label{lm:G_2_PSL2Z_int_repn}For $ z\in(\Int\mathfrak D_1)\smallsetminus\{i\}$, we have \begin{align}
G_2^{\mathfrak H/\overline{\varGamma}_0(1)}(z)={}&G_{2}^{\mathfrak H/\overline{\varGamma}_0(4)}\left( -\frac{1}{2z} ,\frac{z}{2}\right) +G_{2}^{\mathfrak H/\overline{\varGamma}_0(4)}\left( \frac{z+1}{2} ,\frac{z}{2}\right)+ G_{2}^{\mathfrak H/\overline{\varGamma}_0(4)}\left( \frac{z}{2(z+1)} ,\frac{z}{2}\right)\notag\\{}&+2G_{2}^{\mathfrak H/\overline{\varGamma}_0(4)}\left( -\frac{1}{2(z+1)} ,\frac{z}{2}\right) -2\log2.\label{eq:G_2_PSL2Z_RN_decomp}
\end{align}Meanwhile, there exist two analytic functions $ \mathfrak q(z)$ and $ \mathfrak Q(z)$ such that \begin{align}
G_{2}^{\mathfrak H/\overline{\varGamma}_0(4)}\left( -\frac{1}{2(z+1)} ,\frac{z}{2}\right) ={}&\R\left[(\I z)^{2}
\left(\frac{\partial}{\partial\I z}\frac{1}{\I z}\right)^2 \mathfrak q(z)\right],\label{eq:frakq_deriv}\\G_2^{\mathfrak H/\overline{\varGamma}_0(1)}(z)={}&\R\left[(\I z)^{2}
\left(\frac{\partial}{\partial\I z}\frac{1}{\I z}\right)^2 \mathfrak Q(z)\right]-\frac{\log|j(z)-1728|}{3}.\label{eq:frakQ_deriv}
\end{align}\end{lemma}\begin{proof}Taking the $ z'\to z$ limit in the addition formula \cite[cf.][Eq.~2.2.6]{AGF_PartI} \begin{align}
G_2^{\mathfrak H/\overline{\varGamma}_0(1)}(z,z')={}&G_2^{\mathfrak H/\overline{\varGamma}(2)}(z,z')+G_2^{\mathfrak H/\overline{\varGamma}(2)}\left(-\frac{1}{z},z'\right)+G_2^{\mathfrak H/\overline{\varGamma}(2)}\left(-\frac{z+1}{z},z'\right)+G_2^{\mathfrak H/\overline{\varGamma}(2)}\left(\frac{z\vphantom{1}}{z+1},z'\right)\notag\\& +G_2^{\mathfrak H/\overline{\varGamma}(2)}(z+1,z')+G_2^{\mathfrak H/\overline{\varGamma}(2)}\left(-\frac{1}{z+1},z'\right),
\end{align}while  referring to Eqs.~\ref{eq:frakG_defn}, \ref{eq:G_2_HeckeN_GZ_renorm_alt} and \ref{eq:G_2_z_Hecke4_eval}, we arrive at Eq.~\ref{eq:G_2_PSL2Z_RN_decomp}.

A special case for Eq.~\ref{eq:G2Hecke234_Pnu1} reads\begin{align}&
G_{2}^{\mathfrak H/\overline{\varGamma}_0(4)}\left( -\frac{1}{2(z\pm1)} ,\frac{z}{2}\right)+\frac{(\I z)^{2}}{3}\left(\frac{\partial}{\partial\I z}\frac{1}{\I  z}\right)^2\R\left[z(z\pm1)\varPhi\left( -\frac{1}{z},z\pm1 \right)\right]\notag\\={}&\frac{1}{3}\R\left\{\left[1-(z\pm2)\frac{\partial}{\partial z}-z(z\pm1)\frac{\partial^{2}}{\partial z^{2}}\right]\varPhi\left( -\frac{1}{z},z\pm1 \right)\right\}\pm\R\left[\left.\frac{\partial}{\partial w}\right|_{w=z\pm1}\varPhi\left( -\frac{1}{z},w \right)\right]\notag\\{}&+\R\left[\frac{z\pm1}{z}\left.\frac{\partial^{2}}{\partial w\partial w'}\right|_{w=-1/z,w'=z\pm1}\varPhi(w,w')\right],\label{eq:G2Hecke4_G3PSL2Z_corr}
\end{align}where $ \varPhi(w,w')=\mathscr J(\lambda(w)\Vert\lambda(w'))+\mathscr J(1-\lambda(w)\Vert1-\lambda(w'))$ (cf.~Eq.~\ref{eq:J_lambda_mu}), so long as $ z/2$ and $ (z\pm1)/2$ both reside in $ \Int\mathfrak D_4$. As $ z\to i\infty$, one can use Eqs.~\ref{eq:J_lambda_0}--\ref{eq:J_lambda_1} and their partial derivatives to check that the right-hand side of Eq.~\ref{eq:G2Hecke4_G3PSL2Z_corr} vanishes.
Relying on the asymptotic methods  in \cite[][Proposition 2.1.1]{AGF_PartI}, one may show that the right-hand side of  Eq.~\ref{eq:G2Hecke4_G3PSL2Z_corr} remains bounded as $z\to(i\sqrt{3}\mp1)/2$. Furthermore, the branch cut of $ \mathbf K(\sqrt{t}),t\in\mathbb C\smallsetminus[1,+\infty)$ does not affect the continuous extension of Eq.~\ref{eq:G2Hecke4_G3PSL2Z_corr}. Thus, the left-hand side of Eq.~\ref{eq:G2Hecke4_G3PSL2Z_corr} can be identified with a bounded (hence constant) harmonic function on the compact Riemann surface  $ SL(2,\mathbb Z)\backslash\mathfrak H^*$, which vanishes at the cusp. This proves Eq.~\ref{eq:frakq_deriv}.

With \cite[][Eqs.~2.0.4, 3.4.26, 3.4.27]{AGF_PartI}, we are able to verify that\begin{align}{}&
G_{2}^{\mathfrak H/\overline{\varGamma}_0(4)}\left( -\frac{1}{2z} ,\frac{z}{2}\right) +G_{2}^{\mathfrak H/\overline{\varGamma}_0(4)}\left( \frac{z+1}{2} ,\frac{z}{2}\right)+ G_{2}^{\mathfrak H/\overline{\varGamma}_0(4)}\left( \frac{z}{2(z+1)} ,\frac{z}{2}\right)\notag\\={}&\left[ G_2^{\mathfrak H/\overline{\varGamma}_{0}(2)}\left(\frac{z-1}{2}\right) - G_2^{\mathfrak H/\overline{\varGamma}_{0}(4)}\left(\frac{z}{2}\right)\right]+\left[ G_2^{\mathfrak H/\overline{\varGamma}_{0}(2)}\left( -\frac{1}{z} \right)- G_2^{\mathfrak H/\overline{\varGamma}_{0}(4)}\left(-\frac{1}{2z}\right)\right]\notag\\{}&+\left[ G_2^{\mathfrak H/\overline{\varGamma}_{0}(2)}(z)- G_2^{\mathfrak H/\overline{\varGamma}_{0}(4)}\left(\frac{z}{2}\right)\right]
+\frac{1}{3}\log\frac{2^4}{|1-\lambda(z)|^{4}|\lambda(z)|}.
\end{align} In view of Eqs.~\ref{eq:G_2_Hecke234_renorm_int_repn} and \ref{eq:G_2_z_Hecke4_eval}, as well as the relation $ \alpha_2(z)=1-[1-2\lambda(2z+1)]^{-2}$ \cite[][Eq.~2.1.22]{AGF_PartI}, we convert the last displayed equation into the form of \begin{align}
\R\left\{(\I z)^{2}
\left(\frac{\partial}{\partial\I z}\frac{1}{\I z}\right)^2 \widetilde{\mathfrak Q}(z)\right\}-\frac{\log|j(z)-1728|}{3}+2\log2,
\end{align} where $  \widetilde{\mathfrak Q}(z)$ is analytic. Adding up with Eqs.~\ref{eq:G_2_PSL2Z_RN_decomp} and \ref{eq:frakq_deriv}, we arrive at Eq.~\ref{eq:frakQ_deriv}.\end{proof}
One can now use Lemmata~\ref{lm:G_2_Hecke234_int_repn} and \ref{lm:G_2_PSL2Z_int_repn} to establish some algebraic relations between weight-4 automorphic self-energies and
certain special cases of weight-6 automorphic Green's functions, as revealed in the proposition below.

\begin{proposition}[Special Correspondence Between Automorphic Green's Functions of Weights 4 and 6]
So long as $z$ is not an elliptic point on $ \varGamma_0(N)$ for $ N\in\{1,2,3\}$, we have\begin{align}
G_{2}^{\mathfrak H/\overline{\varGamma}_0(1)}(z)+\frac{\log|j(z)-1728|}{3}={}&\frac{2}{3}G_3^{\mathfrak H/\overline{\varGamma}_0(1)}(i,z)+\frac{2}{3}G_3^{\mathfrak H/\overline{\varGamma}_0(1)}\left(\frac{1+i\sqrt{3}}{2},z\right),\label{eq:G2_PSL2Z_weight_lift}\\
G_{2}^{\mathfrak H/\overline{\varGamma}_0(2)}(z)+G_{2}^{\mathfrak H/\overline{\varGamma}_0(2)}\left( -\frac{1}{2z} \right)-\frac{1}{3}\log\frac{|\alpha_{2}(z)[1-\alpha_2(z)]|}{2^{12}}={}&\frac{4}{3}G_3^{\mathfrak H/\overline{\varGamma}_{0}(2)}\left(\frac{1+i}{2},z\right),\label{eq:G2_Hecke2_weight_lift}\\G_{2}^{\mathfrak H/\overline{\varGamma}_0(3)}(z)+G_{2}^{\mathfrak H/\overline{\varGamma}_0(3)}\left( -\frac{1}{3z} \right)+2\log 3={}&\frac{4}{3}G_3^{\mathfrak H/\overline{\varGamma}_{0}(3)}\left(\frac{3+i\sqrt{3}}{6},z\right).\label{eq:G2_Hecke3_weight_lift}
\end{align}Furthermore, the right-hand side of Eq.~\ref{eq:G2_PSL2Z_weight_lift} has the following breakdown:\begin{align}G_3^{\mathfrak H/\overline{\varGamma}_0(1)}(i,z)={}&\frac{3}{2}\left[ G_{2}^{\mathfrak H/\overline{\varGamma}(2)}\left( -\frac{1}{z} ,z\right) +G_{2}^{\mathfrak H/\overline{\varGamma}(2)}\left( z+1 ,z\right)+ G_{2}^{\mathfrak H/\overline{\varGamma}(2)}\left( \frac{z}{z+1} ,z\right)\right]\notag\\{}&+\frac{\log|j(z)-1728|}{2}-3\log2,\label{eq:G3_PSL2Z_i_add_form}\quad \\
G_3^{\mathfrak H/\overline{\varGamma}_0(1)}\left(\frac{1+i\sqrt{3}}{2},z\right)={}&3G_{2}^{\mathfrak H/\overline{\varGamma}(2)}\left( -\frac{1}{z+1} ,z\right).\label{eq:G3_PSL2Z_rho_add_form}
\end{align}\end{proposition}\begin{proof}It is easy to appreciate that both sides of  Eqs.~\ref{eq:G2_PSL2Z_weight_lift}--\ref{eq:G2_Hecke3_weight_lift} share the same symmetry and asymptotic behavior (as $ \alpha_N(z)$ approaches $0$, $1$ or $\infty$, for $ N\in\{1,2,3\}$). Furthermore, in view of  Lemmata~\ref{lm:G_2_Hecke234_int_repn} and \ref{lm:G_2_PSL2Z_int_repn}, the left-hand side for each of these three formulae can be rewritten in the form of\begin{align}
(\I z)^{2}
\left(\frac{\partial}{\partial\I z}\frac{1}{\I z}\right)^2\R F(z),
\end{align} where $F(z) $ is an analytic function. Such an expression is annihilated by the differential operator\begin{align}
\Delta_{\mathfrak H}-6:=(\I z)^{2}\left[\frac{\partial^{2}}{\partial(\R z)^{2}}+\frac{\partial^{2}}{\partial(\I z)^{2}}\right]-6,
\end{align} as is the right-hand side of each aforementioned formula.

Thus, for every formula among Eqs.~\ref{eq:G2_PSL2Z_weight_lift}--\ref{eq:G2_Hecke3_weight_lift}, the difference between the two sides defines a bounded function on a compact Riemann surface $ X_0(N)(\mathbb C)=\varGamma_0(N)\backslash\mathfrak H^*$ (for $ N\in\{1,2,3\}$), which is annihilated by $ \Delta_{\mathfrak H}-6$, and is vanishing at the cusps. Such a bounded function must be identically zero \cite[Lemma 2.0.1]{AGF_PartI}.

In a similar vein, one can prove Eq.~\ref{eq:G3_PSL2Z_rho_add_form} by Eq.~\ref{eq:frakq_deriv} and asymptotic analysis. Subtracting Eq.~\ref{eq:G3_PSL2Z_rho_add_form} from Eq.~\ref{eq:G2_PSL2Z_weight_lift}, and referring back to Eq.~\ref{eq:G_2_PSL2Z_RN_decomp}, we arrive at Eq.~\ref{eq:G3_PSL2Z_i_add_form}    in their wake.    \end{proof}
\begin{remark}Plugging the results $ G_{2}^{\mathfrak H/\overline{\varGamma}_0(2)}(i/\sqrt{2})=-3\log2$ and $ G_{2}^{\mathfrak H/\overline{\varGamma}_0(3)}(i/\sqrt{3})=-2\log\frac{3}{\sqrt[3]{4}}$  \cite[][Theorem 1.2.2(a)]{AGF_PartI} into Eqs.~\ref{eq:G2_Hecke2_weight_lift} and \ref{eq:G2_Hecke3_weight_lift}, we are able to  compute the following special values of weight-6 automorphic Green's functions:\begin{align}
G_3^{\mathfrak H/\overline{\varGamma}_{0}(2)}\left(\frac{1+i}{2},\frac{i}{\sqrt{2}}\right)={}&-\log2,\\G_3^{\mathfrak H/\overline{\varGamma}_{0}(3)}\left(\frac{3+i\sqrt{3}}{6},\frac{i}{\sqrt{3}}\right)={}&2\log2-\frac{3\log 3}{2}.
\end{align}One may wish to check these evaluations  against the integral representations of these weight-6 automorphic Green's functions \cite[cf.][Eqs.~2.3.19 and 2.3.20]{AGF_PartI}:\begin{align}
G_3^{\mathfrak H/\overline{\varGamma}_{0}(2)}\left(\frac{1+i}{2},\frac{i}{\sqrt{2}}\right)={}&\frac{\pi^2}{8}\int_0^1\xi[P_{-1/4}(\xi)]^2\{[P_{-1/4}(\xi)]^{2}-[P_{-1/4}(-\xi)]^{2}\}\D\xi,\\G_3^{\mathfrak H/\overline{\varGamma}_{0}(3)}\left(\frac{3+i\sqrt{3}}{6},\frac{i}{\sqrt{3}}\right)={}&\frac{\pi^2}{27}\int_0^1\xi[P_{-1/3}(\xi)]^2\{[P_{-1/3}(\xi)]^{2}-[P_{-1/3}(-\xi)]^{2}\}\D\xi,
\end{align} and a generic integral identity \cite[][Eq.~1.2]{Zhou2016Int4Pnu} \begin{align}
\int_{0}^1
x[P_\nu(x)]^4\D x-\int_{0}^1
x[P_\nu(x)]^2 [P_\nu(-x)]^{2}\D x={}&\frac{4\sin^2(\nu\pi)}{(2\nu+1)^{2}\pi^{2}}\left[ \frac{\psi^{(0)}(\nu+1)+\psi^{(0)}(-\nu)}{2}+ \gamma_0+2\log2\right]
,\end{align}which is applicable to $ \nu\in\mathbb C$ (with the understanding that limits are taken for the right-hand side when $ \nu\in\{-1/2\}\cup\mathbb Z$).\eor\end{remark}

\begin{table}[t]\scriptsize\caption{Some special values of  $j$-invariants and automorphic Green's functions}\label{tab:class_number_1_examples}

\begin{align*}\begin{array}{r@{\;\in\;}l|l|l|l|l}\hline\hline  \vphantom{\dfrac{\frac12}{\frac12}}z&\mathfrak Z_{D}& j(z)-1728&e^{-4(\I z)^2G_2^{\mathfrak H/\overline{\varGamma}_0(1)}(z)}&e^{-4(\I z)^2G_3^{\mathfrak H/\overline{\varGamma}_0(1)}(i,z)}\!\!\!\!\!&e^{-\frac{2(\I z)^2}{3}G_3^{\mathfrak H/\overline{\varGamma}_0(1)}\left(\frac{1+i\sqrt{3}}{2},z\right)}\!\!\!\!\!\\\hline  \vphantom{\dfrac{\frac\int2}{\frac12}} \dfrac{1+i\sqrt{7}}{2}&\mathfrak Z_{-7}&-3^6 7&\dfrac{1}{3^{14}5^{12}7^7}&3^67^7&\dfrac{5^{3}}{3}\\[6pt]i\sqrt{2}&\mathfrak Z_{-8}&2^{7}7^2&\left(\dfrac{1}{2^{6}5^27^{3}}\right)^4&(2^47^{5})^{2}&(5)^{2}\\[8pt] \dfrac{1+i\sqrt{11}}{2}&\mathfrak Z_{-11}&-2^67^211&\dfrac{1}{2^{42}7^611^{11}}&\dfrac{11^{11}}{7^2}&2^{5}\\[5pt]i\sqrt{3}&\mathfrak Z_{-12}&2^4 3^3 11^2&\left(\dfrac{5^2}{2^{4}3^{6}11^{5}}\right)^4&\left(\dfrac{11^{9}}{3^{9}}\right)^2&\left(\dfrac{3^{3}}{5}\right)^2\\[8pt]i\sqrt{4}&\mathfrak Z_{-16}&2^33^67^2&\left(\dfrac{7}{2^{4}3^{8}11^{6}}\right)^4&\left(\dfrac{3^{30}}{7^{11}}\right)^2&\left(\dfrac{11^{3}}{3^{5}}\right)^2\\[8pt]\dfrac{1+i\sqrt{19}}{2}&\mathfrak Z_{-19}&-2^6 3^619&\dfrac{1}{2^{26}3^{38}19^{19}}&\dfrac{19^{19}}{3^{30}}&\dfrac{3^{5}}{2^{3}}\\[8pt]\dfrac{1+i3\sqrt{3}}{2}&\mathfrak Z_{-27}&-2^6 3^111^2 23^2&\dfrac{3^911^{10}}{2^{10}5^{52}23^{38}}&\dfrac{3^{27}23^{30}}{11^{42}}&\dfrac{5^{13}}{2^{11}3^9}\\[6pt]i\sqrt{7}& \mathfrak Z_{-28}&3^8 7^1 19^2&\left(\dfrac{5^{18}}{3^{27}7^717^619^5}\right)^4&\left( \dfrac{7^{14}19}{3^{17}} \right)^2&\left( \dfrac{3^{7}17^{3}}{5^{9}} \right)^2\\[9pt]\dfrac{1+i\sqrt{43}}{2}&\mathfrak Z_{-43}&-2^63^87^243&\dfrac{5^{12}7^{58}}{2^{254}3^{108}43^{43}}&\dfrac{3^{104}43^{43}}{7^{130}}&\dfrac{2^{42}}{3^{19}5^3}\\[8pt]\dfrac{1+i\sqrt{67}}{2}&\mathfrak Z_{-67}&-2^63^67^231^267&\dfrac{2^{70}5^{108}31^{10}}{3^{134}7^{118}11^{84}67^{67}}&\dfrac{7^{110}67^{67}}{3^{174}31^{82}}&\dfrac{3^{29}11^{21}}{2^{51}5^{27}}\\[8pt]\dfrac{1+i\sqrt{163}}{2}&\mathfrak Z_{-163}&\hspace{-.6em}\begin{array}{l}
-2^63^67^211^2\times \\
\times 19^2127^2 163 \\
\end{array}&\dfrac{2^{466}5^{12}7^{74}19^{250}}{3^{326}11^{70}23^{84}29^{276}127^{182}163^{163}}&\dfrac{3^{834}127^{110}163^{163}}{7^{274}11^{58}19^{538}}&\dfrac{23^{21}29^{69}}{2^{198}3^{139}5^3}\\[8pt]\hline\hline\end{array}
\end{align*}\end{table}

\begin{remark}Suppose that $z\in\mathfrak H$ is a quadratic irrational whose discriminant $ D$ belongs to the  finite set $ \{-7$, $-8$, $-11$, $-12$, $-16$, $-19$, $-27$, $-28$, $-43$, $-67$, $-163\}$, then both sides of Eq.~\ref{eq:G2_PSL2Z_weight_lift}  are computable from the Gross--Kohnen--Zagier formula (see \cite[][Chap.~V, Corollary 4.3]{GrossZagierI} and \cite[][p.~50, Theorem~II.2]{ZagierKyushuJ1}). We display these exact evaluations in  Table~\ref{tab:class_number_1_examples}.
\eor\end{remark}

\subsection*{Acknowledgements} This work was partly supported by the Applied Mathematics Program within the Department
of Energy (DOE) Office of Advanced Scientific Computing Research (ASCR) as part
of the Collaboratory on Mathematics for Mesoscopic Modeling of Materials (CM4). The manuscript was completed during the author's visit to Prof.\ Weinan E at Princeton  in 2014 and at BICMR\ in 2015. The author thanks Prof.~E for discussions on renormalization group theory, and Prof.\ Shou-Wu Zhang for his comments on Theorem~\ref{thm:app_HC} at Princeton in 2014. The author is especially grateful to Prof.\ Don B. Zagier (MPIM, Bonn) for his encouragements on this series of works. The author appreciates the suggestions from   Dr.~Qingtao Chen (ETH\ Z\"urich) and an anonymous referee  on improving the organization of this paper.

\noindent{\textit{Erratum/Addendum}} In the published version of this work, we claimed, shortly before concluding the proof of Lemma \ref{lm:G_2_Hecke234_int_repn}, that $\mathfrak L_\nu(\alpha) $ (defined in Eq.~\ref{eq:frakL_nu_defn} for $ \alpha\in(\mathbb C\smallsetminus\mathbb R)\cup(0,1)$ and $ \nu\in\{-1/4,-1/3,-1/2\}$) should extend to  a harmonic function for $ \alpha\in\mathbb C\smallsetminus\{0,1\}$. To support our claim, in the paragraph following  Eq.~\ref{eq:frakL_nu_defn}, we pointed out the continuous extension  $\mathfrak L_\nu(\alpha+i0^{+})=\mathfrak L_\nu(\alpha-i0^{+}) $ for $ \alpha\in(-\infty,0)\cup(1,+\infty)$. Such an argument would appear insufficient from an analyst's point of view: one still needs the continuity of normal derivatives, \textit{i.e.} $ \lim_{h\to0^+}\partial \mathfrak L_\nu(\alpha+ih)/\partial h=-\lim_{h\to0^+}\partial\mathfrak L_\nu(\alpha-ih)/\partial h\ $ for $ \alpha\in(-\infty,0)\cup(1,+\infty)$ to guarantee that the \textit{continuous} extension is indeed a \textit{harmonic} extension.

Fortunately, the continuity of normal derivatives can be verified, \textit{a fortiori}, by symmetries of automorphic functions, as we explain below. In fact, in the next two paragraphs, we will show that both sides of Eq.~\ref{eq:G2RN_via_frakL} satisfy the homogenous Neumann boundary condition (vanishing normal derivatives) for non-elliptic points $z$ on $\mathfrak H\cap\partial \mathfrak D_N $ (where $\alpha_N(z)\in(-\infty,0)\cup(1,+\infty) $ for  $ N=4\sin^2(\nu\pi)\in\{2,3,4\}$), and so does $ \mathfrak L_\nu(\alpha_N(z))$.

The reflection symmetry $ G_2^{\overline{\varGamma}_0(N)}(z)=G_2^{\overline{\varGamma}_0(N)}(-\overline{z})$ and the modular invariance $ G_2^{\overline{\varGamma}_0(N)}(z)=G_2^{\overline{\varGamma}_0(N)}(z+1)$ together enforce the vanishing normal derivatives for $ \left.\frac{\partial}{\partial x}\right|_{x=\pm1/2} G_2^{\overline{\varGamma}_0(N)}(x+iy)=0, $ where $y>0$ and $x+iy$ is not an elliptic point. For other points on $\mathfrak H\cap\partial \mathfrak D_N $ (cf.~\cite[][Fig.~1(c)--(e)]{AGF_PartI}), one may exploit the  formula $ G_2^{\overline{\varGamma}_0(N)}(z)-G_2^{\overline{\varGamma}_0(N)}(-1/(Nz))=2\log\left|\frac{\alpha_N'(z)\eta^{4}(-1/(Nz))}{\eta^{4}(z)\alpha'_N(-1/(Nz))}\right|$  (where $ \alpha'_N(w):=\partial \alpha_N(w)/\partial w$ is expressible as $ \eta^4(w)$ times certain powers of $ \alpha_N(w)$ and $ 1-\alpha_N(w)$  \cite[][Eq.~2.1.8]{AGF_PartI}) to establish the homogenous Neumann boundary condition.

Now that the normal derivative for the  left-hand side of  Eq.~\ref{eq:G2RN_via_frakL} vanishes  for every  $z\in\mathfrak H\cap\partial \mathfrak D_N $ (excluding elliptic points), it would suffice to prove the same for the first term on the right-hand side of  Eq.~\ref{eq:G2RN_via_frakL}, before reaching our goal concerning  $ \mathfrak L_\nu(\alpha_N(z))$. As a minor variation on \cite[][Eq.~2.1.1$'$]{AGF_PartI}, we may use an Eichler integral identity (cf.~Eq.~\ref{eq:f-g}) \begin{align}&
-\frac{\sin^{2}(\nu\pi)}{2\pi^2}\frac{z^{2}\left[\frac{\mathfrak Q_{\nu}({\alpha_N(z))}}{P_\nu(1-2\alpha_{N}(z))}+\frac{\mathfrak Q_{\nu}({1-\alpha_N(z))}}{P_\nu(2\alpha_{N}(z)-1)}\right]}{P_\nu(1-2\alpha_{N}(z))P_\nu(2\alpha_{N}(z)-1)}\notag\\\equiv{}&\int_{z}^{i\infty} \frac{\alpha_{N}(z)[1-\alpha_N(z)][\alpha_N'(\zeta)]^{2}}{\alpha_N(\zeta)[1-\alpha_N(\zeta)]\alpha_{N}'(z)}\frac{(z-\zeta )^{2}\D \zeta}{\alpha_N(z)-\alpha_N(\zeta)}-\int_{0}^{i\infty} \frac{\alpha_{N}(z)[1-\alpha_N(z)][\alpha_N'(\zeta)]^{2}}{\alpha_N(\zeta)[1-\alpha_N(\zeta)]\alpha_{N}'(z)}\frac{\zeta^2\D \zeta}{\alpha_N(z)-\alpha_N(\zeta)}\notag\\{}&+\frac{2\pi z}{iN}+\pi i z^2\pmod{2\pi iz^2\mathbb Z}
\end{align}to rewrite the first term on the right-hand side of  Eq.~\ref{eq:G2RN_via_frakL}, and extend its domain of definition to all the non-elliptic points $z$. Such an extension enables us to show that   \begin{align*}
-\frac{\sin^{2}(\nu\pi)}{6\pi^2}\R\left\{(\I z)^{2}
\left(\frac{\partial}{\partial\I z}\frac{1}{\I z}\right)^2\frac{z^{2}\left[\frac{\mathfrak Q_{\nu}({\alpha_N(z))}}{P_\nu(1-2\alpha_{N}(z))}+\frac{\mathfrak Q_{\nu}({1-\alpha_N(z))}}{P_\nu(2\alpha_{N}(z)-1)}\right]}{P_\nu(1-2\alpha_{N}(z))P_\nu(2\alpha_{N}(z)-1)}\right\}
\end{align*}remains invariant as one trades $z$ for $z+1$. Subsequently, we can combine this with reflection symmetry $ z\mapsto-\overline z$  and the Fricke involution $ z\mapsto-1/(N z)$ (as in the last paragraph) to establish  the homogenous Neumann boundary condition for     $ G_2^{\overline{\varGamma}_0(N)}(z)- \mathfrak L_\nu(\alpha_N(z))$, where the normal derivatives are evaluated at non-elliptic points $z$ on  $\mathfrak H\cap\partial \mathfrak D_N $.

The analysis above establishes $C^1$-smoothness of the function  $\mathfrak L_\nu(\alpha) $  for $ \alpha\in\mathbb C\smallsetminus\{0,1\}$ and $ \nu\in\{-1/4,-1/3,-1/2\}$, thereby filling a minor gap in our published arguments for Lemma \ref{lm:G_2_Hecke234_int_repn}. Likewise, in  Lemma \ref{lm:G_2_PSL2Z_int_repn}, the statement about the ``continuous extension'' of  Eq.~\ref{eq:G2Hecke4_G3PSL2Z_corr} should also have read ``$C^1$ (hence harmonic) extension''.
All the conclusions in our original paper thus remain unscathed.


\end{document}